\documentclass[10pt]{article}
\usepackage[toc, page]{appendix}

\usepackage{setspace}

\usepackage{graphicx}

\usepackage{fullpage}
\usepackage{amsmath}
\usepackage{amssymb}
\usepackage{amsopn}
\usepackage{amsthm}

\usepackage{float}
\usepackage[font={small, it}]{caption}

\usepackage{lipsum}
\usepackage[pdftex]{hyperref}

\newcommand{\R}{\mathbb{R}} 

\newcommand{\e}{\varepsilon}
\renewcommand{\le}{\leqslant}
\renewcommand{\ge}{\geqslant}

\def\Xint#1{\mathchoice {\XXint\displaystyle\textstyle{#1}}%
 {\XXint\textstyle\scriptstyle{#1}}%
{\XXint\scriptstyle\scriptscriptstyle{#1}}%
{\XXint\scriptscriptstyle\scriptscriptstyle{#1}}%
\!\int}
\def\XXint#1#2#3{{\setbox0=\hbox{$#1{#2#3}{\int}$} \vcenter{\hbox{$#2#3$}}\kern-.5\wd0}}
 \def\dashint{\Xint-}
\newcommand{\avg}{\dashint} 

\renewcommand{\phi}{\varphi}

\renewcommand{\P}{\mathbb{P}}
\newcommand{\To}{\Rightarrow}
\renewcommand{\i}{^{-1}}

\newcommand{\nm}[1]{\left\|#1\right\|} 
\newcommand{\abs}[1]{\left|#1\right|} 
\newcommand{\f}{\frac}

\renewcommand{\bar}[1]{\overline{#1}}

\newcommand{\lec}{\lesssim}

\newcommand{\Th}{\Theta}
\renewcommand{\th}{\theta}

\newcommand{\bound}[1]{\left.#1\right|_\partial}
\newcommand{\paren}[1]{\left(#1\right)}
\newcommand{\bracket}[1]{\left[#1\right]}
\newcommand{\braces}[1]{\left\{#1\right\}}

\newcommand{\setspace}{\thinspace}
\newcommand{\set}[2]{\left\{ #1 \setspace : \setspace #2\right\}} 

\newtheorem{theorem}{Theorem}
\newtheorem{lemma}[theorem]{Lemma}
\newtheorem{proposition}[theorem]{Proposition}

\newcommand{\HH}{\mathbb{H}}
\newcommand{\af}{\mathfrak{a}}

\renewcommand{\a}{\alpha} 
\renewcommand{\b}{\beta} 
\renewcommand{\aa}{{\alpha '}} 
\newcommand{\aao}{\alpha_0'}
\newcommand{\bb}{{\beta '}} 

\renewcommand{\AA}{\mathcal{A}} 
\newcommand{\AAt}{\mathcal{A}_t}
\newcommand{\ang}{\nu}
\newcommand{\hdenomconst}{2 i}
\newcommand{\hdenomconstreal}{2}
\renewcommand{\vec}[1]{{\bf #1}}

\begin{document}

\title{A Priori Estimates for Two-Dimensional Water Waves with Angled Crests}
\author{Rafe H.~Kinsey\thanks{R.H.~Kinsey is supported in part by a University of Michigan Rackham Regents Fellowship and by NSF grants DMS-0800194 and DMS-1101434.} \and  Sijue Wu\thanks{S.~Wu is supported in part by NSF grant  DMS-1101434.}}
\newcommand{\Addresses}{{\bigskip\footnotesize

R.H.~Kinsey, \textsc{Department of Mathematics, University of Michigan, Ann Arbor, MI 48109}
\par\nopagebreak
  \textit{E-mail address}: \texttt{rkinsey@umich.edu  }

  \medskip

  S.~Wu (Corresponding author), \textsc{ Department of Mathematics, University of Michigan, Ann Arbor, MI 48109   }\par\nopagebreak
  \textit{E-mail address}: \texttt{sijue@umich.edu}
  }}

\date{}

\maketitle

\begin{abstract}
We consider the two-dimensional water wave problem in the case where the free interface of the fluid meets a vertical wall at a possibly non-right angle; and where the free interface can be non-$C^1$ with angled crests.  We assume that the air has density zero, the fluid is inviscid, incompressible,  irrotational, and subject to the gravitational force, and the surface tension is zero.
In this regime,  only a degenerate Taylor stability criterion $-\f{\partial P}{\partial \vec{n}} \ge 0$ holds, with degeneracies at the singularities on the interface and at the point where it meets the wall if the angle is  non-right. We construct a low-regularity energy functional and prove an a priori estimate. Our estimate differs from existing work in that it doesn't require a positive lower bound for $-\f{\partial P}{\partial \vec{n}}$.
\end{abstract}

\tableofcontents

\section{Introduction}\label{sec:intro}
\subsection{Water wave problems}\label{sec:1}

A class of water wave problems concerns the dynamics of the free surface separating an  incompressible fluid, under the influence of gravity,  from a zero-density region (air).

Let $\Omega(t)$ be the fluid region, $\Sigma(t)$ be the free surface between the fluid and the air, and $\Upsilon$, if it exists, be the fixed rigid boundary of $\Omega(t)$, for time $t \ge 0$; thus $\partial(\Omega(t)) = \Sigma(t) \cup \Upsilon$.  We assume that the fluid is inviscid, incompressible and irrotational, and we neglect surface tension. 
Assume that the fluid density is $1$. If the gravity field is $-\vec{j}$, the governing equations of motion are
\begin{align}
  \label{eq:1}  &\vec{v}_t + \vec{v} \cdot \nabla \vec{v} = - \vec{j} - \nabla P \quad \text{ on } \Omega(t),\\
 \label{eq:2} &\text{div } \vec{v} = 0,\quad
  \text{curl } \vec{v} = 0 \qquad \text{ on } \Omega(t), \\
\label{eq:4} &P = 0 \quad \text{ on } \Sigma(t), \\
  \label{eq:5} &(1,\vec{v}) \text{ is tangent}\text{ to the free surface }(t,\Sigma(t)),\\
 \label{eq:6}   &\vec{v} \text{ is tangent to}\text{ the fixed boundary } \Upsilon \text{ (if it exists)},
\end{align}
where $\vec{v}$ is the fluid velocity and $P$ is the fluid pressure.  
There is an important condition for these problems:
\begin{equation}\label{taylor}
-\frac{\partial P}{\partial\bold n}\ge 0
\end{equation}
pointwise on the interface, where $\bold n$ is the outward unit normal to the free interface 
$\Sigma(t)$ \cite{taylor};
it is well known that when surface tension is neglected and the Taylor sign condition \eqref{taylor} fails, the water wave motion can be subject to the Taylor instability \cite{ taylor, ebin, bhl}.

The study of water waves dates back centuries. Early mathematical works include  Stokes \cite{stokes}, Levi-Civita \cite{levi-civita} and G.I.~Taylor \cite{taylor}.  Nalimov \cite{nalimov}, Yosihara \cite{yosihara}, and Craig \cite{craig} obtained local in time existence and uniqueness of solutions for the 2-D water wave problem for small Sobolev data.  In 1997 Wu \cite{wu1997} showed, for the infinite depth two-dimensional water wave problem   \eqref{eq:1}-\eqref{eq:5}  with $\Upsilon = \emptyset$, 
that the strong Taylor stability criterion 
\begin{equation}\label{taylor-s}
-\f{\partial P}{\partial \vec{n}} \ge c_0 > 0,
\end{equation}
 always holds for $C^{1+\epsilon}$  interfaces and that the problem is locally well-posed in Sobolev spaces $H^s, s \ge 4$, for arbitrary data.  In \cite{wu1999} Wu proved a similar result in three dimensions. Since then, there have been numerous results on local well-posedness in both two and three dimensions, under the assumption  \eqref{taylor-s}, for the water wave equations with nonzero vorticity, with a fixed bottom, and with nonzero surface tension, cf. \cite{christodoulou-lindblad, iguchi, ambrose-masmoudi, lannes, lindblad, coutand-shkoller, shatah-zeng, zhang-zhang}. Alazard, Burq \& Zuily \cite{alazard-b-z, alazard-b-z-strichartz} proved local well-posedness of the problem in low regularity Sobolev spaces where the interfaces are only in $C^{3/2}$. Hunter, Ifrim \& Tataru \cite{hunter-ifrim-tataru} obtained a low-regularity result that improves on \cite{alazard-b-z} for 2-D water waves.    In addition, in the past several years Wu \cite{wu2009} \cite{wu2011}, Germain, Masmoudi \& Shatah \cite{germain-masmoudi-shatah}, Ionescu \& Pusateri \cite{pusateri-ionescu} and Alazard \& Delort \cite{alazard-delort} 
have proved results showing almost global or  global well-posedness in two and three dimensions for sufficiently small, smooth and localized initial data. See \cite{cf} for other related developments.

All these work either prove or assume the strong Taylor condition \eqref{taylor-s}, and assume either no fixed boundary or else a fixed boundary of a positive distance away from the free interface. And the lowest regularity 
considered are $C^{3/2}$ interfaces. 

Consider  the water wave equation \eqref{eq:1}-\eqref{eq:6} in two space dimensions. In the case where the fixed boundary $\Upsilon$ is a vertical wall,  by Schwarz reflection, the water wave problem \eqref{eq:1}-\eqref{eq:6} can be reduced to the one without fixed boundary in the expanded symmetric domain. Assume for example that the fluid is in a region in $\{x\ge 0\}$, bounded by 
 the fixed boundary $\Upsilon: x=0$ and the free surface.    
We define,  for $\vec{v} = (v_1, v_2)$ and $x >0$, that
 \begin{equation}
  \label{eq:366}
  \vec{v}(-x, y, t) = (-v_1(x, y, t), v_2(x, y, t)),\qquad P(-x, y, t)=P(x,y, t).
\end{equation}
Notice that \eqref{eq:6} implies that $v_1(\cdot, t) \equiv 0$ on $ \Upsilon$.  It is easy to check that equations \eqref{eq:1}-\eqref{eq:5} continue to hold in the expanded domain. Assume that the interface makes an angle $\nu$ with the vertical wall. When $\nu\ne \frac \pi2$ the extended interface is non-$C^1$, with an angled crest in the middle, see Figure \ref{fig:a}. In \cite{alazard-b-z-2} Alazard, Burq \& Zuily studied the case where the strong Taylor sign condition \eqref{taylor-s} holds and the angle $\nu=\frac\pi 2$. We investigate in this paper the question of whether the water wave problem \eqref{eq:1}-\eqref{eq:6} admits non-right angles $\nu$ at the wall, and more generally, whether equations \eqref{eq:1}-\eqref{eq:6} admit non-$C^1$ interfaces.\footnote{The reflection/periodization procedure described here dates back to \cite{boussinesq}. In \cite{alazard-b-z-2} it was shown  that in order for the strong Taylor sign condition \eqref{taylor-s} to hold, it is necessary that the angle $\nu=\frac\pi 2$.  
}
\begin{figure}[H]
\centering
\includegraphics[width=2in]{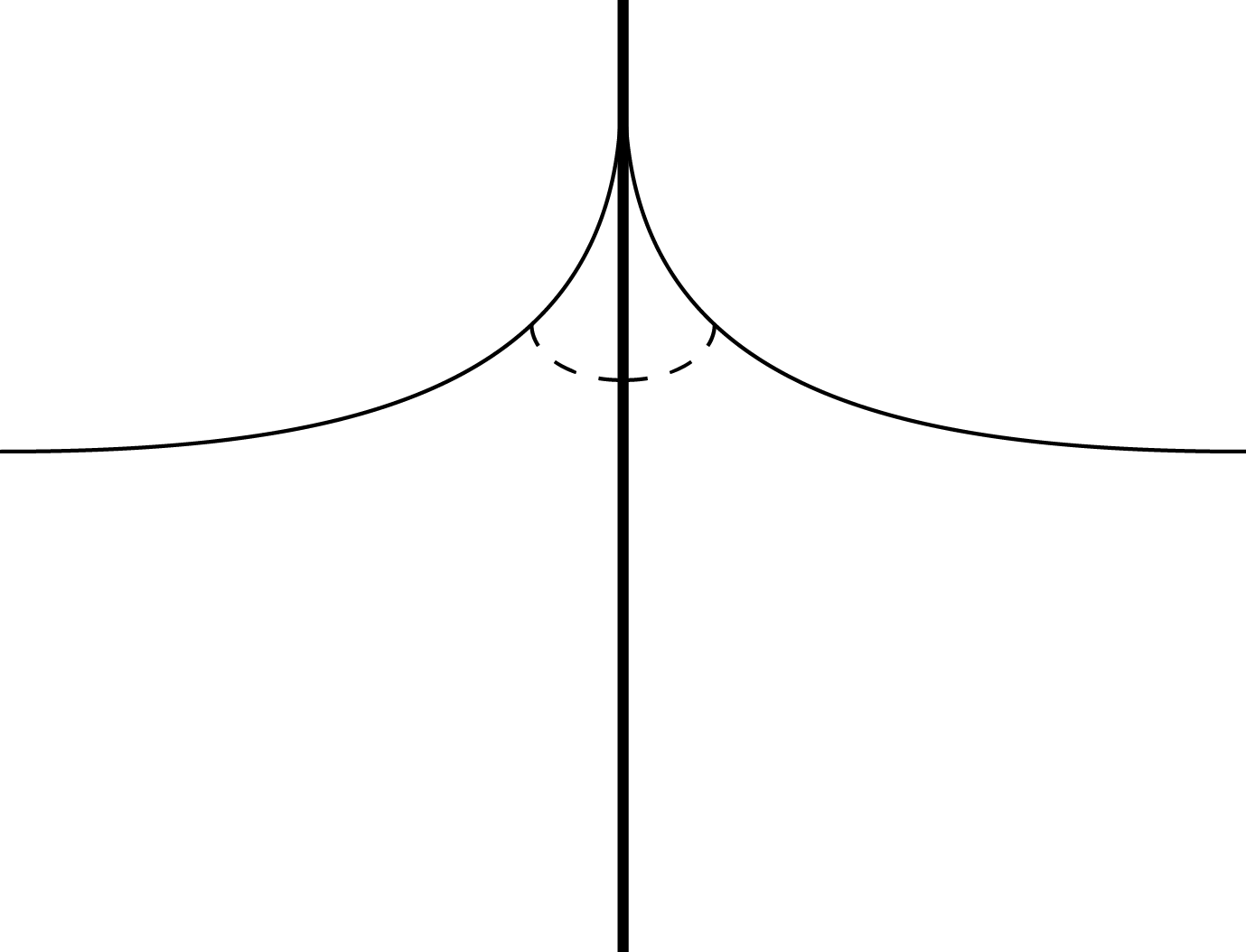}
\caption[Symmetric Angled Crests]{Under a Schwarz reflection, a non-right angle at a vertical wall corresponds to a symmetric angled crest in the middle of surface.}
\label{fig:a}
\end{figure}

To set up our problem we consider a fixed rigid boundary consisting of two vertical walls at $x = 0, 1$, with water of infinite depth in between the walls.  We assume that  the fluid region $\Omega_0(t) \subset [0,1] \times (-\infty, c)$ for some $c < \infty$. Using Schwarz reflection  we expand $\Omega_0(t)$ across the $y$-axis, arriving at a symmetric fluid domain $\Omega(t)\subset [-1,1] \times (-\infty, c)$.  We shall henceforth study the water wave equation \eqref{eq:1}-\eqref{eq:6} in $\Omega(t)$,  with fixed {\it walls} $\Upsilon$ at $x=-1, 1$. 
 We denote the angle at $x = 1$ by $\ang$, and call the corner of the free surface at the wall $x=\pm 1$ just by the {\it corner}.  We assume that
\begin{equation}
  \label{eq:620}
  \vec{v}(x,y,t) \to 0 \qquad \text{ as } y \to -\infty.
\end{equation}
\begin{figure}[H]
\centering
\includegraphics[width=2in]{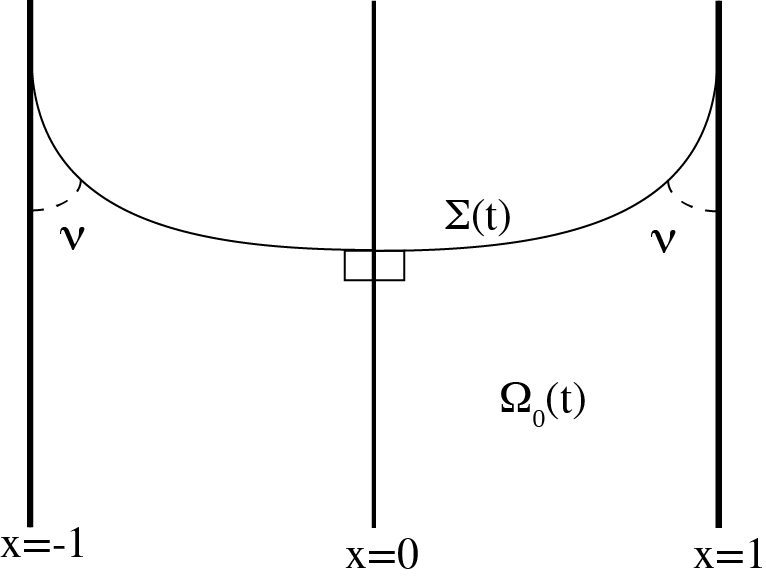}
\caption[The Geometric Framework]{The fluid domain $\Omega_0(t)$ has solid walls at $x = 0, 1$.  The figure shows the reflected domain $\Omega(t)$ under a Schwarz reflection across $x=0$, with a symmetric free surface $\Sigma(t)$. We denote the angle at $x = 1$ by $\ang$.}
\label{fig:b}
\end{figure}

We note that the extended fluid satisfies the periodic boundary condition
\begin{equation}\label{periodic}
\vec{v}(-1,y,t)=\vec{v}(1,y,t),\quad P(-1,y,t)=P(1,y,t).
\end{equation}

 A  
 serious challenge 
 in allowing non-$C^1$ interfaces with angled crests and non-right angles at the wall is that in this case, 
only a degenerate Taylor stability inequality $-\f{\partial P}{\partial \vec{n}} \ge 0$ holds, with degeneracy $-\f{\partial P}{\partial \vec{n}}= 0$ occurs at the singularity on the interface and at the wall when 
the interface meets the wall with a non-right angle.\footnote{We assume the acceleration is finite.} From existing work (cf. \cite{wu1997, wu1999, lannes, lindblad, alazard-b-z} etc.) we know
the problem of solving the water wave equation \eqref{eq:1}-\eqref{eq:6} can be reduced to solving a quasilinear equation of the interface $z=z(\alpha,t)$, of type
\begin{equation}\label{quasi1}
\partial_t^2\frak u+ a \nabla_{\bf n} \frak u=f(\frak u, \partial_t \frak u),
\end{equation}
where $a=-\frac{\partial P}{\partial \bf{n}}$, and $\nabla_{\bf n}$ is the Dirichlet-to-Neumann operator; 
when the strong Taylor sign condition \eqref{taylor-s} holds,  equation \eqref{quasi1} is of the hyperbolic type with the right hand side consisting of lower order terms, and the Cauchy problem can be solved using classical tools.  In our case however,  only the degenerate Taylor sign condition \eqref{taylor} holds, with the second term on the left hand side of \eqref{quasi1} being degenerative at the singularities; it is not clear if the right hand side of equation \eqref{quasi1} is still of the lower order, and what the type of the equation is. This is the main difficulty of the problem.

In this paper, we construct an energy functional and prove an a priori estimate  for solutions of the water wave equation.  The novelty is that our a priori estimate {\em does not} require a positive lower bound for $-\f{\partial P}{\partial \vec{n}}$. We follow the general approach of Wu's earlier work \cite{wu1997, wu1999} and \cite{wu2009}, in reducing the water wave problem to an equation on the free surface,  and analyzing the free surface equation in the Riemann mapping variable. As an immediate consequence we obtain that, provided that the acceleration is finite,  the water wave equation \eqref{eq:1}-\eqref{eq:6} only allows for angles at the wall $\ang\le \frac \pi 2$ and similarly, it only allows for interior angles at angled crests to be no more than $\pi$. Using the Riemann mapping we obtain a precise characterization of the singularities of the interface.  Two elements  played significant roles in the construction of our energy: one is an interface dependent weight function,  which is zero at the singularities; and the other is  
the convection-driven self-similar solutions constructed in \cite{wu:self},  which has an angled crest type singularity. 
 Our energy is finite for all interfaces and velocities in Sobolev spaces $H^s(\mathbb R)$, $s\ge 3$, it is also finite  for interfaces making non-right angles at the wall, and for interfaces with angled crests. Most importantly, it is finite for the self-similar solutions constructed in \cite{wu:self}.

The a priori estimate obtained in this paper holds for general periodic water wave equations \eqref{eq:1}-\eqref{eq:5}-\eqref{periodic}. An analogous energy functional can be constructed and the a priori estimate can be proved for the whole line case using a similar approach.

Our energy inequality is a crucial step towards proving local existence, uniqueness and stability in this framework. This will be the content of an upcoming paper. 

\vskip.1in
{\bf Remark}: This paper was first written and posted on arXiv in June 2014 \cite{rkwu}. Since then, Wu has proved the local existence in the energy class constructed in this paper and posted the work on arXiv in February 2015 \cite{wu2015}. Most recently in December 2016, Thibault de Poyferr\'e \cite{dePoy} obtained an a priori estimate for the water wave equations \eqref{eq:1}-\eqref{eq:6} with an emerging bottom that interacts with the free surface. One of the important assumptions in \cite{dePoy} is that the strong Taylor sign condition \eqref{taylor-s} holds.

{\bf Acknowledgement}: The second author would like to thank Edriss Titi, Roberto Camassa and Bob Pego for the discussion in clarifying a statement in the original version of the paper.

\subsection{Outline of the paper}\label{sec:4}
In the next subsection, \S \ref{sec:5}, we present some of the notations and conventions  
and introduce the function spaces and norms to be used in the paper. Then, in \S \ref{sec:6} and \S \ref{sec:11}, we derive the free surface equations of the water wave problem \eqref{eq:1}-\eqref{eq:6} in Lagrangian and Riemann mapping coordinates,  following \cite{wu1997} and \cite{wu2009}.  
  The derivation in \S\ref{sec:6} and \S\ref{sec:11} is carried out under the assumption that the interface, velocity and acceleration are smooth. In \S \ref{sec:energy}, we define the energy (in \S \ref{sec:35}) and state our main result, the a priori inequality (in \S \ref{sec:38}),  for solutions of the free surface equations.  We begin the proof in \S \ref{sec:39}, and then in \S \ref{sec:42} we outline the remainder of the proof, which takes up sections \S \ref{sec:43} through \S \ref{controlE1}.

In \S \ref{sec:36e} we give a characterization of the energy in terms of the  velocity and position of the free surface, and we discuss the types of singularities allowed when our energy is finite.

 The derivation of the free surface equations and the proof of the main result rely on understanding the boundary behaviors, the holomorphicity, and the means of various quantities; we leave these, as well as some basic identities and inequalities used in the proof of the main result, to appendices \S\ref{sec:hmb}-\S \ref{sec:tech}. The reader may want to read these appendices before certain sections in the main text. We have two additional appendices that might be useful to the reader. In \S \ref{notation}, we provide an overview of the notation used in the paper, with cross references to where everything was initially defined.  In \S \ref{quantities}, we list various quantities controlled by the energy, again with cross references.

\subsection{Notations, conventions and function spaces}\label{sec:5}
We will define most of our notations throughout the text, as we introduce our various quantities.  Here we only list some general conventions and notations.

Since we are in two dimensions, we will often work in complex coordinates $(x,y) = x + iy$.  We will use $\Re z:=x$ and $\Im z:=y$  to represent the real and imaginary parts, respectively, of  $z=x+iy$.

Compositions are always in terms of the spatial variable. For example, for $f=f(\alpha, t)$, $g=g(\alpha,t)$, we define $f\circ g=f\circ g(\alpha,t):=f(g(\alpha,t),t)$. An expression $f_{x}(x,y)$ means $\partial_xf(x,y)$; we occasionally use the notation $f'$, which is always the spatial derivative in whatever coordinates we are using.

Once we have reduced the water wave equations to an equation on the interface, we will primarily be working with the spatial domain $I := [-1,1]$. We will often refer to the ``boundary''; this refers to what happens at $\pm 1$. We write $\bound{f} := f(1) - f(-1)$. We will use 
\begin{equation}
  \label{eq:660}
  \avg_I f := \f{1}{\abs{I}} \int_I f(x) dx = \f{1}{2} \int_{-1}^1 f(x) dx
\end{equation}
for the mean of a function $f$. Here, and elsewhere for other integrals,  
when there is no risk of ambiguity, we will often drop the subscript $I$.

  We define
\begin{equation}
  \label{eq:659}
  [A,B] := AB - BA.
\end{equation}
We will use the following notation as an abbreviation for a type of higher-order Calderon commutator:
\begin{equation}
  \label{eq:624}
  [f,g;h](\aa) := \frac\pi{4i}\int \f{f(\aa) -f(\bb)}{\sin(\f{\pi}{2}(\aa-\bb))} \f{g(\aa) - g(\bb)}{\sin(\f{\pi}{2}(\aa-\bb))} h(\bb) d\bb.
\end{equation}
We will often deal with one dimensional singular integrals of the type $\int k(\aa,\bb)\,d\bb$  where $|k(\aa,\bb)|= \frac{O(1)}{|\aa-\bb|}$, in this case, the  integral $\int k(\aa,\bb)\,d\bb$ is defined to be the principle valued integral. 

We will use $C$ as a placeholder to refer to a universal constant, possibly varying from line to line. We will also often use the notation $f \lec g$, which means that there exists some universal constant $C$ such that $f \le C g$.

We will at several points have long series of identities or inequalities.  When we say ``on the RHS'' of an equation block with a string of multiple equalities or inequalities, we mean all the terms on the right hand side of the last equality or inequality sign in the string. Similarly, when we say ``on the LHS,'' we mean all the terms to the left of the very first equality or inequality sign in  the string of equalities and inequalities.  We have tried to avoid saying ``on the $n$th line'' when any of the mathematical formulas splits into more than one typographic line, but if we have, ``line'' refers to the mathematical, not typographic, line.

We have tried to give extensive cross references for each time we use a result or estimate.  We tend to refer to equation numbers, rather than propositions, since it seems that these will be easier to find as cross references. When we refer to an equation number as part of a proposition, we are of course referring to the whole proposition, including any conditions assumed.

When we are deriving estimates, we sometimes use the cross references within our equations, e.g.:
\begin{equation}
  \label{eq:661}
  \begin{aligned}
    f  & \le g
    \\ & \le h
  \end{aligned}
\end{equation}
and
\begin{equation}
  \label{eq:662}
  \begin{aligned}
    h & \lec j + f
    \\ & \lec j + \eqref{eq:661}
    \\ & \lec j + h.
  \end{aligned}
\end{equation}
This means \eqref{eq:661} is used to obtain \eqref{eq:662}. We hope this will help the reader locate the previous estimate or estimates. 

In several of our more complicated estimates, we will split terms up $f = I + II$ and then $I = I_1 + I_2$, $I_1 = I_{11} + I_{12}$, etc.  Such notation will be {\em local to each section}. There is an ambiguity between the use of $I$ as a placeholder, its use as the identity operator, and its use as $I:= [-1,1]$. It should be clear from the context which one is being used.

We now introduce  the function spaces and norms we will use.   We will work with functions $f(\cdot,t)$ defined on $I=[-1,1]$. Except when necessary to avoid ambiguity, we neglect to write the time variable; when it is not specified,  function spaces and norms are in terms of the spatial variable.

We say $f\in C^k(J)$, $J=(-1,1)$ or $[-1,1]$,  if for every $0 \le l \le k$, $\partial_x^l f$ is a continuous function on the interval $J$. 
We say $f \in C^k(S^1)$ (i.e., periodic $C^k$) if for every $0 \le l \le k$, 
$\partial_x^l f\in C^0[-1,1]$ and 
$\partial_x^lf(1) = \partial_x^lf(-1)$. ($\partial_x^lf$ at the endpoints $1$ or $-1$ is the derivative from the left or right, respectively.) Note in particular that saying $f \in C^0(S^1)$ implies that $\bound{f}=0$.

For $1 \le p < \infty$, we define our $L^p$ spaces by the norms
\begin{equation}
  \label{eq:395}
  \nm{f}_{L^p} := \nm{f}_{L^p(I)}:=\paren{\int_{I} \abs{f}^p}^{1/p},
\end{equation}
and we define $L^\infty$ analogously. We will sometimes deal with weighted $L^p$ spaces. We write
\begin{equation}
  \label{eq:428}
  \nm{f}_{L^p( \omega)} = \nm{f}_{L^p(\omega dx)} := \paren{\int_I \abs{f(x)}^p \omega(x) dx}^{1/p}
\end{equation}
for weights $\omega \ge 0$.
Whenever we write $L^p$, we will be referring to $L^p(I)$, in the spatial variable. For weighted $L^p$ spaces, we always write $L^p(\omega)$ or $L^p(\omega dx)$,   where $\omega$ is the weight function.
 
We now define the   periodic Sobolev space $H^k(S^1)$. Let $f\in L^1(I)$, and $\tilde f$ be the periodic extension of $f$ to the whole line: $\tilde f(x+2)= \tilde f(x)$ for all $x\in \mathbb R$, and $\tilde f(x):=f(x)$ for $x\in I$.  We say $f\in   H^k(S^1)$ if $\tilde f\in H^k\paren{(-3, 3)}$; and we define
\begin{equation}
\nm{f}_{H^k( S^1)}  := \paren{\sum_{j=0}^k \int_I \abs{\partial_x^j f(x)}^2dx}^{1/2}.
\end{equation}
By Sobolev embedding, we know $H^{k+1}(S^1)\subset C^k(S^1)$, for $k\ge 0$.

We define the homogeneous half-derivative space $\dot{H}^{1/2}$ by the norm
\begin{equation}
  \label{eq:23}
  \nm{f}_{\dot{H}^{1/2}} := \paren{\frac\pi 8 \iint_{I \times I} \f{\abs{f(\aa) - f(\bb)}^2}{\sin^2(\f{\pi}{2}(\aa - \bb))} d\aa d\bb}^{1/2}.
\end{equation}

Through the remainder of the paper, when we say the boundary value of a function $G$ defined on the fluid region $\Omega(t)$ (resp., on $P^-:=I\times (-\infty, 0]$), we mean the value of  $G$ on the free surface (resp., on $I\times \{0\}$); we do not include the value on vertical boundaries $x=\pm1$. Except when there's a risk of confusion, we will slightly abuse notation and say that a function $f$ on the free surface (resp., on $I\times \{0\}$) is ``holomorphic'' (or ``antiholomorphic''); what we mean, precisely, is that it is the boundary value of a function that is holomorphic (or antiholomorphic) in the fluid region $\Omega(t)$ (resp., on $P^-$).
 
In the next two sections, \S\ref{sec:6} and \S\ref{sec:11}, we assume the interface, velocity, acceleration and their time derivatives are sufficiently smooth.

\section{The free surface equation in the Lagrangian coordinate}\label{sec:6}

Let $z(\alpha, t)=x(\a,t) + i y(\a,t)$, $\a\in I=[-1,1]$  be a parametrization  of the free surface $\Sigma(t)$ in  the {\em Lagrangian} variable $\a$, i.e., $z_t(\a,t) = \vec{v}(z(\a,t),t)$ is the velocity and $z_{tt}$ is the acceleration of the particle occupying position $z(\a,t)$ at time $t$. 
Along the free surface,  the Euler equation \eqref{eq:1} is $z_{tt} + i = -\nabla P$. By equation \eqref{eq:4},  we know $\nabla P$ is orthogonal to the free surface.  Since $iz_\a$ is normal to the free surface,  we can rewrite our main equation as
\begin{equation}
  \label{eq:8}
  z_{tt} + i = i \mathfrak{a} z_\a,
\end{equation}
where 
\begin{equation}
  \label{eq:28}
  \mathfrak{a} = - \f{\partial P}{\partial \vec{n}} \f{1}{\abs{z_\a}} \in \R
\end{equation}
for $\f{\partial P}{\partial \vec{n}}:=\vec{n}\cdot \nabla P$ the outward-facing normal derivative. The incompressibility and irrotationality condition \eqref{eq:2} and the periodicity \eqref{periodic} imply that the conjugate velocity $\bar{\bf v}$ is periodic holomorphic; therefore $\bar{z}_t$ is the boundary value of a periodic holomorphic function in the fluid region.

\subsection{The quasilinear equation}\label{sec:9}
We henceforth focus on the equations on the free surface.\footnote{We may solve for 
 the velocity on $\Omega(t)$ from its boundary values (including the condition that it goes to zero as $y \to -\infty$), and then solve for the pressure from the velocity. The free surface equations is equivalent to the 
water wave system \eqref{eq:1}-\eqref{eq:6} in the smooth regime. 
} As in \cite{wu1997} and following works, we differentiate  \eqref{eq:8} with respect to time and take conjugates, turning it into the quasilinear equation\footnote{We call it ``quasilinear'' because in the classical situation \cite{wu1997}, this equation is quasilinear with the RHS the lower-order term. However, in our setting, due to the degeneracy of $-\f{\partial P}{\partial \vec{n}}$ we do not know a priori that this is still the case; only by our proof do we show that the RHS is, indeed, lower-order and that \eqref{eq:11} is in fact quasilinear. All references to \eqref{eq:11} and related equations being ``quasilinear'' should be interpreted with this in mind.}
\begin{equation}
  \label{eq:11}
  \bar{z}_{ttt} + i \af \bar{z}_{t\a} = -i \af_t \bar{z}_\a,
\end{equation}
where we continue to have $\bar{z}_t$  the boundary value of a periodic holomorphic function. This is the basic equation we will work with throughout the paper.

The holomorphicity of $\bar{z}_t$ implies that $i\f{1}{\abs{z_\a}} \partial_\a\bar{z}_t=\nabla_{\bf n} \bar{z}_t$,  
where $\nabla_{\bf n}$ is the Dirichlet-to-Neumann operator. We know $\nabla_{\bf n}$ is a positive operator. 

In \cite{wu1999} and \cite{wu2009}, coordinate-independent formulas for the RHS were derived, using the holomorphicity of $\bar{z}_t$ and the invertibility of the double-layer potential.  We will instead follow the original approach of \cite{wu1997}, relying on the Riemann mapping version of the equation to derive the RHS.  We do so in \S \ref{sec:28}.

\subsection{A special derivative}\label{sec:10}

We introduce a special derivative
\begin{equation}
  \label{eq:367}
  D_\a := \f{1}{z_\a}\partial_\a.
\end{equation}
If $g(\a,t) = G(z(\a,t),t)$, and $G$ is holomorphic, then
 $ \partial_\a g = (G_z \circ z) z_\a$, and
 $$D_\a g=( \partial_zG)\circ z=- i (\partial_y  G)\circ z.
 $$
So $D_\a^k g$ is the boundary value of holomorphic function $\partial_z^k G$, provided $G$ is holomorphic. $D_\a^k g$ is in addition periodic for any $k \ge 1$, so long as $G$ is periodic and holomorphic.

We may therefore conclude from the fact that $\bar z_t$ is the boundary value of periodic holomorphic function $\bar {\vec{v}}$ that  $D_\a^k \bar{z}_t$ is the boundary value of periodic holomorphic function $\partial_z^k \bar {\vec{v}}$ in $\Omega(t)$. 

We will use $D_\a$ as the spatial derivative in constructing higher-order energies.  In addition to preserving   holomorphicity and periodic boundary behavior, it transforms well under the Riemann mapping, to be discussed in the next section.

\section{The Riemann mapping version}\label{sec:11}
We now analyze the  water wave equations \eqref{eq:8}, \eqref{eq:11} using the Riemann mapping that flattens out the curved free interface.\footnote{
To the best of our knowledge, \cite{wu1997} was the first paper that used Riemann mapping to analyze the quantities $A_1, \AAt$ and $h_t\circ h^{-1}$ in the water wave equations and prove the wellposedness of 2-d water waves in Sobolev spaces. Using Riemann mapping, \cite{hunter-ifrim-tataru} later carried out a similar analysis as in \cite{wu1997, wu2009} and re-derived the formulas for the quantities $A_1, h_t\circ h^{-1}$. Here we follow the approach of \cite{wu1997, wu2009} to analyze the quantities $A_1, \AAt$ and $h_t\circ h^{-1}$. Riemann mapping is a common tool in the study of 2d potential flows. In water waves, Ovsjannikov \cite{ovs} used  Riemann mapping  to justify the shallow water equation from  equations \eqref{eq:1}-\eqref{eq:5}  in the analytic class; Zakharov et. al. \cite{dksz} used  Riemann mapping  to carry out efficient 
numerical computations for the water waves. 
} 
The Riemann mapping version of the equations offers a key advantage, because the Hilbert transform associated to the periodic domain $P^-$ is  
\begin{equation}
  \label{eq:174a}
  \HH f(\aa) := \f{1}{\hdenomconst}  \int_I \cot(\f{\pi}{2}(\aa - \bb)) f(\bb) d\bb.
\end{equation}
Since $\HH f \in i \R$ for $f$ real-valued, this allows us to invert the operator $(I-\HH)$ on purely real (resp., purely imaginary) functions by taking real (resp., imaginary) parts.

\subsection{The Riemann mapping variables and notations}\label{sec:12a}

Let 
\begin{equation}
  \label{eq:414}
\Phi: {\Omega(t)} \to P^- := \set{(x,y)}{x \in [-1,1], y \le 0}. 
\end{equation}
be the unique Riemann mapping that takes the two upper corners of the interface at the walls $x=-1, 1$ to $(-1,0)$ and $(1,0)$, and $\infty$ to $\infty$.  
We know  $\Phi$ takes the free surface $\Sigma(t)$ to $I\times \{0\}$,   wall to  wall,  $\Phi_z$ is periodic: $\Phi_z(-1, y,t)=\Phi_z(1,y,t)$, and
\begin{equation}\label{riemannmean}
\lim_{\Im z\to-\infty} \Phi_z(z, t)=1.
\end{equation}
 Let 
\begin{equation}
  \label{eq:175}
\aa=  h(\a,t) := \Phi(z(\a,t),t) : I \to I
\end{equation}
be the change of coordinates taking the Lagrangian variable $\a$ to the Riemann mapping variable $\aa$, and let $h\i$ be the spatial inverse of $h$, defined by $h(h\i(\aa,t),t) = \aa$. We define 
\begin{equation}
  \label{eq:107}
  Z(\aa,t) := z \circ h\i (\aa,t) = z(h\i(\aa,t),t).
\end{equation}
$Z=Z(\aa,t)$ is a parametrization of the free surface $\Sigma(t)$ in Riemann mapping variable $\aa$. 
We  write
\begin{equation}\label{eq:370}
\begin{aligned}
  & Z_t := z_t \circ h\i; \quad Z_{tt} := z_{tt} \circ h\i;\\ Z_{,\aa} :=& \partial_\aa Z;  \quad  Z_{t,\aa}: =\partial_\aa Z_t;  \quad  Z_{tt,\aa}: =\partial_\aa Z_{tt}; \text{ etc.}
  \end{aligned}
\end{equation}
and
\begin{equation}
  \label{eq:632}
\AA := (\af h_\a) \circ h\i; \quad   \AAt := (\af_t h_\a) \circ h\i.
\end{equation}

Observe that $Z = z \circ h\i = \Phi\i$. Therefore
\begin{equation}
  \label{eq:63}
 Z_{,\aa}(\aa,t) = \partial_\aa (\Phi\i(\aa,t)), \quad\text{ and }\quad \f{1}{Z_{,\aa}} = \Phi_z \circ Z;
\end{equation}
so $Z_{,\aa}(\aa,t)$ and $\f{1}{Z_{,\aa}}(\aa,t)$ are boundary values of the periodic holomorphic functions $\paren{\Phi^{-1}}_z(\cdot, t)$ and $\Phi_z(\cdot, t)$. 

Observe also that under the change of variables
$(D_\a f) \circ h\i = \f{1}{Z_{,\aa}} \partial_\aa (f \circ h\i)$.
We  define
\begin{equation}
  \label{eq:124}
  D_\aa := \f{1}{Z_{,\aa}} \partial_\aa.
\end{equation}

\subsection{An assumption at the spatial infinity}\label{assume}
For the derivation of the water wave equations in Riemann mapping variable, besides assuming all the quantities involved are sufficiently smooth, we assume that the Riemann mapping $\Phi$ satisfies
\begin{equation}\label{riemanntime}
\lim_{\Im z \to -\infty} \Phi_t\circ \Phi^{-1}(z,t)=0.
\end{equation}

In \cite{wu1997} analogous assumptions were made  to derive the quasilinear equation in Riemann mapping variable for the whole line case; it was then shown that the quasilinear equation is well-posed in Sobolev spaces and the solutions of the quasilinear equation give rise to solutions of the water wave equation \eqref{eq:1}-\eqref{eq:5}. Similar results can be proved for the periodic case as considered in this paper.


\subsection{The water wave equations in the Riemann mapping variable}\label{sec:26}
We now derive  the water wave equations in Riemann mapping variable. We follow the approach of \cite{wu1997}, although we work with the real and imaginary parts together instead of separating $Z_t = X_t + iY_t$ into real and imaginary parts.

Beginning with the conjugated form of our equation \eqref{eq:8} and with \eqref{eq:11},
 we precompose both sides with $h\i$ to get the free surface equations in the flattened Riemann mapping coordinate:
\begin{equation}
  \label{eq:179}
  \bar{Z}_{tt} - i = - i \AA \bar{Z}_{,\aa};
\end{equation}
\begin{equation}
  \label{eq:178}
   \bar Z_{ttt} + i \AA \bar Z_{t,\aa} = -i \AAt  \bar Z_{,\aa},
\end{equation}
where $\bar Z_t$ is the boundary value of the periodic holomorphic function $\bar{\vec{v}}\circ \Phi^{-1}$. 
 By chain rule, the quantities $\bar Z_{tt}$, $\bar Z_{ttt}$ are related to $\bar Z_t$ by 
$$\bar Z_{tt}=(\partial_t+\frak b\partial_\aa)\bar Z_t,\qquad \bar Z_{ttt}=(\partial_t+\frak b\partial_\aa)^2\bar Z_t,$$
where
\begin{equation}\label{b}
\frak b:=h_t\circ h^{-1}.
\end{equation}
 
The following proposition gives a characterization of the boundary value of  a periodic holomorphic function on $P^-$.
\begin{proposition}\label{prop:hilbe}
a.  Let $g \in L^1(I)$. 
  Then $g$ is the boundary value of a holomorphic function $G$ on $P^-$ satisfying $G(-1+iy) = G(1+iy)$ for all $y < 0$ and $G(x+iy) \to c_0$ as $y \to -\infty$ if and only if 
  \begin{equation}
    \label{eq:1571}
    (I-\HH) g = c_0.
  \end{equation}
Moreover, $c_0 = \avg_I g$.

b. Let $ f \in L^1(I)$. Then $\mathbb P_H  f:=\frac12(I+\mathbb H)  f$ is the boundary value of a periodic holomorphic function $\mathcal F$ on $P^-$, with $\mathcal F(x+iy)\to \frac12 \avg_I  f$ as $y\to -\infty$.

\end{proposition}
Proposition~\ref{prop:hilbe} is a classical result, which can be proved by the Cauchy integral formula; see \cite{Journe}. As a consequence, $ \bar Z_t$ satisfies 
\begin{equation}
\label{eq:1026}
(I-\mathbb H)\bar Z_t=0,
\end{equation}
and by \eqref{riemannmean}, $\frac1{Z_{,\aa}}$ satisfies
\begin{equation}
\label{eq:1026a}
(I-\mathbb H) \frac1{Z_{,\aa}}=1.
\end{equation}

We define the following projection operators:
\begin{equation}
  \label{eq:1581}
   \P_H f := \f{(I + \HH)}{2} f; \quad \P_A f := \f{(I-\HH)}{2} f.
\end{equation}
We will refer to  $\P_H$ as the ``holomorphic projection'' and $\P_A$ as the ``antiholomorphic projection''.  These operators are, indeed, proper projections when interpreted modulo a constant.
 
 We now seek formulas for $\AA$, $\AAt$ and $\frak b$.  For derivations here 
 we will rely on the technicalities given in appendix \S \ref{sec:hmb}. 

\subsubsection{\texorpdfstring{$\AA$}{A}  and the quantity  \texorpdfstring{$A_1$}{A1} }\label{sec:27a}
We first derive a formula for $\AA$.   Historically, in \cite{wu1997}, this derivation showed that the strong Taylor stability criterion would automatically hold. There is now \cite{wu1999} a direct proof of this via basic elliptic theory without using the Riemann mapping. For our purposes, though, this original derivation here will be crucial because it introduces a quantity, $A_1$, that compares the degeneracy of $-\f{\partial P}{\partial \vec{n}}$ directly with that of the Riemann mapping and therefore the geometry of the free surface.

We begin with \eqref{eq:179}. The key observation is that $Z_{,\aa}$ is periodic holomorphic, and 
$\bar{Z}_{tt} - i$ is in some sense close to periodic holomorphic, since $\bar{Z}_{t}$ is periodic holomorphic, therefore $Z_{,\aa}(\bar{Z}_{tt} - i)$ is close to periodic holomorphic; and by \eqref{eq:179}, $Z_{,\aa}(\bar{Z}_{tt} - i)$ is purely imaginary.  Applying $(I-\mathbb H)$ to this, using Proposition~\ref{prop:hilbe} and taking imaginary parts, we shall get a formula for $\AA$. 

Multiplying both sides of \eqref{eq:179} by  $Z_{,\aa}$, we get
\begin{equation}
  \label{eq:180}
  Z_{,\aa}(\bar{Z}_{tt} - i) =- i \AA \abs{Z_{,\aa}}^2.
\end{equation}
We now expand out $\bar{Z}_{tt}$. Let
\begin{equation}
  \label{eq:376}
  F(z(\a,t),t) := \bar{z}_t(\a,t),
\end{equation}
where $F= \bar{\vec{v}}$, a periodic holomorphic function in $\Omega(t)$.  We will use this expansion several times in the sequel, always with this definition of $F$. By the chain rule,
\begin{equation}
  \label{eq:181}
  \bar{z}_{tt} = \f{d}{dt} F(z(\a,t),t) = (F_z \circ z) z_t + (F_t \circ z).
\end{equation}
Recall from \S \ref{sec:10} that $\partial_z = D_\a$ for  holomorphic functions. Therefore, $  F_z \circ z = \f{\bar{z}_{t\a}}{z_\a}$, 
and thus $
  \bar{z}_{tt} = \f{\bar{z}_{t\a}}{z_\a} z_t + F_t\circ z$. 
We precompose with $h\i$:
\begin{equation}
  \label{eq:184}
  \bar{Z}_{tt} = \paren{\f{\bar{Z}_{t,\aa}}{Z_{,\aa}}} Z_t + F_t \circ Z.
\end{equation}
We can now write our equation \eqref{eq:180} as
\begin{equation}
  \label{eq:185}
  \bar{Z}_{t,\aa} Z_t + Z_{,\aa} (F_t \circ Z) - i Z_{,\aa} = -i \AA \abs{Z_{,\aa}}^2.
\end{equation}
We apply $(I-\HH)$ to both sides. By \eqref{eq:474a}, \eqref{eq:186a} and \eqref{eq:516}, writing $(I-\HH)( Z_t \bar{Z}_{t,\aa}) = [Z_t,\HH] \bar{Z}_{t,\aa}$, we get 
\begin{equation}
  \label{eq:174}
  [Z_t,\HH] \bar{Z}_{t,\aa} -i = (I-\HH) \paren{- i \AA \abs{Z_{,\aa}}^2}.
\end{equation}
We now take imaginary parts of both sides. This gives us the new quantity $A_1$:
\begin{equation}
  \label{eq:208}
  A_1 : = \AA \abs{Z_{,\aa}}^2 
  = \Im\paren{- [Z_t,\HH] \bar{Z}_{t,\aa}} + 1.
\end{equation}
This is the same $A_1$ as that in \cite{wu1997}. It's easy to see that $\Im \paren{-[Z_t,\HH] \bar{Z}_{t,\aa}}$ is non-negative, by integration by parts. Indeed, if $Z_t = X_t + iY_t$, then
\begin{equation}
  \label{eq:132}
  \begin{aligned}
   \Im \paren{-[Z_t,\HH] \bar{Z}_{t,\aa}} & = - \Im \f{1}{\hdenomconst} \int (Z_t(\aa) - Z_t(\bb)) \cot(\f{\pi}{2}(\aa -\bb)) \bar{Z}_{t,\bb} d\bb
 \\ & =  \f{1}{\hdenomconstreal} \int \f{1}{2} \braces{- \partial_\bb \bracket{(X_t(\aa) - X_t(\bb))^2 + (Y_t(\aa) - Y_t(\bb))^2}} \cot(\f{\pi}{2}(\aa - \bb)) d\bb
\\ & = \f{\pi}{ 8} \int  \f{(X_t(\aa) - X_t(\bb))^2 + (Y_t(\aa) - Y_t(\bb))^2}{\sin^2(\f{\pi}{2}(\aa -\bb))} d\bb
\\&  \ge 0.
\end{aligned}
\end{equation}
Here, there is no boundary term in the integration by parts because $\bound{\bar Z_t }=0$. 
Therefore
\begin{equation}
  \label{eq:394}
  A_1 \ge 1.
\end{equation}
Combine \eqref{eq:180} and \eqref{eq:208} we get
\begin{equation}
  \label{eq:251}
  \f{1}{Z_{,\aa}} = i \f{\bar{Z}_{tt} - i}{A_1}.
\end{equation}

\subsubsection{Degenerate Taylor stability criterion and the singularity of the free surface}\label{sec:27}

We can draw a few important conclusions from the derivations in \S \ref{sec:27a}. For the sake of exposition, we will in this section focus on the angle $\ang$ at the wall, and we will move the corner from $\pm 1$ to $0$; angled crests and other singularities in the middle of the surface can be handled similarly.  

Let $\ang$ be the angle at the corner.   By Christoffel-Schwarz Theorem, the Riemann mapping $\Phi(z)\approx z^r$ at the corner, where $r \ang = \f{\pi}{2}$. By \eqref{eq:63}, ${Z_{,\aa}} = (\Phi^{-1})_{z'}$. Therefore, 
\begin{equation}\label{eq;1010}
Z_{,\aa}=\partial_\aa\Phi^{-1}\approx (\aa)^{1/r-1}
\end{equation}
 at the corner.
 
We observe from \eqref{eq:394}-\eqref{eq:251} that  
 if the acceleration $|Z_{tt}|<\infty$, then $Z_{,\aa}\ne0$.  This implies $r\ge 1$ and $\ang\le \pi/2$. Similarly, this implies that an angled crest have interior angle $\le \pi$.

In this paper we work in the regime where the acceleration $|Z_{tt}|<\infty$.
 
Now, because
\begin{equation}
  \label{eq:209}
  A_1 \circ h = \f{\af \abs{z_\a}^2}{h_\a}
\end{equation}
and
\begin{equation}
  \label{eq:50}
 - \f{\partial P}{\partial \vec{n}} = \abs{z_\a}\af = \f{A_1 \circ h}{\abs{Z_{,\aa} \circ h}},
\end{equation}
$-\f{ \partial P}{\partial \vec{n}}\ge 0$ always holds.
In the regime where the free surface is $C^{1,\gamma}$ and makes a right angle at the corner ($\ang=\pi/2$), 
$0<c_0\le\abs{ (\Phi^{-1})_{z'}}\le C_0<\infty$.   This together with the estimate $A_1 \ge 1$  gave \cite{wu1997} a strictly positive lower bound for the Taylor coefficient $-\f{\partial P}{\partial \vec{n}}$.  In our situation,   $\f{1}{Z_{,\aa}} \to 0$ at the corner if $\ang<\frac \pi 2$; similarly, $\f{1}{Z_{,\aa}} \to  0$ at an angled crest if the interior angle is $<\pi$. If $A_1$ is in addition bounded from above --- which will be true when our energy is finite --- we know that the degeneracy of $-\f{\partial P}{\partial \vec{n}}$ corresponds precisely to the degeneracy of $\f{1}{Z_{,\aa}}$. 

We note that our spatial derivative $D_\aa = \f{1}{Z_{,\aa}} \partial_\aa$ is less singular, in a sense, than $\partial_\aa$. We know $\f{1}{Z_{,\aa}} \to 0$ at singularities, which indicates that the weight function $\f{1}{Z_{,\aa}}$ has some ``regularizing'' effect. Indeed, we have $D_\aa \bar Z_t, \ D_\aa \bar Z_{tt} \in L^\infty$, but $\partial_\aa \bar Z_t,\ \partial_\aa \bar Z_{tt}$ are only in $L^2$ in our energy space, for example; see \S \ref{sec:quants}.

\subsubsection{The quantities \texorpdfstring{$\AAt$}{At} and \texorpdfstring{$\f{\af_t}{\af}\circ h^{-1}$}{at/a}
}\label{sec:28}
Now we seek a formula for the quantity on the RHS of \eqref{eq:178}, $\AAt$. As in \S \ref{sec:27a}, we 
start by multiplying both sides of \eqref{eq:178} by $Z_{,\aa}$ to get a purely imaginary and almost holomorphic quantity, then apply
 $(I-\HH)$ to both sides and take imaginary parts to get a formula for $\AAt$. We have
\begin{equation}
  \label{eq:194}
  Z_{,\aa}(\bar{Z}_{ttt} + i \AA \bar{Z}_{t,\aa}) = - i \AAt \abs{Z_{,\aa}}^2.
\end{equation}
We once again carefully expand the LHS. As before, let $F(z(\a,t),t) = \bar{z}_t(\a,t)$. Again, we have
\begin{equation}
  \label{eq:380}
  \bar{z}_{tt} = (F_z\circ z) z_t + F_t\circ z,
\end{equation}
so
\begin{equation}
  \label{eq:195}
  \bar{z}_{ttt}  = (F_{zz} \circ z) z_t^2 + 2 (F_{tz} \circ z) z_t + (F_z\circ z) z_{tt} + F_{tt}\circ z.
\end{equation}
We now solve for $F_z \circ z,$ $F_{zz} \circ z$ and $F_{tz} \circ z$. Since $\partial_z = D_\a$ on  holomorphic functions,
\begin{equation}
  \label{eq:199}
  F_z \circ z= D_\a \bar{z}_t,\quad F_{zz} \circ z= D_\a^2 \bar{z}_t.
\end{equation}
We solve for $F_{tz} \circ z$ by applying $\partial_z = D_\a$ to \eqref{eq:380}:
\begin{equation}
  \label{eq:198}
  F_{tz} \circ z = D_\a \paren{\bar{z}_{tt} - (D_\a \bar{z}_t) z_t}.
\end{equation}
Therefore, by substituting \eqref{eq:199} and \eqref{eq:198} into \eqref{eq:195}, we get
\begin{equation}
  \label{eq:200}
  \bar{z}_{ttt} = (D_\a^2 \bar{z}_t) z_t^2 + 2 z_t D_\a \paren{\bar{z}_{tt} - (D_\a \bar{z}_t) z_t} + (D_\a \bar{z}_t) z_{tt} + F_{tt} \circ z.
\end{equation}
Precomposing with $h\i$, we have
\begin{equation}
  \label{eq:202}
  \bar{Z}_{ttt} = (D_\aa^2 \bar{Z}_t) Z_t^2 + 2 Z_t D_\aa(\bar{Z}_{tt} - (D_\aa \bar{Z}_t) Z_t) + (D_\aa \bar{Z}_t) Z_{tt} + F_{tt} \circ Z.
\end{equation}
We now go back to \eqref{eq:194}, substituting in \eqref{eq:202} to get
\begin{equation}
  \label{eq:203}
  Z_{,\aa} \paren{(D_\aa^2 \bar{Z}_t) Z_t^2 + 2 Z_t D_\aa(\bar{Z}_{tt} - (D_\aa \bar{Z}_t) Z_t) + (D_\aa \bar{Z}_t) Z_{tt} + F_{tt} \circ Z + i \AA \bar{Z}_{t,\aa}} = -i \AAt \abs{Z_{,\aa}}^2.
\end{equation}
We simplify, distributing the $Z_{,\aa}$ and then using the identity $Z_{tt} + i = i \AA Z_{,\aa}$ on the last term:
\begin{equation}
  \label{eq:204}
  (\partial_\aa D_\aa \bar{Z}_t) Z_t^2 + 2 Z_t \partial_\aa(\bar{Z}_{tt} - (D_\aa \bar{Z}_t)Z_t) + 2\bar{Z}_{t,\aa} Z_{tt} + Z_{,\aa} (F_{tt} \circ Z) + i \bar{Z}_{t,\aa} = - i \AAt \abs{Z_{,\aa}}^2.
\end{equation}
We now apply $(I-\HH)$ to both sides.  Various terms will disappear on the LHS and others will turn into commutators, due to holomorphicity; specifically, we use \eqref{eq:474a}, \eqref{eq:222a}, and \eqref{eq:504a}. We get
\begin{equation}
  \label{eq:205}
  \begin{split}
      [Z_t^2,\HH]\partial_\aa D_\aa \bar{Z}_t + 2[Z_t,\HH] \partial_\aa (\bar{Z}_{tt} - (D_\aa \bar{Z}_t) Z_t) + 2[Z_{tt},\HH] \bar{Z}_{t,\aa} 
\\ = (I-\HH) \braces{-i \AAt \abs{Z_{,\aa}}^2}.
\end{split}
\end{equation}
We could continue working with this equation, but two integrations by parts will give us a nicer equation to work with.  We take the first term and the second part of the second term and integrate by parts both terms, noting that we have no boundary terms. We get
\begin{equation}
  \label{eq:265}
  \bracket{Z_t^2,\HH} \partial_\aa D_\aa \bar{Z}_t - 2[Z_t,\HH]\partial_\aa\paren{(D_\aa \bar{Z}_t) Z_t} = - \f{\pi}{4i} \int \f{(Z_t(\aa) - Z_t(\bb))^2}{\sin^2(\f{\pi}{2}(\aa - \bb))} D_\bb \bar{Z}_t(\bb) d\bb.
\end{equation}
This is a type of higher-order Calderon commutator, which we write as $-[Z_t,Z_t; D_\aa \bar{Z}_t]$ (see \eqref{eq:624}). We therefore can rewrite \eqref{eq:205} as
\begin{equation}
  \label{eq:266}
  - i (I-\HH)  \braces{\AAt \abs{Z_{,\aa}}^2} =   2[Z_t,\HH] \bar{Z}_{tt,\aa} + 2[Z_{tt},\HH] \bar{Z}_{t,\a} - [Z_t, Z_t; D_\aa \bar{Z}_t].
\end{equation}
Taking imaginary parts, we get
\begin{equation}
  \label{eq:207}
   \AAt \abs{Z_{,\aa}}^2 = - \Im \paren{2[Z_t,\HH] \bar{Z}_{tt,\aa} + 2[Z_{tt},\HH] \bar{Z}_{t,\a} - [Z_t, Z_t; D_\aa \bar{Z}_t]}.
\end{equation}
Observe that dividing \eqref{eq:207} by \eqref{eq:208} we have
\begin{equation}
  \label{eq:210}
  \f{\af_t}{\af} \circ h\i = 
  \f{\AAt}{\AA} = \f{\AAt \abs{Z_{,\aa}}^2}{\AA \abs{Z_{,\aa}}^2} = \f{- \Im \paren{2[Z_t,\HH] \bar{Z}_{tt,\aa} + 2[Z_{tt},\HH] \bar{Z}_{t,\aa} - [Z_t, Z_t; D_\aa \bar{Z}_t]}}{A_1}.
\end{equation}

\subsubsection{The quantity  $\frak b:=h_t\circ h^{-1}$}\label{hta-ha-formula}
Here we derive a formula for $\frak b:=h_t\circ h\i$, following \cite{wu1997}. We recall \eqref{eq:175}:  $ h(\a,t) = \Phi(z(\a,t),t) = \Phi \circ z$. Therefore
$ h_\a = (\Phi_z \circ z) z_\a$, 
and
\begin{equation}
  \label{eq:244}
  h_t = (\Phi_t \circ z)+ (\Phi_z \circ z) z_t = (\Phi_t \circ z) + \f{h_\a}{z_\a} z_t.
\end{equation}
We precompose with $h^{-1}$: 
\begin{equation}
  \label{eq:246}
  (h_t \circ h\i) (\aa,t) = \Phi_t \circ Z + \f{1}{Z_{,\aa}} Z_t.
\end{equation}
Apply $(I-\mathbb H)$ to both sides, then take the real parts. By \eqref{eq:312a}, we get 
\begin{equation}
  \label{eq:246a}
 \frak b:= (h_t \circ h\i) (\aa,t) =  \Re (I-\mathbb H) (\f{1}{Z_{,\aa}} Z_t).
\end{equation}

In what follows we will also use the following evolution equation for $\frac1{z_\a}$, or equivalently in the Riemann mapping variable, $\frac1{Z_{,\aa}}$.  We have

\begin{equation}\label{eq:2024}
\partial_t\frac1{z_\a}=-\frac1{z_\a}D_\a z_t, \qquad\text{and}
\end{equation}
\begin{equation}\label{eq:2024a}
\paren{\partial_t\frac {h_\a}{z_\a}}\circ {h}^{-1}=(\partial_t+\frak b\partial_\aa)\paren{\frac1{Z_{,\aa}}}=\frac1{Z_{,\aa}}(\partial_\aa\frak b-D_\aa Z_t).
\end{equation}

\section{The main result}\label{sec:energy}

Observe that the system of free surface equations \eqref{eq:179}-\eqref{eq:178}-\eqref{eq:208}-\eqref{eq:210}-\eqref{eq:246a}-\eqref{eq:1026}-\eqref{eq:1026a}-\eqref{eq:2024a} is a closed system for the quantities $\bar Z_t$, $\bar Z_{tt}$ and $\frac1{Z_{,\aa}}$. In what follows we will focus on this system and construct an energy that allows for $\frac1{Z_{,\aa}}=0$, i.e. singularities on the interface and at the corner, and prove an a priori estimate.

\subsection{Definition of the energy}\label{sec:35}

  We consider a general equation of the form
\begin{equation}
  \label{eq:171}
  (\partial_t^2 + i \mathfrak{a} \partial_\a) \th = G_\th
\end{equation}
with the constraint that $\theta$ is the boundary value of a periodic holomorphic function on $\Omega(t)$. 
Precompose with $h^{-1}$ and use \eqref{eq:209}, we obtain the equation in the Riemann mapping variable
\begin{equation}
  \label{eq:171h}
  \paren{\partial_t^2\th} \circ h^{-1} + i \frac{A_1}{|Z_{,\aa}|^2} \partial_\aa\paren{ \th\circ h^{-1}} = G_\th\circ h^{-1}.
\end{equation}
There are two mutually related basic energies. One is 
\begin{equation}
  \label{eq:328h}
  E_{a,\th}(t) := \int_I \f{\abs{\th_t\circ h^{-1}}^2} {A_1 } d\aa +  \int_I i \paren{\partial_\aa \paren{\f{1}{Z_{,\aa}} (\th \circ h\i)}}\paren{\f{1}{\bar{Z}_{,\aa}} (\bar{\th} \circ h\i)} d\aa +  \int_I \f{\abs{\th\circ h^{-1}}^2}{A_1} d\aa,
\end{equation}
 for \eqref{eq:171h} in the Riemann mapping variables, another is 
\begin{equation}
  \label{eq:172}
  E_{b,\th}(t) := \int_I \f{1}{\mathfrak{a}} \abs{\th_t}^2 d\a + \int_I (i\partial_\a \th) \bar{\th} d\a + \int_I \f{(A_1 \circ h)}{\mathfrak{a}} \abs{\th}^2 d\a
\end{equation}
for \eqref{eq:171}. A form of $E_{b,\th}(t)$ has appeared in \cite{wu2009}. Upon changing to the Riemann mapping coordinates, and by \eqref{eq:209}, we have
\begin{equation}
  \label{eq:172h}
  E_{b,\th}(t) = \int_I \f{\abs{Z_{,\aa}\th_t\circ h^{-1}}^2}{A_1}  d\aa + \int_I i(\partial_\aa \th\circ h^{-1}) (\bar{\th\circ h^{-1}}) d\aa + \int_I \abs{Z_{,\aa}\th\circ h^{-1}}^2 d\aa.
\end{equation}
In $E_{a,\th}$ and $E_{b,\th}$, the first two terms are from the LHS of the equations,  the last terms are of lower order. Because $\th\circ h^{-1}$ and $\f{1}{Z_{,\aa}}$ are holomorphic,  the second terms in $E_{a,\th}(t) $ and $E_{b,\th}(t) $ are nonnegative and equal to $\|\f{1}{Z_{,\aa}} (\th \circ h\i)\|_{\dot H^{1/2}}^2$ and 
 $\| \th \circ h\i\|_{\dot H^{1/2}}^2$ respetively, see \S \ref{sec:half}.

Notice that the primary difference between the two basic energies $E_{a,\th}$ and $E_{b,\th}$ is either to multiply or to divide by the weight function $\f{1}{Z_{,\aa}} $. In the classical case where $|Z_{,\aa}|$ is bounded away from $0$ and $\infty$, there is no real difference between the two energies. However it does make a difference if we want to allow $\f{1}{Z_{,\aa}} \to 0$. 

We now construct our energy functional, by applying the two basic energies to  our equations. We begin with equation 
 \eqref{eq:11} (equivalently \eqref{eq:178}, by a change of the coordinates.):
\begin{equation}
  \label{eq:362}
  (\partial_t^2 + i \mathfrak{a} \partial_\a) \bar{z}_t = -i \af_t \bar{z}_{\a}.
\end{equation}
Apply weighted derivatives $D_\a^k$ to \eqref{eq:362} we get
$$(\partial_t^2 + i \mathfrak{a} \partial_\a) D_\a^k\bar{z}_t =G_{D_\a^k \bar{z}_t}$$
 with 
\begin{equation}
  \label{eq:30}
 G_{D_\a^k \bar{z}_t} =D_\a^k(-i\mathfrak{a}_t \bar{z}_\a) + [\partial_t^2 + i \mathfrak{a}\partial_\a,D_\a^k] \bar{z}_t.  
\end{equation}
Our total energy consists primarily of $E_{a, D_\a^2\bar{z}_t}$ and $E_{b,D_\a\bar{z}_t}$. In addition, we will include one other term in our total energy:  $\abs{\bar{z}_{tt}(\a_0,t) - i}$ for some fixed $\a_0 \in I$. Our total energy therefore is 
\begin{equation}
  \label{eq:49}
  E=E(t) := E_{a,D_\a^2\bar{z}_t}(t) + E_{b,D_\a\bar{z}_t} (t) + \abs{\bar{z}_{tt}(\a_0,t)-i}.
\end{equation}

We henceforth abbreviate notation and write  $E_a := E_{a,D_\a^2 \bar{z}_t}$ and 
$E_b := E_{b,D_\a\bar{z}_t}$. We know
\begin{equation}
  \label{eq:400}
  E_a:= E_{a,D_\a^2 \bar{z}_t} = \nm{(\partial_t D_\a^2 \bar{z}_t) \circ h\i}_{L^2(1/A_1)}^2 +  \nm{\f{1}{Z_{,\aa}}D_\aa^2 \bar{Z}_t}_{\dot{H}^{1/2}}^2 + \nm{D_\aa^2 \bar{Z}_t}_{L^2(1/A_1)}^2.
\end{equation}
and 
\begin{equation}
  \label{eq:408}
  E_b:= E_{b,D_\a \bar{z}_t} = \nm{\partial_t D_\a \bar{z}_t}_{L^2(\f{1}{\af})}^2 +  \nm{D_\aa \bar{Z}_t}_{\dot{H}^{1/2}}^2 + \nm{\bar{Z}_{t,\aa}}_{L^2}^2.
 \end{equation}
As we will see in \S \ref{sec:quants}, the first term  in $E_b$, $  \nm{\partial_t D_\a \bar{z}_t}_{L^2(\f{1}{\af})} \approx \nm{\bar{Z}_{tt,\aa}}_{L^2}$; this together with the last term $\abs{\bar{z}_{tt}(\a_0,t)-i}$ in the energy $E$  gives us control of $\nm{\bar{Z}_{tt}(t)-i}_{L^\infty}$. The inclusion of the last term in $E_b$ gives us an upper bound for $A_1$, see \eqref{eq:361}. 

After developing all necessary tools, in \S \ref{sec:36e} we will show that our energy $E$ is equivalent to the following 
\begin{equation}\label{energy1}
\begin{aligned}
\mathcal E(t)=\|\bar Z_{t,\alpha'}\|_{L^2}^2&+ \|D_{\alpha'}^2\bar Z_t\|_{L^2}^2+\|\partial_{\alpha'}\frac1{Z_{,\alpha'}}\|_{L^2}^2+\|D_{\alpha'}^2\frac1{Z_{,\alpha'}}\|_{L^2}^2\\&+\|\frac1{Z_{,\alpha'} }D_{\alpha'}^2\bar Z_t  \|_{\dot H^{1/2}}^2+\| D_{\alpha'}\bar Z_t  \|_{\dot H^{1/2}}^2+\|\frac1{Z_{,\alpha'}}\|_{L^\infty}^2;
\end{aligned}
\end{equation}
that is, there are universal polynomials $P_1$ and $P_2$, such that 
$E(t)\le P_1(\mathcal E(t))$, and $ \mathcal E(t)\le P_2(E(t))$. 
One may now get a  glimpse of this fact by \eqref{eq:251}. Notice that there is no control of $\nm{Z_{,\aa}}_{L^\infty}$ by  our energy $E$ or $\mathcal E$. So in the regime where $E<\infty$, we do allow $\f{1}{Z_{,\aa}} \to 0$.
We will discuss what types of singularities are allowed by a finite energy $E$ in \S \ref{sec:13a}.
In particular we will show that our energy class $ E(t)<\infty$ allows for non-right angles at the corner with angle $<\frac\pi 4$ as well as 
angled crest type interfaces with interior angles at the singularities $<\frac\pi 2$, which coincides with the angles of the self-similar solutions constructed in \cite{wu:self}. The Stokes extreme waves have interior angle $=\frac{2\pi}3$ at the singularity, hence it is not in our energy class $ E(t)<\infty$. 

The lack of control of $\nm{Z_{,\aa}}_{L^\infty}$ in $E$, or equivalently the lack of a positive lower bound for $-\f{\partial P}{\partial \bf n}$ means that we need to circumvent in the proof of our a priori estimate, Theorem~\ref{maintheorem}.


\subsection{The main result}\label{sec:38}
We now state our main result. 

\begin{theorem}\label{maintheorem}
There exists a polynomial $p=p(x)$ with universal coefficients such that for any solution of water wave equations \eqref{eq:179}-\eqref{eq:178}-\eqref{eq:208}-\eqref{eq:210}-\eqref{eq:246a}-\eqref{eq:1026}-\eqref{eq:1026a}-\eqref{eq:2024a} with  $(Z_t, Z_{tt})\in C^l([0, T], H^{k-l}(S^1)\times H^{k-l}(S^1))$ for $l=0,1$ and $k\ge 4$,
\begin{equation}
  \label{eq:690}
  \f{d}{dt} E(t) \le  p\paren{E(t)}
\end{equation}
for all $t \in [0,T]$.
\end{theorem}
Observe that the quantities $\AA$, $\AA_t$ and $h_t\circ h\i$ in equations \eqref{eq:179}-\eqref{eq:178} are given by formulas \eqref{eq:208}, \eqref{eq:210} and \eqref{eq:246a}. By \eqref{eq:228}, \eqref{eq:305} and \eqref{eq:251}, the assumption in Theorem~\ref{maintheorem} implies, for $l=0,1,\ k\ge 4$,
\begin{equation}\label{assumption2}
\paren{\f{1}{Z_{,\aa}}, A_1, h_t\circ h^{-1}, \frac{\AA_t}{\AA}}\in C^l([0, T], H^{k-l}(S^1)\times H^{k-l}(S^1)\times H^{k-l}(S^1)\times H^{k-l}(S^1)).
\end{equation}

{\bf Remark}: 1. Observe that in Theorem~\ref{maintheorem} no regularity assumptions are made on $Z$. This is because by substituting \eqref{eq:251}, \eqref{eq:208}, \eqref{eq:210} and \eqref{eq:246a} into \eqref{eq:178},  we see that the quasilinear equation \eqref{eq:178} is an equation of the velocity $Z_t$ and acceleration $Z_{tt}$, the quantity $Z$ itself doesn't appear explicitly. The assumption of Theorem~\ref{maintheorem} is consistent with this fact. 

2. It appears that there is an $\infty\cdot 0$ ambiguity in the definition of $E_b$ if we allow $Z_\aa\to \infty$.  This can be resolved by expanding out  \eqref{eq:408},
\begin{equation}\label{eb}
E_b=\int_I\paren{\frac{ \abs{  \bar Z_{tt,\aa}-Z_{t,\aa}D_\aa\bar Z_t}^2}{A_1}+ i \partial_\aa (D_\aa \bar Z_t) \overline{D_\aa \bar Z_t}+ |\bar Z_{t,\aa}|^2}\,d\aa,
\end{equation}
and use the RHS of \eqref{eb} as the definition for $E_b$ and directly take derivative to $t$ to the RHS of \eqref{eb}. We opt for the current version \eqref{eq:408} for the clarity of the origin and the more intuitive proof associate with this definition. The same remark applies to \S\ref{sec:quants}, \S\ref{sec:E2}, \S\ref{controlE1} and \S\ref{sec:36e}, where the $\frac{\infty}{\infty}$ ambiguity can be resolved by using some different algebraic identities.

3. The existence of solutions to the Cauchy problem in the class where $E(t)<\infty$ is obtained by mollifying the initial data and taking the limit of the sequence of approximating solutions. In the proof for the existence theorem we will apply Theorem~\ref{maintheorem} only to the approximating sequence which satisfy in addition that $Z_{,\aa}\in L^\infty$, see \cite{wu2015}. What's important is that the inequality \eqref{eq:690} doesn't depend on any $\|Z_{,\aa}\|_{L^\infty}$ bound.

\subsection{The proof}\label{sec:39}

In the proof we will switch freely between Lagrangian and Riemann mapping variables: we will use Lagrangian coordinates when we need to take a time derivative, but use Riemann mapping variables when we need to estimate terms, since that gives us access to the easily invertible $(I-\HH)$ operator. We choose the initial parametrization of the interface such that $h(\alpha, 0)-\alpha\in C^1(S^1)$.\footnote{Notice that the a priori estimate \eqref{eq:690} is independent of the initial parametrization.} By basic ODE theorey and \eqref{assumption2}: $h_t\circ h\i\in C^1([0, T], H^2(S^1))$, we know $h(\alpha, t)-\alpha\in C^1([0, T], C^1(S^1))$.

We start with writing the energy $E_{a,\th}$ in Lagrangian coordinates. We use \eqref{eq:209} to calculate
\begin{equation}
  \label{eq:331}
  \begin{aligned}
  \int_I i \paren{\partial_\aa \paren{\f{1}{Z_{,\aa}} (\th \circ h\i)}}&\paren{\f{1}{\bar{Z}_{,\aa}} (\bar{\th} \circ h\i)} d\aa
  = \Re \int_I i  \paren{\partial_\a \paren{\th \f{h_\a}{z_\a}}} \paren{\bar{\th} \f{h_\a}{\bar{z}_\a}} d\a
\\ & =  \Re \int_I i \f{ h_\a^2}{ \abs{z_\a}^2} (\partial_\a \th) \bar{\th} d\a + \Re \int_I \paren{i \f{1}{\bar{z}_\a}\partial_\a \f{h_\a}{z_\a}} \abs{\th}^2 h_\a d\a 
\\&= \Re \int_I (i \af \partial_\a \th) \bar{\th} \f{h_\a}{A_1 \circ h} d\a + \Re \int_I \paren{i \f{1}{\bar{z}_\a}\partial_\a \f{h_\a}{z_\a}} \abs{\th}^2 h_\a d\a,
\end{aligned}
\end{equation}
so
\begin{equation}
  \label{eq:328}
  E_{a,\th}(t) = \int_I \abs{\th_t}^2 \f{h_\a}{A_1 \circ h} d\a + \Re \int_I (i \af \partial_\a \th) \bar{\th} \f{h_\a}{A_1 \circ h} d\a + \Re \int_I \paren{i \f{1}{\bar{z}_\a}\partial_\a \f{h_\a}{z_\a}} \abs{\th}^2 h_\a d\a +  \int_I \abs{\th}^2 \f{h_\a}{A_1 \circ h} d\a.
\end{equation}

We prove \eqref{eq:690} by differentiating each component of $E(t)$ in time. For the two main energies, $E_a$ and $E_b$, we then integrate by parts to arrive at a term $\partial_t^2\theta+i\af \partial_\alpha \theta$ and use the basic equation $\partial_t^2\theta+i\af \partial_\a \theta=G_\theta$ to replace it with $G_\theta$.  What remain to be estimated will be $G_\theta$, along with several ancillary terms.  We control those quantities  in  \S \ref{sec:43} through \S \ref{controlE1} in terms of a polynomial of the energy.

\subsubsection{The estimate for \texorpdfstring{$E_a$}{Ea}}\label{sec:40a}
We begin by differentiating $E_a$ with respect to $t$.

We will work initially with general $\th$ satisfying $\bound{\th} = 0$, $(I-\HH) (\th \circ h\i) = 0$, and the basic equation \eqref{eq:171}, and then we will specialize to the $\th = D_\a^2 \bar{z}_t$ in our energies.  The periodicity ensures that there is no boundary term in the integration by parts.

We differentiate \eqref{eq:328} with respect to $t$ and use the fact that $\f{\af h_\a}{(A_1 \circ h)} = \f{h_\a^2}{\abs{z_\a}^2}$ (by \eqref{eq:209} or equivalently the definition for $\AA$ and $A_1$) in the following calculation.
\begin{equation}
  \label{eq:385}
  \begin{aligned}
  \f{d}{dt} E_{a,\th}(t) & = \int (\th_{tt} \bar{\th}_t + \th_t \bar{\th}_{tt}) \f{h_\a}{A_1 \circ h} d\a  + \int \abs{\th_t}^2 \f{h_{t\a}}{A_1 \circ h} d\a - \int \abs{\th_t}^2 \f{h_\a}{A_1 \circ h} \f{(A_1 \circ h)_t}{A_1 \circ h} d\a
\\ & + \Re \int i \paren{\f{h_\a^2}{\abs{z_\a}^2}}_t  \th_\a \bar{\th} d\a +\underbrace{\Re \int i \f{h_\a^2}{\abs{z_\a}^2} \th_{t\a} \bar{\th} d\a}_{-\Re \int i \paren{\f{h_\a^2}{\abs{z_\a}^2}}_\a \th_t \bar{\th} d\a+ \Re \int \f{h_\a^2}{\abs{z_\a}^2} \th_t \bar{i\th_\a}d\a} + \Re \int i \f{h_\a^2}{\abs{z_\a}^2} \th_\a \bar{\th}_t  d\a
\\ & + \Re \int i \paren{\f{1}{\bar{z}_\a} \partial_\a \f{h_\a}{z_\a}}_t \abs{\th}^2 h_\a d\a + \Re \int i \paren{\f{1}{\bar{z}_\a} \partial_\a \f{h_\a}{z_\a}} (\th_t \bar{\th} + \th \bar{\th}_t) h_\a d\a + \Re \int i \paren{\f{1}{\bar{z}_\a} \partial_\a \f{h_\a}{z_\a}} \abs{\th}^2 h_{t\a} d\a
\\ & +  \int (\th_t \bar{\th} + \bar{\th}_t \th) \f{h_\a}{A_1 \circ h} +  \int \abs{\th}^2 \f{h_{t\a}}{A_1 \circ h} d\a -  \int \abs{\th}^2 \f{h_\a}{A_1 \circ h} \f{(A_1 \circ h)_t}{A_1 \circ h} d\a
\\ &= \int 2 \Re \paren{(\th_{tt} + i\af \th_\a)\bar{\th}_t} \f{h_\a}{A_1 \circ h} d\a
 + \int \paren{\abs{\th_t}^2+\abs{\th}^2}   \paren{\f{h_{t\a}}{h_\a} - \f{(A_1 \circ h)_t}{A_1 \circ h}} \f{h_\a}{A_1 \circ h} d\a 
\\ & + \Re \int i \paren{\f{1}{\bar{z}_\a} \partial_\a \f{h_\a}{z_\a}}_t \abs{\th}^2 h_\a d\a
 + \Re \int i \paren{\f{1}{\bar{z}_\a} \partial_\a \f{h_\a}{z_\a}} (\th_t \bar{\th} + \th \bar{\th}_t) h_\a d\a 
 \\&+ \Re \int i \paren{\f{1}{\bar{z}_\a} \partial_\a \f{h_\a}{z_\a}} \abs{\th}^2 \f{h_{t\a}}{h_\a} h_\a d\a
+   2 \Re \int \th_t \bar{\th} \f{h_\a}{A_1 \circ h} d\a 
\\& - \Re \int i \paren{\f{h_\a^2}{\abs{z_\a}^2}}_\a \th_t \bar{\th} d\a
 + \Re  \int i \paren{\f{h_\a^2}{\abs{z_\a}^2}}_t  \th_\a \bar{\th} d\a.
\end{aligned}
\end{equation}

Now we show how we control each of these terms.

For the first, we replace $\th_{tt} + i \af \th_\a$ with the RHS $G_\th$ by the main equation \eqref{eq:171} and then use Cauchy-Schwarz inequality:
\begin{equation}
  \label{eq:43}
  \int 2 \Re \paren{(\th_{tt} + i\af \th_\a)\bar{\th}_t} \f{h_\a}{A_1 \circ h} d\a  
\\ \lec  \paren{\int  \abs{G_\th}^2\f{h_\a}{A_1 \circ h}d\a}^{1/2} \paren{\int \abs{\bar{\th}_t}^2 \f{h_\a}{A_1 \circ h}d\a}^{1/2}.
\end{equation}
The first factor,  which involves the RHS of the basic equation, 
is the main term to control. For $\th= D_\a^2 \bar{z}_t$, 
\begin{equation}
  \label{eq:391}
  G_\th = D_\a^2(-i\af_t \bar{z}_\a) + [\partial_t^2 +i\af \partial_\a,D_\a^2] \bar{z}_t.
\end{equation}
We estimate these terms in \S \ref{controlE1}.

In \S \ref{sec:98a}, we will control
\begin{equation}
  \label{eq:189}
  \abs{\Re \int i \paren{\f{1}{\bar{z}_\a} \partial_\a \f{h_\a}{z_\a}}_t \abs{\th}^2 h_\a d\a} \lec \eqref{eq:368}.
\end{equation}
Because $\f{1}{{z}_\a} \partial_\a \f{h_\a}{z_\a}= D_\a \f{h_\a}{z_\a}=\paren{D_\aa \f{1}{Z_{,\aa}}}\circ h$, we estimate
\begin{equation}
  \label{eq:190}
    \abs{\Re \int i \paren{\f{1}{\bar{z}_\a} \partial_\a \f{h_\a}{z_\a}} (\th_t \bar{\th} + \th \bar{\th}_t) h_\a d\a}     \lec \nm{D_\aa \f{1}{Z_{,\aa}}}_{L^\infty} \nm{A_1}_{L^\infty} E_{a,\th}.
 \end{equation}
Similarly, we estimate
\begin{equation}
  \label{eq:191}
    \abs{\Re \int i \paren{\f{1}{\bar{z}_\a} \partial_\a \f{h_\a}{z_\a}} \abs{\th}^2 \f{h_{t\a}}{h_\a} h_\a d\a} 
    \lec \nm{D_\aa \f{1}{Z_{,\aa}}}_{L^\infty} \nm{\f{h_{t\a}}{h_\a}}_{L^\infty} \nm{A_1}_{L^\infty} E_{a,\th}.
\end{equation}
We observe that
\begin{equation}
  \label{eq:388}
  \nm{\f{\paren{\f{h_\a^2}{\abs{z_\a}^2}}_\a}{h_\a}}_{L^\infty} = \nm{\partial_\aa \f{1}{\abs{Z_{,\aa}}^2}}_{L^\infty} \le 2 \nm{D_\aa \f{1}{Z_{,\aa}}}_{L^\infty}, 
\end{equation}
so, using Cauchy-Schwarz inequality, we have
\begin{equation}
  \label{eq:389}
  \abs{- \Re \int i \paren{\f{h_\a^2}{\abs{z_\a}^2}}_\a \th_t\bar{\th}d\a} \lec \nm{A_1}_{L^\infty}  \nm{D_\aa \f{1}{Z_{,\aa}}}_{L^\infty} E_{a,\th}.
\end{equation}
In \S \ref{sec:48} we control 
\begin{equation}
  \label{eq:193}
  \Re  \int i \paren{\f{h_\a^2}{\abs{z_\a}^2}}_t  \th_\a \bar{\th} d\a  \lec \eqref{eq:552}.
\end{equation}
We estimate  the remaining two terms of \eqref{eq:385} by Cauchy-Schwarz and H\"older's inequalities. 

We now combine these estimates and specialize to $\th = D_\a^2 \bar{z}_t$. Each of the remaining factors we will control in \S \ref{sec:43}; we list the location of the final estimate for each quantity of the following in the subscripts. We get
\begin{equation}
  \label{eq:66}
  \begin{aligned}
    \abs{\f{d}{dt} E_{a}} & \lec \underbrace{\nm{G_{D_\a^2 \bar{z}_t}}_{L^2(\f{h_\a}{A_1 \circ h})}}_{\lec \eqref{eq:1600}} E_a^{1/2} 
     + \underbrace{\nm{\f{h_{t\a}}{h_\a}}_{L^\infty}}_{\lec \eqref{eq:451}}E_a + \underbrace{\nm{\f{(A_1 \circ h)_t}{A_1 \circ h}}_{L^\infty}}_{\lec \eqref{eq:455}}E_a 
\\ & + \underbrace{\paren{1 + \nm{\f{h_{t\a}}{h_\a}}_{L^\infty}}}_{\lec 1+ \eqref{eq:451}} \nm{A_1}_{L^\infty} \underbrace{\nm{D_\aa \f{1}{Z_{,\aa}}}_{L^\infty}}_{\lec \eqref{eq:1023}} E_a
 + E_a
\\ & + \underbrace{\Re \int i \paren{\f{1}{\bar{z}_\a} \partial_\a \f{h_\a}{z_\a}}_t \abs{\th}^2 h_\a d\a}_{\lec \eqref{eq:368}}
 + \underbrace{\Re  \int i \paren{\f{h_\a^2}{\abs{z_\a}^2}}_t  \th_\a \bar{\th} d\a}_{\lec \eqref{eq:552}}.
  \end{aligned}
\end{equation}

\subsubsection{The estimate for \texorpdfstring{$E_b$}{Eb}}\label{sec:40}
Now we consider our second term $E_b$. Once again, we work first with general $\th$ satisfying $\bound{\th} = 0$, $(I-\HH)(\th \circ h\i) = 0$, and the main equation \eqref{eq:171}.  Then we specialize to $\th = D_\a \bar{z}_t$.  The periodicity ensures there is no boundary term when we integrate by parts. We differentiate \eqref{eq:172} with respect to $t$:
\begin{equation}
  \label{eq:224}
  \begin{aligned}
    \f{d}{dt} E_{b,\th}(t) & = \int \f{1}{\mathfrak{a}} (\th_{tt} \bar{\th}_t + \th_t \bar{\th}_{tt}) d\a - \int \f{\mathfrak{a}_t}{\mathfrak{a}} \f{1}{\mathfrak{a}} \abs{\th_t}^2 d\a 
+ \underbrace{\int i \th_{t\a} \bar{\th} d\a}_{= {\int \bar{i \th_\a} \th_t d\a}} + \int i \th_\a \bar{\th}_t d\a
\\ & + \int \f{(A_1 \circ h)}{\mathfrak{a}} (\th_t \bar{\th} + \th \bar{\th}_t) d\a + \int \f{(A_1 \circ h)_t}{\af} \abs{\th}^2 d\a - \int \f{\mathfrak{a}_t}{\mathfrak{a}} \f{(A_1 \circ h)}{\mathfrak{a}} \abs{\th}^2 d\a
\\ & 
= 2 \Re \int \f{G_\th}{\mathfrak{a}} \bar{\th}_t d\a  - \int \f{\mathfrak{a}_t}{\mathfrak{a}} \f{1}{\mathfrak{a}} \abs{\th_t}^2 d\a 
\\ & + \int \f{A_1\circ h}{\mathfrak{a}} (\th_t \bar{\th} + \th \bar{\th}_t) d\a + \int \paren{\f{(A_1 \circ h)_t}{(A_1 \circ h)} - \f{\mathfrak{a}_t}{\mathfrak{a}}} \f{(A_1 \circ h)}{\mathfrak{a}} \abs{\th}^2 d\a.
\end{aligned}
\end{equation}

By H\"older and Cauchy-Schwarz inequalities, we conclude that
\begin{equation}
  \label{eq:173}
  \abs{\f{d}{dt} E_{b,\th}(t)} \lec  \nm{\f{G_\th}{\sqrt{\mathfrak{a}}}}_{L^2} E_{b,\th}^{1/2}  + \paren{\|A_1\|_{L^\infty}^{1/2} + \nm{\f{\mathfrak{a}_t}{\mathfrak{a}}}_{L^\infty} + \nm{\f{(A_1 \circ h)_t}{(A_1 \circ h)}}_{L^\infty}} E_{b,\th}.
\end{equation}
For $\th = D_\a \bar{z}_t$, we control $\nm{\f{G_\th}{\sqrt{\mathfrak{a}}}}_{L^2}$ in \S \ref{sec:E2}, at  \eqref{eq:311}. We control $\|A_1\|_{L^\infty}$ at \eqref{eq:361},  $\nm{\f{\af_t}{\af}}_{L^\infty}$ at \eqref{eq:424} and  $\nm{\f{(A_1 \circ h)_t}{(A_1 \circ h)}}_{L^\infty}$ at \eqref{eq:455}.

\subsubsection{The estimate for \texorpdfstring{$\abs{z_{tt}(\a_0,t)-i}$}{| ztt(a0,t) - i |}}
Finally, we show that we can control $\frac d{dt}\abs{\bar{z}_{tt}(\a_0,t)-i}$. By differentiating with respect to $t$, we have, by the basic equations \eqref{eq:178}-\eqref{eq:179},
\begin{equation}
  \label{eq:664}
  \begin{aligned}
    \f{d}{dt}\abs{\bar{z}_{tt}(\a_0) - i} & \le \abs{\bar{z}_{ttt}(\a_0)}
     = \abs{\frac{\AA_t}{\AA}\circ h(\a_0)+\bar {D_\a z_t}(\a_0)}\abs{\bar{z}_{tt}(\a_0) - i}
    \\ & \lec \paren{\nm{\f{\AA_t}{\AA}}_{L^\infty} + \nm{D_\a \bar{z}_t}_{L^\infty}} \abs{\bar{z}_{tt}(\a_0) - i}.
  \end{aligned}
\end{equation}
We control $\nm{\f{\AA_t}{\AA}}_{L^\infty}$ below at \eqref{eq:424} and $\nm{D_\a \bar{z}_t}_{L^\infty}$ at \eqref{eq:15}.

\subsection{Outline of the remainder of the proof}\label{sec:42}
In  sections \S \ref{sec:43} through \S \ref{controlE1}, we complete the proof of the a priori inequality \eqref{eq:690}.

In \S \ref{sec:43}, we control various quantities that are necessary for our proof. In \S \ref{sec:quants}, we carefully list the basic quantities controlled by our energy. In \S \ref{sec:at-a}-\S \ref{sec:46}, we estimate various other quantities that are listed above in \S \ref{sec:39}. 
In appendix \S \ref{quantities}, we list and give references to all the quantities controlled in \S \ref{sec:43}, which we then use, sometimes without citation, in \S \ref{sec:98a} through \S \ref{controlE1}.

In \S \ref{sec:98a} and \S \ref{sec:48} we estimate the terms from \eqref{eq:189} and \eqref{eq:193} in the estimate of $\f{d}{dt}E_a$ above. Finally, in \S \ref{sec:E2} and \S \ref{controlE1} we conclude the estimates for $\f{d}{dt}E_b$ and $\f{d}{dt}E_a$, respectively, by controlling the $G_\th$ terms, completing the proof of Theorem~\ref{maintheorem}.

The basic approach for many of the estimates is to try and use the fact that certain quantities are purely real-valued and others are holomorphic to express the terms in question as commutators involving the Hilbert transform, and then use the commutator estimates from \S \ref{sec:31} to avoid loss of derivatives. 
Because our estimates are very tight, we have to take care in using different estimates for different terms, including treating certain terms as commutators while keeping others in $(I-\HH)$ form. 
Very often we  have to carefully expand the quantities, and then decompose the factors and regroup the terms to make sure no further cancellations are possible and desired estimates can be obtained.  We will give enough details to facilitate reading. 

We use $C(E)$ to indicate a universal polynomial of $E$,  which may differ from line to line.

Throughout the remaining derivations, we will repeatedly rely on the identity \eqref{eq:209}.

\section{Quantities controlled by our energy}\label{sec:43}
Here we collect together many quantities that are controlled by our energy. 

\subsection{Basic quantities controlled by the energy}\label{sec:quants}

In this section, we present a list of basic quantities controlled by our energy. Because conjugations and commutations of $\partial_t$ with $D_\a$ add complexity, we take care to list some of those estimates as well. We list all of the basic terms controlled here at \eqref{eq:1050} below.

We start with \eqref{eq:400} and \eqref{eq:408}. $E_a$ and $E_b$ directly control
  \begin{align}
   \label{eq:498}
\nm{(\partial_t+\frak b\partial_\aa)D_\aa^2\bar Z_t}_{L^2},  \nm{D_\a^2 \bar{z}_t}_{L^2(\f{h_\a}{A_1 \circ h} d\a)},\quad\nm{D_\aa^2 \bar{Z}_t}_{L^2(\f{1}{A_1} d\aa)}, &\quad  \nm{\f{1}{Z_{,\aa}} D_\aa^2 \bar{Z}_t}_{\dot{H}^{1/2}}  \le  E_a^{1/2},\\
 \label{eq:239}
   \paren{ \int \abs{D_\a \bar{z}_t}^2 \f{d\a}{\af}}^{1/2}\le \nm{\bar{Z}_{t,\aa}}_{L^2} &\le E_b^{1/2},
\end{align}
where we used $A_1\ge 1$ \eqref{eq:394} in \eqref{eq:239}. By \eqref{eq:208}, using  the commutator estimate \eqref{eq:32}, and then \eqref{eq:239},  we have
\begin{equation}
  \label{eq:361}
    \nm{A_1}_{L^\infty} 
    \lec \nm{Z_{t,\aa}}_{L^2}^2 + 1
     \lec E_b + 1.
\end{equation}
Thanks  to \eqref{eq:361}, we can now control $\nm{D_\a^2 \bar{z}_t}_{L^2(h_\a d\a)}=\nm{D_\aa^2 \bar{Z}_t}_{L^2}$ by a polynomial of our energy $E$.

Now we control $\nm{D_\a \bar{z}_t}_{L^\infty} = \nm{D_\aa \bar{Z}_t}_{L^\infty}$. We work in Riemann mapping variables and use the weighted Sobolev inequality \eqref{eq:638} with weight $\omega =\f{1}{\abs{Z_{,\aa}}^2}$ (and $\e = 1$). Note that $\avg (D_\aa \bar{Z}_t)^2 = 0$ by \eqref{eq:48}. This gives
\begin{equation}
  \label{eq:206}
  \begin{aligned}
 \nm{D_\aa \bar{Z}_t}_{L^\infty} & \lec \paren{\int \abs{D_\aa \bar{Z}_t}^2 \abs{Z_{,\aa}}^2 d\aa}^{1/2} + \paren{\int \abs{\partial_\aa D_\aa \bar{Z}_t}^2 \f{1}{\abs{Z_{,\aa}}^2} d\aa}^{1/2}
 \\ & = \nm{\bar{Z}_{t,\aa}}_{L^2} + \nm{D_\aa^2 \bar{Z}_t}_{L^2}
 \le E_b^{1/2}+ \nm{A_1 }_{L^\infty}^{1/2} E_a^{1/2}
\\ & \lec C(E).
\end{aligned}
\end{equation}
 We conclude that
\begin{equation}
  \label{eq:15}
  \nm{D_\a z_t}_{L^\infty} = \nm{D_\aa Z_t}_{L^\infty} = \nm{D_\a \bar{z}_t}_{L^\infty} = \nm{D_\aa \bar{Z}_t}_{L^\infty}  \lec C(E).
\end{equation}

Now we use the commutator identity \eqref{eq:120} to move the $\partial_t$ inside the first term in $E_b$:
\begin{equation}
  \label{eq:403}
\begin{aligned}
  \paren{\int \abs{D_\a \bar{z}_{tt}}^2 \f{d\a}{\af}}^{1/2} & \le \paren{\int \abs{\partial_t D_\a \bar{z}_t}^2 \f{d\a}{\af}}^{1/2} + \paren{\int \abs{[\partial_t,D_\a] \bar{z}_t}^2 \f{d\a}{\af}}^{1/2}
\\ & \le E_b^{1/2} + \nm{D_\a z_t}_{L^\infty} \paren{ \int \abs{D_\a \bar{z}_t}^2 \f{d\a}{\af}}^{1/2}
\\ & \lec C(E)
\end{aligned}
\end{equation}
by \eqref{eq:15} and \eqref{eq:239}.  By changing variables and by \eqref{eq:209}, we conclude from \eqref{eq:403} and \eqref{eq:361} that
\begin{equation}
  \label{eq:406}
\begin{aligned}
   \nm{\bar{Z}_{tt,\aa}}_{L^2} & \lec \nm{A_1}_{L^\infty}^{1/2}  \paren{\int \abs{D_\a \bar{z}_{tt}}^2 \f{d\a}{\af}}^{1/2} 
\\ & \lec C(E).
\end{aligned}
\end{equation}

Note that  $\abs{D_\a^2 f}\ne \abs{D_\a^2 \bar{f}}$. Nevertheless, for generic $f$, we can control $D_\a^2 f$ by $D_\a^2 \bar{f}$ in $L^2({h_\a} d\a)$ norm, at the expense of some lower-order terms.    For notational convenience, we define here
\begin{equation}
  \label{eq:401}
  \abs{D_\a} := \f{1}{\abs{z_\a}} \partial_\a = \f{z_\a}{\abs{z_\a}} D_\a.
  \end{equation}
  We expand:
\begin{equation}
  \label{eq:65aa}
  \begin{aligned}
   D_\a^2 f & = \paren{\f{\abs{z_\a}}{z_\a}}^2 \abs{D_\a}^2 f + \f{\abs{z_\a}}{z_\a} \paren{\abs{D_\a} \f{\abs{z_\a}}{z_\a}} \abs{D_\a} f\\
   D_\a^2 \bar f & = \paren{\f{\abs{z_\a}}{z_\a}}^2 \abs{D_\a}^2 \bar f + \f{\abs{z_\a}}{z_\a} \paren{\abs{D_\a} \f{\abs{z_\a}}{z_\a}} \abs{D_\a} \bar f.
   \end{aligned}
   \end{equation}
Therefore,
   \begin{equation}\label{eq:65a}
   \abs{D_\a^2 f }\le 
\abs{D_\a^2 \bar{f}} + 2 \abs{\abs{D_\a} \f{\abs{z_\a}}{z_\a}} \abs{D_\a \bar{f}}
\end{equation}
and so
\begin{equation}
  \label{eq:283}
    \paren{\int \abs{D_\a^2 f}^2 {h_\a d\a}}^{1/2}    \le  \paren{\int \abs{D_\a^2 \bar{f}}^2 {h_\a d\a}}^{1/2} + 2 \nm{D_\a \bar{f}}_{L^\infty} \paren{\int \abs{D_\a \f{\abs{z_\a}}{z_\a}}^2 {h_\a d\a}}^{1/2}.
\end{equation}
By \eqref{eq:279} and then \eqref{eq:100}, \eqref{eq:8}, \eqref{eq:209}, and the fact $A_1 \ge 1$,
\begin{equation}
  \label{eq:284}
    \abs{D_\a \f{\abs{z_\a}}{z_\a}}^2 h_\a   =  \abs{D_\a \f{\bar{z}_{tt} - i}{\abs{\bar{z}_{tt} -i}}}^2 h_\a 
 \le  \abs{D_\a \bar{z}_{tt}}^2 \f{h_\a}{\abs{\bar{z}_{tt} - i}^2 } 
  =  \abs{D_\a \bar{z}_{tt}}^2 \f{h_\a}{\af^2 \abs{z_\a}^2 } 
\le  \abs{D_\a \bar{z}_{tt}}^2 \f{1}{\af }.
\end{equation}
Plugging this into \eqref{eq:283}, and using \eqref{eq:403}, we get
\begin{equation}
  \label{eq:404}
  \paren{\int \abs{D_\a^2 f}^2 h_\a d\a}^{1/2}
\lec  \paren{\int \abs{D_\a^2 \bar{f}}^2 h_\a d\a}^{1/2}
  + \nm{D_\a \bar{f}}_{L^\infty} C(E).
\end{equation}

We now apply \eqref{eq:404} to $f = z_t$, using \eqref{eq:15} to control $\nm{D_\a \bar{z}_t}_{L^\infty}$ and \eqref{eq:498} and \eqref{eq:361} to control $\nm{D_\aa^2 \bar Z_t}_{L^2}$:
\begin{equation}
  \label{eq:405}
 \nm{D_\aa^2 Z_t}_{L^2}  = \paren{\int \abs{D_\a^2 z_t}^2 h_\a d\a}^{1/2} 
 \lec C(E).
\end{equation}

We now control $\nm{D_\a^2 \bar{z}_{tt}}_{L^2(h_\a d\a)}$. We use the commutator identity \eqref{eq:121} to get
\begin{equation}
  \label{eq:122}
  \begin{aligned}
    \nm{D_\aa^2 \bar{Z}_{tt}}_{L^2} & = \nm{D_\a^2 \bar{z}_{tt}}_{L^2({h_\a d\a})}  
\\ & \le \nm{\partial_t D_\a^2 \bar{z}_{t}}_{L^2({h_\a d\a})} + 2 \nm{(D_\a z_t)D_\a^2 \bar{z}_t}_{L^2({h_\a d\a})} + \nm{(D_\a^2 z_t) D_\a \bar{z}_t}_{L^2({h_\a d\a})}
\\ & \le\nm{A_1}^{1/2}_{L^\infty} E_a^{1/2} + 2\nm{D_\a z_t}_{L^\infty} \nm{D_\a^2 \bar{z}_t}_{L^2({h_\a d\a})} + \nm{D_\a^2 z_t}_{L^2({h_\a d\a})} \nm{D_\a \bar{z}_t}_{L^\infty}
\\ & \lec C(E).
\end{aligned}
\end{equation}
We will also need to control $D_\aa^2 Z_{tt}$; we delay doing this until later, after we control $\nm{D_\a \bar{z}_{tt}}_{L^\infty}$.

We will also at one point need to control
\begin{equation}
  \label{eq:419}
  \begin{aligned}
      \nm{D_\a \partial_t D_\a \bar{z}_t}_{L^2({h_\a}d\a)} & \le \nm{\partial_t D_\a D_\a \bar{z}_t}_{L^2({h_\a}d\a)} + \nm{[\partial_t,D_\a] D_\a \bar{z}_{t}}_{L^2({h_\a}d\a)}
      \\ & \le \|A_1\|_{L^\infty}^{1/2}E_a^{1/2} + \nm{D_\a z_t}_{L^\infty} \nm{D_\a^2 \bar{z}_{t}}_{L^2({h_\a}d\a)}
      \\ & \lec C(E).
\end{aligned}
\end{equation}

We now control $\nm{\bar{z}_{tt} - i}_{L^\infty} = \nm{\bar{Z}_{tt} - i}_{L^\infty}$. Recall from our definition of the energy \eqref{eq:49} that the energy includes $\abs{\bar{z}_{tt}(\a_0,t)-i}$ for some fixed $\a_0 \in I$. Let $\aao = h(\a_0, t)$. Then, by the fundamental theorem of calculus, for arbitrary $\aa \in I$,
\begin{equation}
  \label{eq:29}
  \begin{aligned}
    \abs{\bar{Z}_{tt}(\aa,t)- i}     & \lec \abs{\bar{Z}_{tt}(\aao,t)- i} + \nm{\bar{Z}_{tt,\aa}}_{L^1(I)}
    \\ & \lec \abs{\bar{Z}_{tt}(\aao,t)- i} + \eqref{eq:406}.
  \end{aligned}
\end{equation}
We conclude that
\begin{equation}
  \label{eq:378}
  \begin{aligned}
    \nm{z_{tt} + i}_{L^\infty} & = \nm{Z_{tt}+i}_{L^\infty} = \nm{\bar{z}_{tt}-i}_{L^\infty} = \nm{\bar{Z}_{tt}-i}_{L^\infty}
  \\ & \lec C(E).
\end{aligned}
\end{equation}
Because of \eqref{eq:251} and \eqref{eq:394}, we can also conclude that
\begin{equation}
  \label{eq:357}
  \nm{\f{1}{Z_{,\aa}}}_{L^\infty} \lec C(E).
\end{equation}

We use this to control $\nm{D_\a \bar{z}_{tt}}_{L^\infty}$, using the weighted Sobolev inequality \eqref{eq:396} in Riemann mapping variables with weight $\omega = \f{1}{\abs{Z_{,\aa}}^2}$ (and $\e = 1$):\footnote{Note that unlike our proof for $\nm{D_\aa \bar{Z}_t}_{L^\infty}$ at \eqref{eq:206} above, we don't necessarily have that $\avg (D_\aa \bar{Z}_{tt})^2$ is zero, so we get a third term in the Sobolev inequality.}
\begin{equation}
  \label{eq:17}
  \begin{aligned}
     \nm{D_\a z_{tt}}_{L^\infty} & = \nm{D_\aa Z_{tt}}_{L^\infty} = \nm{D_\a \bar{z}_{tt}}_{L^\infty} = \nm{D_\aa \bar{Z}_{tt}}_{L^\infty}
\\ & \lec \nm{\bar{Z}_{tt,\aa}}_{L^2} + \nm{D_\aa^2 \bar{Z}_{tt}}_{L^2} + \paren{\int \abs{D_\aa \bar{Z}_{tt}}^2 d\aa}^{1/2}
\\ & \lec \paren{1 + \nm{1/Z_{,\aa}}_{L^\infty}} \nm{\bar{Z}_{tt,\aa}}_{L^2} + \nm{D_\aa^2 \bar{Z}_{tt}}_{L^2} 
\\ & \lec \paren{1 +\eqref{eq:357}} \eqref{eq:406} +  \eqref{eq:122}\\&\lec C(E).
\end{aligned}
\end{equation}

Finally, we use \eqref{eq:404}, \eqref{eq:122}, and \eqref{eq:17} to control $D_\a^2 z_{tt}$ and $D_\aa^2 Z_{tt}$:
\begin{equation}
  \label{eq:418}
  \begin{aligned}
    \nm{D_\aa^2 Z_{tt}}_{L^2} &=  \nm{D_\a^2 z_{tt}}_{L^2({h_\a d\a})} 
\\ & \lec \nm{D_\a^2 \bar{z}_{tt}}_{L^2(h_\a d\a)}+  \nm{D_\a \bar{z}_{tt}}_{L^\infty}C(E)\\ & \lec \eqref{eq:122} + \eqref{eq:17} C(E)\lec C(E).
  \end{aligned}
\end{equation}

To sum up, we have the following quantities and their counterparts in Lagrangian coordinates controlled by universal polynomials of the energy $E$: 
\begin{equation}\label{eq:1050}
\begin{aligned}
&\nm{D_\aa^2 \bar{Z}_{tt}}_{L^2}, \nm{D_\aa^2 {Z}_{tt}}_{L^2},  \nm{D_\aa^2 \bar{Z}_t}_{L^2}, \nm{D_\aa^2 {Z}_t}_{L^2}, \nm{D_\a \partial_t D_\a \bar{z}_t}_{L^2({h_\a}d\a)},\\& \nm{\f{1}{Z_{,\aa}} D_\aa^2 \bar{Z}_t}_{\dot{H}^{1/2}}, \nm{D_\aa \bar{Z}_{tt}}_{L^\infty}, \nm{D_\aa {Z}_{tt}}_{L^\infty},  \nm{D_\aa \bar{Z}_t}_{L^\infty}, \nm{D_\aa {Z}_t}_{L^\infty},\\&  \nm{ \bar{Z}_{tt, \aa}}_{L^2},   \nm{ \bar{Z}_{t, \aa}}_{L^2},  \int \abs{D_\a \bar{z}_t}^2 \f{d\a}{\af},\  \int \abs{D_\a \bar{z}_{tt}}^2 \f{d\a}{\af},  \ \nm{\f{1}{Z_{,\aa}}}_{L^\infty}, \nm{Z_{tt}+i}_{L^\infty} , \nm{A_1}_{L^\infty} \\&
\lec C(E).
\end{aligned}
\end{equation}

\subsection{Controlling \texorpdfstring{$\nm{\f{\af_t}{\af}}_{L^\infty}$}{|| at/a || L infinity}} \label{sec:at-a}

We now show that we can control $\nm{\f{\af_t}{\af}}_{L^\infty}$, using \eqref{eq:210}.  Because $A_1 \ge 1$ \eqref{eq:394}, it suffices to control
\begin{equation}
  \label{eq:267}
  \nm{2[Z_t,\HH] \bar{Z}_{tt,\aa} + 2[Z_{tt},\HH] \bar{Z}_{t,\aa} - [Z_t, Z_t; D_\aa \bar{Z}_t]}_{L^\infty}.
\end{equation}
We control the first two terms by \eqref{eq:32}, and the last term by H\"older's inequality and then Hardy's inequality \eqref{eq:77}.
We have
\begin{equation}
  \label{eq:424}
  \nm{\f{\af_t}{\af}}_{L^\infty} \lec  \nm{Z_{t,\aa}}_{L^2} \nm{\bar{Z}_{tt,\aa}}_{L^2}+ \nm{Z_{t,\aa}}_{L^2}^2 \nm{D_\aa \bar{Z}_t}_{L^\infty}\lec C(E).
\end{equation}

\subsection{Controlling \texorpdfstring{$\nm{\partial_\aa \f{1}{Z_{,\aa}}}_{L^2}$}{|| partial a' 1/Z,a' || L2}}\label{sec:44}
Recall from \eqref{eq:251} that
 $ \f{1}{Z_{,\aa}} = i \f{\bar{Z}_{tt} - i}{A_1}$.
Therefore,
\begin{equation}
  \label{eq:252}
  \partial_\aa \f{1}{Z_{,\aa}} = i \f{\bar{Z}_{tt,\aa}}{A_1} - i \f{\bar{Z}_{tt} -i }{A_1^2} \partial_\aa A_1.
\end{equation}
Because $A_1 \ge 1$ \eqref{eq:394}, we can control the first term by $\nm{\bar{Z}_{tt,\aa}}_{L^2}$.  Now we address the second term.

We recall that  $ A_1 = \Im\paren{ - [Z_t,\HH] \bar{Z}_{t,\aa}} + 1$ \eqref{eq:208}. Therefore,
\begin{equation}
  \label{eq:254}
  \begin{aligned}
      \partial_\aa A_1 & = \partial_\aa \Im \f{-1}{\hdenomconst} \int (Z_t(\aa) - Z_t(\bb)) \cot(\f{\pi}{2}(\aa -\bb)) \bar{Z}_{t,\bb} d\bb
\\ & = - \Im Z_{t,\aa} \HH \bar{Z}_{t,\aa} + \Im \f{1}{\hdenomconst} \int \f{\pi}{2} \f{(Z_t(\aa) - Z_t(\bb))}{\sin^2(\f{\pi}{2}(\aa - \bb))} \bar{Z}_{t,\bb}(\bb) d\bb 
\\ & =   \Im \f{1}{\hdenomconst} \int \f{\pi}{2} \f{(Z_t(\aa) - Z_t(\bb))}{\sin^2(\f{\pi}{2}(\aa - \bb))} \bar{Z}_{t,\bb}(\bb) d\bb,
\end{aligned}
\end{equation}
where the first term disappears because $\HH \bar{Z}_{t,\aa} = \bar{Z}_{t,\aa}$ \eqref{eq:474} and so $Z_{t,\aa} \HH \bar{Z}_{t,\aa}$ is purely real. Therefore, multiplying \eqref{eq:254} by $\abs{\bar{Z}_{tt}(\a) - i}$ 
 and splitting into two parts, we have
\begin{equation}
  \label{eq:255}
  \begin{aligned}
      \abs{\bar{Z}_{tt} - i} \partial_\aa A_1 & = \Im \f{1}{\hdenomconst} \int \f{\pi}{2} \f{(Z_t(\aa) - Z_t(\bb))}{\sin^2(\f{\pi}{2}(\aa - \bb))} (\abs{\bar{Z}_{tt}(\aa)-i} - \abs{\bar Z_{tt}(\bb)-i}) \bar{Z}_{t,\bb}(\bb) d\bb 
\\&+ \Im \f{1}{\hdenomconst} \int \f{\pi}{2} \f{(Z_t(\aa) - Z_t(\bb))}{\sin^2(\f{\pi}{2}(\aa - \bb))}\abs{\bar{Z}_{tt}(\bb) - i} \bar{Z}_{t,\bb}(\bb) d\bb 
\\ & = I + II.
\end{aligned}
\end{equation}
We need to control $\nm{I}_{L^2}$ and $\nm{II}_{L^2}$. 
By \eqref{eq:113},
\begin{equation}
  \label{eq:257}
  \nm{I}_{L^2} \lec\nm{Z_{t,\aa}}_{L^2}  \nm{\bar{Z}_{tt,\aa}}_{L^2} \nm{\bar{Z}_{t,\aa}}_{L^2} = \nm{\bar{Z}_{t,\aa}}_{L^2}^2 \nm{\bar{Z}_{tt,\aa}}_{L^2}.
\end{equation}
For $II$, we replace  $\abs{\bar Z_{tt}(\bb)-i}$ by $\abs{\f{(-i A_1(\bb))}{Z_{,\bb}}}$ \eqref{eq:251} and use estimate \eqref{eq:258}, noticing that $\f{1}{Z_{,\bb}}\bar{Z}_{t,\bb}=D_\bb \bar{Z}_t$:
\begin{equation}
  \label{eq:427}
  \nm{II}_{L^2} \lec \nm{Z_{t,\aa}}_{L^2} \nm{A_1 D_\aa \bar{Z}_{t}}_{L^\infty} \le \nm{Z_{t,\aa}}_{L^2} \nm{D_\aa \bar{Z}_t}_{L^\infty} \nm{A_1}_{L^\infty}.
\end{equation}

We conclude that
\begin{equation}
  \label{eq:1018}
  \begin{aligned}
  \nm{ \paren{\bar{Z}_{tt} - i} \partial_\aa A_1  }_{L^2} &\lec \nm{\bar{Z}_{tt,\aa}}_{L^2}\nm{Z_{t,\aa}}_{L^2}^2 + \nm{Z_{t,\aa}}_{L^2} \nm{D_\aa \bar{Z}_t}_{L^\infty} \nm{A_1}_{L^\infty}\\& \lec C(E)
  \end{aligned}
\end{equation}
and
\begin{equation}
  \label{eq:425}
  \begin{aligned}
  \nm{\partial_\aa \f{1}{Z_{,\aa}}}_{L^2} &\lec \nm{\bar{Z}_{tt,\aa}}_{L^2}+  \nm{ \paren{\bar{Z}_{tt} - i} \partial_\aa A_1  }_{L^2}
    \\& \lec C(E).
  \end{aligned}
\end{equation}

\subsection{Controlling 
\texorpdfstring{$\nm{\partial_\aa (I-\HH) \f{Z_t}{Z_{,\aa}}}_{L^\infty}$, \texorpdfstring{$\nm{\f{h_{t\a}}{h_\a}}_{L^\infty}$}{|| hta/ha || L infinity}}{|| partial a' (I-H) Zt/Z,a' || L infinity} and related quantities}\label{sec:45c}

Observe that by the assumption of Theorem~\ref{maintheorem}, $\frac{Z_t}{Z_{,\aa}}$, $Z_t$, $\frac1{Z_{,\aa}}$ are in $C^1([0, T], C^2(S^1))$. We begin with computing
\begin{equation}
  \label{eq:127}
  \begin{aligned}
    \partial_\aa (I-\HH) \f{Z_t}{Z_{,\aa}} 
    & =   (I-\HH) D_\aa Z_t + (I-\HH) \braces{Z_t \partial_\aa \f{1}{Z_{,\aa}}}
\\& =2 D_\aa Z_t  - (I+\HH) D_\aa Z_t + (I-\HH) \braces{Z_t \partial_\aa \f{1}{Z_{,\aa}}}. 
\end{aligned}
\end{equation}
Now we use  \eqref{eq:242}, \eqref{eq:474}, and the fact that $\HH$ is purely imaginary to rewrite the second and third terms on the RHS above into commutators, and use \eqref{eq:32} to control 
\begin{equation}
  \label{eq:170}
 \nm{(I+\HH) D_\aa Z_t}_{L^\infty}  = \nm{\bracket{\f{1}{{Z}_{,\aa}},\HH} {Z}_{t,\aa}}_{L^\infty} \lec \nm{\partial_\aa \f{1}{{Z}_{,\aa}}}_{L^2} \nm{{Z}_{t,\aa}}_{L^2},
\end{equation}
\begin{equation}
  \label{eq:151}
  \nm{(I-\HH) \braces{Z_t \partial_\aa \f{1}{Z_{,\aa}}}}_{L^\infty} =  \nm{\bracket{Z_t,\HH} \partial_\aa \f{1}{Z_{,\aa}}}_{L^\infty} \lec \nm{Z_{t,\aa}}_{L^2} \nm{\partial_\aa \f{1}{Z_{,\aa}}}_{L^2}.
\end{equation}
Therefore
\begin{equation}
  \label{eq:513}
  \nm{\partial_\aa (I-\HH) \f{Z_t}{Z_{,\aa}}}_{L^\infty}   \lec \nm{D_\aa Z_t}_{L^\infty} + \nm{\partial_\aa \f{1}{Z_{,\aa}}}_{L^2} \nm{{Z}_{t,\aa}}_{L^2}\lec C(E).
  \end{equation}
  Observe that  
  \begin{equation}
  \label{eq:452}
  \f{h_{t\a}}{h_\a} \circ h\i = \partial_\aa (h_t \circ h\i):=\partial_\aa \frak b,
\end{equation}
so by \eqref{eq:246a} and \eqref{eq:513},
\begin{equation}
  \label{eq:451}
  \nm{\partial_\aa\frak b}_{L^\infty}=\nm{\f{h_{t\a}}{h_\a}}_{L^\infty}\lec C(E).
\end{equation}

\subsection{Controlling \texorpdfstring{$\nm{\f{(A_1 \circ h)_t}{(A_1 \circ h)}}_{L^\infty}$}{|| A1t/A1 || L infinity}}\label{sec:41}
Recall that $ A_1 \circ h = \f{\af \abs{z_\a}^2}{h_\a}$ \eqref{eq:209}.
Therefore,
\begin{equation}
  \label{eq:215}
  \f{\f{d}{dt} (A_1 \circ h)}{A_1 \circ h} = \f{\af_t}{\af} - \f{h_{t\a}}{h_\a} + 2 \Re D_\a z_t.
\end{equation}
We have controlled each of the terms on the RHS in $L^\infty$ in the previous sections. We conclude that
\begin{equation}
  \label{eq:455}
  \begin{aligned}
    \nm{\f{\f{d}{dt}(A_1 \circ h)}{A_1 \circ h}}_{L^\infty} & \lec \nm{\f{\af_t}{\af}}_{L^\infty} + \nm{\f{h_{t\a}}{h_\a}}_{L^\infty} + \nm{D_\a z_t}_{L^\infty}\lec C(E).
\end{aligned}
\end{equation}

\subsection{Controlling \texorpdfstring{$\nm{D_\aa \f{1}{Z_{,\aa}}}_{L^\infty}$}{|| Da' 1/Z,a' || L infinity} and related quantities
} \label{sec:46}
Recall from \eqref{eq:251} that
$ \f{1}{Z_{,\aa}} = i \f{\bar{Z}_{tt} - i}{A_1}$.
Therefore,
\begin{equation}
\label{eq:1015}
D_\aa \f{1}{Z_{,\aa}}= i \f{D_\aa\bar{Z}_{tt} }{A_1}-i\f{\bar{Z}_{tt} - i}{A_1^2}{D_\aa A_1}=i \f{D_\aa\bar{Z}_{tt} }{A_1}+ \f{(\bar{Z}_{tt} - i)^2}{A_1^3}{\partial_\aa A_1}.
\end{equation}
 Because we can control the first term on the RHS by $\nm{D_\aa Z_{tt}}_{L^\infty}$, it suffices to focus on the second term.  
 We start from \eqref{eq:254}, and then use a similar idea to \eqref{eq:255}:
\begin{equation}
  \label{eq:1014}
  \begin{aligned}
      \abs{\bar{Z}_{tt} - i}^2 \partial_\aa A_1 & = \Im \f{1}{\hdenomconst} \int \f{\pi}{2} \f{(Z_t(\aa) - Z_t(\bb))}{\sin^2(\f{\pi}{2}(\aa - \bb))} (\abs{\bar{Z}_{tt}(\aa)-i}^2 - \abs{\bar Z_{tt}(\bb)-i}^2) \bar{Z}_{t,\bb}(\bb) d\bb
\\&+ \Im \f{1}{\hdenomconst} \int \f{\pi}{2} \f{(Z_t(\aa) - Z_t(\bb))}{\sin^2(\f{\pi}{2}(\aa - \bb))}\abs{\bar{Z}_{tt}(\bb) - i}^2 \bar{Z}_{t,\bb}(\bb) d\bb
\\ & = I + II.
\end{aligned}
\end{equation}
To control $\nm{I}_{L^\infty}$, we use the mean value theorem and the periodicity of $Z_{tt}$  to estimate 
\begin{equation}
  \label{eq:7}
  \abs{\f{\abs{\bar{Z}_{tt}(\aa)-i}^2 - \abs{\bar Z_{tt}(\bb)-i}^2} {\sin(\f{\pi}{2}(\aa - \bb))}  }\lec \nm{(\bar{Z}_{tt}-i)\partial_\aa \bar{Z}_{tt} }_{L^\infty}\lec \|A_1\|_{L^\infty} \nm{D_\aa \bar Z_{tt}}_{L^\infty}.
\end{equation}
From Cauchy-Schwarz inequality and Hardy's inequality \eqref{eq:77}, we get
\begin{equation}\label{eq:1016}
\nm{I}_{L^\infty}\lec \nm{Z_{t,\aa}}_{L^2}^2\|A_1\|_{L^\infty} \nm{D_\aa \bar Z_{tt}}_{L^\infty}.
\end{equation}
For $II$, 
observe that $\abs{\bar{Z}_{tt}(\aa) - i}^2 \bar{Z}_{t,\aa}(\aa)=\f{A_1^2}{\bar Z_{,\aa}}D_\aa \bar Z_t$
and $\bound{  \f{A_1^2}{\bar Z_{,\aa}}D_\aa \bar Z_t}=0$. 
We integrate by parts as in \eqref{eq:98}:
\begin{equation}\label{eq:1017}
\begin{split}
 II=\Im\f{1}{\hdenomconst} \int \f{\pi}{2} \f{(Z_t(\aa) - Z_t(\bb))}{\sin^2(\f{\pi}{2}(\aa - \bb))}\abs{\bar{Z}_{tt}(\bb) - i}^2 \bar{Z}_{t,\bb}(\bb) d\bb \\=\Im\mathbb H \paren{Z_{t,\aa} \f{A_1^2}{\bar Z_{,\aa}}D_\aa \bar Z_t}-\Im[Z_t,\mathbb H] \partial_\aa \paren{\f{A_1^2}{\bar Z_{,\aa}}D_\aa \bar Z_t}.
 \end{split}
\end{equation}
We estimate the second term by \eqref{eq:32}:
\begin{equation}\label{eq:1019}
\nm{[Z_t,\mathbb H] \partial_\aa \paren{\f{A_1^2}{\bar Z_{,\aa}}D_\aa \bar Z_t}}_{L^\infty}\lec 
\|Z_{t,\aa}\|_{L^2}\nm{ \partial_\aa \paren{\f{A_1^2}{\bar Z_{,\aa}}D_\aa \bar Z_t}}_{L^2}.
\end{equation}
We expand the first term, using the conjugate of \eqref{eq:474},  noticing that $\Im \paren{Z_{t,\aa} \f{A_1^2}{\bar Z_{,\aa}}D_\aa \bar Z_t}=0$:
\begin{equation}\label{eq:1021}
\Im \mathbb H \paren{Z_{t,\aa} \f{A_1^2}{\bar Z_{,\aa}}D_\aa \bar Z_t}=-\Im \paren{Z_{t,\aa} \f{A_1^2}{\bar Z_{,\aa}}D_\aa \bar Z_t+\bracket{\f{A_1^2}{\bar Z_{,\aa}}D_\aa \bar Z_t,\mathbb H}Z_{t,\aa}}=-\Im\bracket{\f{A_1^2}{\bar Z_{,\aa}}D_\aa \bar Z_t,\mathbb H}Z_{t,\aa}.
\end{equation}
We estimate the RHS by  \eqref{eq:32}:
\begin{equation}\label{eq:1022}
\nm{\bracket{\f{A_1^2}{\bar Z_{,\aa}}D_\aa \bar Z_t,\mathbb H}Z_{t,\aa}}_{L^\infty}\lec \|Z_{t,\aa}\|_{L^2}\nm{ \partial_\aa \paren{\f{A_1^2}{\bar Z_{,\aa}}D_\aa \bar Z_t}}_{L^2}.
\end{equation}
Now,
\begin{equation}\label{eq:1020}
\begin{aligned}
\nm{ \partial_\aa \paren{\f{A_1^2}{\bar Z_{,\aa}}D_\aa \bar Z_t}}_{L^2}
& \lec \nm{A_1}_{L^\infty}^2\nm{D^2_\aa \bar Z_t}_{L^2}\\&+\nm{A_1}_{L^\infty}^2\nm{\partial_\aa\f{1}{Z_{,\aa}}}_{L^2}\nm{D_\aa \bar Z_t}_{L^\infty}+\nm{ \f{A_1}{\bar Z_{,\aa}}\partial_\aa A_1  }_{L^2}\nm{D_\aa \bar Z_t}_{L^\infty}.
\end{aligned}
\end{equation}

Combining the above calculations and using \eqref{eq:1018}, \eqref{eq:425}, the estimates in \S \ref{sec:quants}, and the fact $A_1\ge 1$, we conclude
\begin{equation}\label{eq:1023}
\begin{aligned}
\nm{D_\aa \f{1}{Z_{,\aa}}}_{L^\infty}&\le \nm{D_\aa\bar{Z}_{tt} }_{L^\infty}+\nm{(\bar{Z}_{tt} - i)^2\partial_\aa A_1 }_{L^\infty}\\&\lec \nm{D_\aa\bar{Z}_{tt} }_{L^\infty}+\nm{Z_{t,\aa}}_{L^2}^2\|A_1\|_{L^\infty} \nm{D_\aa \bar Z_{tt}}_{L^\infty}+\|Z_{t,\aa}\|_{L^2}\nm{ \partial_\aa (\f{A_1^2}{\bar Z_{,\aa}}D_\aa \bar Z_t)}_{L^2}\\&\lec C(E).
\end{aligned}
\end{equation}
We  record here the estimate for two related quantities, which we will use in later sections:
\begin{equation}\label{eq:1024}
\nm{(\bar Z_{tt}-i)\partial_\aa \f{1}{Z_{,\aa}}}_{L^\infty}\le \nm{A_1}_{L^\infty}\nm{D_\aa \f{1}{Z_{,\aa}}}_{L^\infty}\lec C(E);
\end{equation}
and
\begin{equation}\label{eq:1024a}
\nm{\partial_\aa\f{ Z_{tt}+i}{Z_{,\aa}}}_{L^\infty}\le \nm{D_\aa Z_{tt}}_{L^\infty}+ \nm{(Z_{tt}+i)\partial_\aa \f{1}{Z_{,\aa}}}_{L^\infty}  \lec C(E).
\end{equation}

\section{Controlling \texorpdfstring{$\Re \int i \paren{\f{1}{\bar{z}_\a} \partial_\a \f{h_\a}{z_\a}}_t \abs{\th}^2 h_\a d\a$}{Re int i 1/za partial a (ha/za)t |theta|2 ha da}}\label{sec:98a}
In this section, we control from \eqref{eq:66} the term
\begin{equation}
  \label{eq:229}
  \begin{aligned}
    \Re \int i \paren{\f{1}{\bar{z}_\a} \partial_\a \f{h_\a}{z_\a}}_t \abs{\th}^2 h_\a d\a & = - \Re \int i \f{\bar{z}_{t\a}}{\bar{z}_\a} \paren{\f{1}{\bar{z}_\a} \partial_\a \f{h_\a}{z_\a}} \abs{\th}^2 h_\a d\a 
\\ & +  \Re \int i \paren{\f{1}{\bar{z}_\a} \partial_\a \partial_t \f{h_\a}{z_\a}} \abs{\th}^2 h_\a d\a
  \end{aligned}
\end{equation}
for $\th = D_\a^2 \bar{z}_t$.
We can control the first of these terms by
\begin{equation}
  \label{eq:169}
  \begin{aligned}
    \abs{-\Re \int i \f{\bar{z}_{t\a}}{\bar{z}_\a} \paren{\f{1}{\bar{z}_\a} \partial_\a \f{h_\a}{z_\a}} \abs{\th}^2 h_\a d\a} & \lec \nm{D_\a \bar{z}_t}_{L^\infty} \nm{D_\aa \f{1}{Z_{,\aa}}}_{L^\infty} \nm{D_\aa^2 \bar{Z}_t}_{L^2}^2.
  \end{aligned}
\end{equation}
Therefore, it suffices to focus on the second term on the RHS of \eqref{eq:229}. We expand it out:
\begin{equation}
  \label{eq:232}
  \begin{aligned}
    \Re \int i &\paren{\f{1}{\bar{z}_\a} \partial_\a \partial_t \f{h_\a}{z_\a}} \abs{\th}^2 h_\a d\a  = \Re \int i \paren{\f{1}{\bar{z}_\a} \partial_\a \paren{\f{h_\a}{z_\a} \paren{ \f{h_{t\a}}{h_\a} - \f{z_{t\a}}{z_\a}}}} \abs{\th}^2 h_\a d\a
\\ & = \Re \int i \paren{\f{1}{\bar{z}_\a} \partial_\a \f{h_\a}{z_\a}} \paren{ \f{h_{t\a}}{h_\a} - \f{z_{t\a}}{z_\a}} \abs{\th}^2 h_\a d\a 
 + \Re \int i \paren{ \f{h_\a}{|z_\a|^2} \partial_\a \paren{ \f{h_{t\a}}{h_\a} - \f{z_{t\a}}{z_\a}}} \abs{\th}^2 h_\a d\a.
  \end{aligned}
\end{equation}
We can estimate the first term on the RHS by
\begin{equation}
  \label{eq:237}
  \abs{\Re \int i \paren{\f{1}{\bar{z}_\a} \partial_\a \f{h_\a}{z_\a}} \paren{ \f{h_{t\a}}{h_\a} - \f{z_{t\a}}{z_\a}} \abs{\th}^2 h_\a d\a} \lec \nm{D_\aa \f{1}{Z_{,\aa}}}_{L^\infty} \paren{\nm{\f{h_{t\a}}{h_\a}}_{L^\infty} + \nm{D_\a z_t}_{L^\infty}} \nm{D_\aa^2 \bar{Z}_t}_{L^2}^2.
\end{equation}
Therefore, it suffices to focus on the second term on the RHS of \eqref{eq:232}. Observe that because $h$ is real-valued, 
\begin{equation}
  \label{eq:238}
  \Re \int i \paren{ \f{h_\a}{|z_\a|^2} \partial_\a \paren{ \f{h_{t\a}}{h_\a} - \f{z_{t\a}}{z_\a}}} \abs{\th}^2 h_\a d\a = - \Re \int i \paren{ \f{h_\a}{|z_\a|^2} \partial_\a \paren{\f{z_{t\a}}{z_\a}}} \abs{\th}^2 h_\a d\a.
\end{equation}
We now drop $\Re$ and the $i$, write $D_\a^2 \bar{z}_t = \th$, and switch to Riemann mapping variables. For consistency with the quantities we've controlled elsewhere, we will take a conjugate. We have
\begin{equation}
  \label{eq:292}
  \int  \paren{\f{1}{\abs{Z_{,\aa}}^2} \partial_\aa \paren{\f{\bar{Z}_{t,\aa}}{\bar{Z}_{,\aa}}}} \abs{D_\aa^2 \bar{Z}_t}^2 d\aa.
\end{equation}

We want to take advantage of the holomorphicity and antiholomorphicity of various of these factors.  To do this, 
we first use the identity
$$\f{1}{\abs{Z_{,\aa}}^2} \partial_\aa \paren{\f{\bar{Z}_{t,\aa}}{\bar{Z}_{,\aa}}}=\f{1}{\bar{Z}_{,\aa}^2} \partial_\aa \paren{\f {\bar{Z}_{t,\aa}}{Z_{,\aa}}}+\bar{D_\aa Z_t}\paren{D_\aa \frac1{\bar{Z}_{,\aa}}-\bar{D_\aa} \frac1{{Z}_{,\aa}}}$$
to replace the first factor 
to make it closer to holomorphic:
\begin{equation}
  \label{eq:294}
  \begin{aligned}
     & \int  \paren{\f{1}{\abs{Z_{,\aa}}^2} \partial_\aa \paren{\f{\bar{Z}_{t,\aa}}{\bar{Z}_{,\aa}}}} \abs{D_\aa^2 \bar{Z}_t}^2 d\aa 
\\ & =\int  \paren{\f{1}{\bar{Z}_{,\aa}^2} \partial_\aa D_\aa \bar{Z}_t} \abs{D_\aa^2 \bar{Z}_t}^2 d\aa 
 +\int    \bar{D_\aa Z_t}\paren{D_\aa \frac1{\bar{Z}_{,\aa}}-\bar{D_\aa} \frac1{{Z}_{,\aa}}}  \abs{D_\aa^2 \bar{Z}_t}^2 d\aa.
\end{aligned}
\end{equation}
 We can estimate the second term by
\begin{equation}
  \label{eq:304}
      \abs{\int   \bar{D_\aa Z_t}\paren{D_\aa \frac1{\bar{Z}_{,\aa}}-\bar{D_\aa} \frac1{{Z}_{,\aa}}}\abs{D_\aa^2 \bar{Z}_t}^2 d\aa} 
 \lec \nm{D_\aa Z_{t}}_{L^\infty} \nm{D_\aa \frac1{Z_{,\aa}}}_{L^\infty} \nm{D_\aa^2 \bar{Z}_t}_{L^2}^2\lec C(E).
\end{equation}
It therefore remains only to control the first term on the RHS of \eqref{eq:294}.  Now we take advantage of holomorphicity. We rewrite this as
\begin{equation}
  \label{eq:306}
     \int  \paren{\f{1}{\bar{Z}_{,\aa}^2} \partial_\aa D_\aa \bar{Z}_t} \abs{D_\aa^2 \bar{Z}_t}^2 d\aa 
      = \int \paren{(\P_A + \P_H) \paren{\f{1}{\bar{Z}_{,\aa}} \partial_\aa D_\aa \bar{Z}_t} }\paren{\bar{\f{1}{Z_{,\aa}} D_\aa^2 \bar{Z}_t}} \P_H D_\aa^2\bar{Z}_t d\aa,
\end{equation}
where we have used \eqref{eq:315} to insert $\P_H$ in front of the $D_\aa^2 \bar{Z}_t$ and decomposed the first factor into the antiholomorphic and holomorphic projections. Now we use the adjoint property \eqref{eq:471} to turn the $\P_H$ into a $\P_A$ on the opposite factors, and control using Cauchy-Schwarz inequality:
\begin{equation}
  \label{eq:308}
  \begin{aligned}
     &\abs{\int \paren{(\P_A + \P_H) \paren{\f{1}{\bar{Z}_{,\aa}} \partial_\aa D_\aa \bar{Z}_t} }\paren{\bar{\f{1}{Z_{,\aa}} D_\aa^2 \bar{Z}_t}}  \P_H D_\aa^2\bar{Z}_t d\aa} 
\\ & \lec \nm{ \P_A\braces{ \paren{(\P_A + \P_H)\paren{ \f{1}{\bar{Z}_{,\aa}} \partial_\aa D_\aa \bar{Z}_t} }\paren{\bar{\f{1}{Z_{,\aa}} D_\aa^2 \bar{Z}_t}}}}_{L^2} \nm{D_\aa^2 \bar{Z}_t}_{L^2}.
  \end{aligned}
\end{equation}
It now remains only to control this first factor.

First we consider the term with the $\P_H$. In this case, we can rewrite this as a commutator:
\begin{equation}
  \label{eq:353}
  \begin{aligned}
   \P_A\braces{ \paren{\P_H \paren{\f{1}{\bar{Z}_{,\aa}} \partial_\aa D_\aa \bar{Z}_t}} \paren{\bar{\f{1}{Z_{,\aa}} D_\aa^2 \bar{Z}_t}}} & = \frac12 \bracket{\bar{\f{1}{Z_{,\aa}} D_\aa^2 \bar{Z}_t},\HH}\P_H \paren{\f{1}{\bar{Z}_{,\aa}} \partial_\aa D_\aa \bar{Z}_t }
\\ & + \f{1}{4} \bar{\f{1}{Z_{,\aa}} D_\aa^2 \bar{Z}_t} \paren{\avg \f{1}{\bar{Z}_{,\aa}} \partial_\aa D_\aa \bar{Z}_t}.
  \end{aligned}
\end{equation}
Here, the mean term appears because of \eqref{eq:159}.  We now use commutator estimate \eqref{eq:96} for the first term and H\"older's inequality for the second term, to conclude that
\begin{equation}
  \label{eq:355}
  \begin{split}
   & \nm{\P_A\braces{ \paren{\P_H \paren{\f{1}{\bar{Z}_{,\aa}} \partial_\aa D_\aa \bar{Z}_t}} \paren{\bar{\f{1}{Z_{,\aa}} D_\aa^2 \bar{Z}_t}}}}_{L^2} \\& \lec  \nm{\bar{\f{1}{Z_{,\aa}} D_\aa^2 \bar{Z}_t}}_{\dot{H}^{1/2}} \nm{\P_H \paren{\f{1}{\bar{Z}_{,\aa}} \partial_\aa D_\aa \bar{Z}_t}}_{L^2} 
 + \nm{\f{1}{Z_{,\aa}} D_\aa^2 \bar{Z}_t}_{L^2} \nm{\f{1}{\bar{Z}_{,\aa}} \partial_\aa D_\aa \bar{Z}_t}_{L^1}
\\ & \lec  \nm{D_\aa^2 \bar{Z}_t}_{L^2} \paren{ \nm{\f{1}{Z_{,\aa}} D_\aa^2 \bar{Z}_t}_{\dot{H}^{1/2}} + \nm{\f{1}{Z_{,\aa}}}_{L^\infty} \nm{D_\aa^2 \bar{Z}_t}_{L^2}}.
\end{split}
\end{equation}

Finally, we consider the $\P_A$ term in the first factor on the RHS of \eqref{eq:308}. By the $L^2$ boundedness of $\P_A$,  
it suffices to control
\begin{equation}
  \label{eq:360}
  \begin{aligned}
    &\nm{\f{1}{Z_{,\aa}} D_\aa^2 \bar{Z}_t(I-\HH)\paren{ \f{1}{\bar{Z}_{,\aa}} \partial_\aa D_\aa \bar{Z}_t}}_{L^2} \\ & \lec \nm{ D_\aa^2 \bar{Z}_t \bracket{\f{1}{Z_{,\aa}},\HH} \f{1}{\bar{Z}_{,\aa}} \partial_\aa D_\aa \bar{Z}_t}_{L^2} 
 + \nm{D_\aa^2 \bar{Z}_t (I-\HH) \f{1}{\bar{Z}_{,\aa}} D_\aa^2 \bar{Z}_t}_{L^2}
\\ & \lec \nm{D_\aa^2 \bar{Z}_t}_{L^2} \nm{ \bracket{\f{1}{Z_{,\aa}},\HH} \f{1}{\bar{Z}_{,\aa}} \partial_\aa D_\aa \bar{Z}_t}_{L^\infty}
 + \nm{D_\aa^2 \bar{Z}_t}_{L^2} \nm{ \bracket{\f{1}{\bar{Z}_{,\aa}},\HH} D_\aa^2 \bar{Z}_t}_{L^\infty}
\\ & \lec \nm{D_\aa^2 \bar{Z}_t}_{L^2}^2 \nm{\partial_\aa \f{1}{Z_{,\aa}}}_{L^2},
\end{aligned}
\end{equation}
where we've used \eqref{eq:315} to get the second commutator and used commutator estimate \eqref{eq:32}.

We now combine our estimates, concluding that
\begin{equation}
  \label{eq:368}
  \begin{aligned}
    \Re \int i &\paren{\f{1}{\bar{z}_\a} \partial_\a \f{h_\a}{z_\a}}_t \abs{\th}^2 h_\a d\a \ \lec \eqref{eq:169} + \eqref{eq:232}
\\ & \lec \eqref{eq:169}  + \eqref{eq:237} + \eqref{eq:304} 
 + \nm{D_\aa^2 \bar{Z}_t}_{L^2} \cdot (\eqref{eq:355} + \eqref{eq:360})
\\ & \lec C(E).
  \end{aligned}
\end{equation}

\section{Controlling \texorpdfstring{$\Re  \int i \paren{\f{h_\a^2}{\abs{z_\a}^2}}_t  \th_\a \bar{\th} d\a$}{Re int i (ha2/za2)t theta a bar theta da}}
\label{sec:48}
We now show that we can control the following term from the RHS of \eqref{eq:66}:
\begin{equation}
  \label{eq:78}
  \Re  \int i \paren{\f{h_\a^2}{\abs{z_\a}^2}}_t  \th_\a \bar{\th} d\a = \Re \int i \paren{2 \f{h_{t\a}}{h_\a} - 2 \Re D_\a z_t} \f{h_\a^2}{\abs{z_{\a}}^2} \th_\a \bar{\th} d\a.
\end{equation}
Here, all results will be expressed in terms of general energy $E_{a,\th}$ for $\th$ satisfying $(I-\HH) (\th \circ h\i) = 0$ and $\bound{\th} = 0$, rather than specifying $\th = D_\a^2 \bar{z}_t$.\footnote{$D_\a^2\bar z_t$ satisfies $(I-\HH)D_\a^2 \bar{z}_t\circ h^{-1}=0, \bound{D_\a^2 \bar{z}_t}=0$, see \eqref{eq:315}.}

We begin by rewriting this as
\begin{equation}
  \label{eq:80}
  \begin{aligned}
     \Re \int i \paren{2 \f{h_{t\a}}{h_\a} - 2 \Re D_\a z_t} \f{h_\a^2}{\abs{z_{\a}}^2} \th_\a \bar{\th} d\a & = \Re \int 2 i\paren{\f{h_{t\a}}{h_\a} - \Re D_\a z_t}\paren{\partial_\a \paren{\th \f{h_\a}{z_\a}}} \bar{\th} \f{h_\a}{\bar{z}_\a} d\a
\\ & - \Re \int 2i \paren{\f{h_{t\a}}{h_\a} - \Re D_\a z_t} \paren{\f{h_\a}{\bar{z}_\a} \partial_\a \f{h_\a}{z_\a}} \th \bar{\th} d\a 
\\ & = I + II.
\end{aligned}
\end{equation}

II is easy to control, via H\"older's inequality and change of variables to Riemann mapping variables:
\begin{equation}
  \label{eq:82}
  \begin{aligned}
    \abs{II} & \lec \paren{\nm{\f{h_{t\a}}{h_\a}}_{L^\infty} + \nm{D_\a z_t}_{L^\infty}}\nm{D_\aa \f{1}{Z_{,\aa}}}_{L^\infty}  \nm{A_1}_{L^\infty} E_{a,\th}.
  \end{aligned}
\end{equation}
Therefore, we can focus on $I$ from \eqref{eq:80}. 
We introduce the following notations:
\begin{equation}
  \label{eq:83}
  \psi := \paren{\f{h_\a}{z_\a} \th} \circ h\i; \quad \Th := \th \circ h\i; \quad B := \paren{\f{h_{t\a}}{h_\a} - \Re D_\a z_t} \circ h\i=(h_t\circ h^{-1})_\aa-\Re D_\aa Z_t.
\end{equation}
We know $(I-\HH) \Th = 0$, $\bound{\Th}=0$, and
\begin{equation}
 \label{eq:555}
  \bound{B} = 0, \quad  \bound{\psi} = 0,\\
\end{equation}
\begin{equation}
  \label{eq:558}
  \quad (I-\HH) \psi = 0,
 \end{equation}
\begin{equation}
  \label{eq:565}
 \nm{\psi}_{\dot{H}^{1/2}}+ \nm{A_1}_{L^\infty}^{-1/2}\nm{\Th}_{L^2} \lec E_{a,\th}^{1/2},
\end{equation}
where 
\eqref{eq:555} follows from the assumption of Theorem~\ref{maintheorem} and \eqref{assumption2}; \eqref{eq:558} follows from \eqref{eq:1026a},  $(I-\HH) \Th = 0$ and  principle no.2~in \S \ref{sec:20} (and for $\th=D_\a^2\bar z_t$ specifically from \eqref{eq:456});
and \eqref{eq:565} is immediate from the definition of $E_{a, \theta}$ and change of variables. Upon changing variables, we can write $I= \Re \int 2i B (\partial_\aa \psi) \bar{\psi} d\aa $.

{\bf Step 1.} {\it{Green's identity.}} 
We now show that we can control $I$ from \eqref{eq:80}. The main idea is to use Green's identity 
to move the derivative from $\psi := \paren{\f{h_\a}{z_\a} \th} \circ h\i$ onto $B$. We note that $i\partial_\aa \psi = i \partial_\aa \HH \psi$ by \eqref{eq:558}, and that the operator  $i\partial_\aa \HH  = \nabla_n $, here $\nabla_n$ is 
the Dirichlet-Neumann operator.\footnote{Recall that the Dirichlet-Neumann operator is defined by $\nabla_n f :=\nabla_n f^\hbar$,  the outward-facing normal derivative of $f^\hbar$, where $f^\hbar$ is the extension of $f$ that is harmonic and periodic in $P^-$ and tending to a constant at infinity.  For $f$ real-valued, we can derive this  by noting that $(I+\HH)f$ is holomorphic, so $i\partial_\aa(I+\HH)f=\nabla_n (I+\HH)f$. Taking real parts gives the identity.} 
Letting  $\psi^\hbar$ and $B^\hbar$ be the (periodic) harmonic extension of $\psi$ and $B$ to $P^-$ respectively, we have 
\begin{equation}
  \label{eq:86}
  I 
   = \Re \int 2i B (\partial_\aa \psi) \bar{\psi} d\aa  = \Re \int 2 B (\nabla_n \psi) \bar{\psi}d S
= \int B \nabla_n (\abs{\psi^\hbar}^2) dS. 
\end{equation}
By Green's identity,\footnote{Here, to justify Green's identity, we can map (biholomorphically) the space $P^- $ to the unit disk minus the slit, and then use the periodicity of all of the functions involved to consider the harmonic extensions of these functions to the whole unit disk.}
\begin{equation}
  \label{eq:550}
  \begin{aligned}
    \int B \nabla_n (\abs{\psi^\hbar}^2) dS & = \int (\nabla_n B) \abs{\psi}^2 dS + \int_{P^-} B^\hbar \Delta (\abs{\psi^\hbar}^2) dV
\\ & = I_1 + I_2.
  \end{aligned}
\end{equation}
We control the second term, $I_2$, by
\begin{equation}
  \label{eq:87}
  \begin{aligned}
       \abs{I_2} & = \abs{\int_{P^-} B^\hbar \Delta (\abs{\psi^\hbar}^2) dV} 
 =   2\abs{ \int_{P^-} B^\hbar \abs{\nabla \psi^\hbar}^2 dV}
\\ &  \le 2\nm{B^\hbar}_{L^\infty} \int_{P^-} \abs{\nabla\psi^\hbar}^2 dV
 = \nm{B^\hbar}_{L^\infty}  \int_{P^-} \Delta (\abs{\psi^\hbar}^2)dV
\\ & = \nm{B}_{L^\infty} \int \nabla_n (\abs{\psi^\hbar}^2) dS
= 2\nm{B}_{L^\infty} \Re \int i (\partial_\aa \psi) \bar{\psi} d\aa,
\\ & =2 \nm{B}_{L^\infty} \nm{\psi}^2_{\dot{H}^{1/2}}\lec \paren{\nm{\f{h_{t\a}}{h_\a}}_{L^\infty} + \nm{D_\aa Z_t}_{L^\infty}}\nm{\psi}^2_{\dot{H}^{1/2}}
     \end{aligned}
   \end{equation}
by the maximum principle and another application of Green's identity.

{\bf Step 2.} {\it Controlling $I_1$.} 
We are left from Step 1 with controlling 
\begin{equation}
  \label{eq:89}
  \begin{aligned}
   I_1 & =  \int (\nabla_n B) \abs{\psi}^2 dS
   = \Re \int (i  \partial_\aa \HH    B) \abs{\psi}^2 d\aa
    \\ & = \Re \int (i  \HH\partial_\aa    B) \abs{\psi}^2 d\aa
      = \Re \int \f{1}{Z_{,\aa}} (i\HH \partial_\aa B) \Th \bar{\psi} d\aa.
\end{aligned}
\end{equation}

We commute the $\f{1}{Z_{,\aa}}$ factor inside the $\HH$, and then apply the adjoint property \eqref{eq:471}:
\begin{equation}
  \label{eq:90}
  \begin{aligned}
    I_{1} & = \Re \int i\paren{\bracket{\f{1}{Z_{,\aa}}, \HH}\partial_\aa B} \Th \bar{\psi} d\aa + \Re \int i \paren{\HH \paren{\f{1}{Z_{,\aa}} \partial_\aa B}} \Th \bar{\psi} d\aa 
\\ & = \Re \int i\paren{\bracket{\f{1}{Z_{,\aa}}, \HH}\partial_\aa B} \Th \bar{\psi} d\aa - \Re \int i \paren{\f{1}{Z_{,\aa}} \partial_\aa B} \HH \paren{\Th \bar{\psi}} d\aa 
\\ & = \Re \int i\paren{\bracket{\f{1}{Z_{,\aa}}, \HH}\partial_\aa B} \Th \bar{\psi} d\aa + \Re \int i \paren{\f{1}{Z_{,\aa}} \partial_\aa B} \bracket{\bar{\psi},\HH} \Th d\aa - \Re \int i \paren{\f{1}{Z_{,\aa}} \partial_\aa B}  \bar{\psi} \HH \Th d\aa 
\\ & = I_{11} + I_{12} + I_{13}.
\end{aligned}
\end{equation}

Observe that because $\HH \Th = \Th$,
\begin{equation}
  \label{eq:559}
    I_{13}  =  -\Re \int i \paren{\f{1}{Z_{,\aa}} \partial_\aa B}  \bar{\psi} \Th d\aa
      = -\Re \int i \paren{\partial_\aa B} \abs{\psi}^2 d\aa
     = 0,
\end{equation}
since $B \in \R$. It remains to control $I_{11}$ and $I_{12}$.

We use Cauchy-Schwarz inequality and then the  $\dot{H}^{1/2} \times L^2$ commutator estimate \eqref{eq:96} to control $I_{12}$:
\begin{equation}
  \label{eq:92}
    \abs{ I_{12}} \le \nm{D_\aa B}_{L^2} \nm{[\bar{\psi},\HH] \Th}_{L^2}
     \lec \nm{D_\aa B}_{L^2}  \nm{\bar{\psi}}_{\dot{H}^{1/2}} \nm{\Th}_{L^2}.
\end{equation}
We have controlled $\nm{\bar{\psi}}_{\dot{H}^{1/2}}$ and $\nm{\Th}_{L^2}$ at \eqref{eq:565}, and we will control $\nm{D_\aa B}_{L^2}$ by \eqref{eq:546} in Step 3 below.  

It remains to control $I_{11}$ from \eqref{eq:90}.  Here we use Proposition \ref{prop:iden}, identity \eqref{eq:19}.  Because $(I-\HH) \f{1}{Z_{,\aa}} = 
1$ \eqref{eq:1026a} and $\avg \partial_\aa B=0$ by \eqref{eq:555}, we can rewrite
\begin{equation}
  \label{eq:562}
    \bracket{\f{1}{Z_{,\aa}}, \HH}\partial_\aa B =
    \P_H  \paren{\bracket{\f{1}{Z_{,\aa}}, \HH}\partial_\aa B} - \f{1}{2}\avg D_\aa B.
\end{equation}
We plug \eqref{eq:562} into  $I_{11}$, and then use adjoint property \eqref{eq:471}:
\begin{equation}
  \label{eq:563}
  \begin{aligned}
    I_{11} 
    & = \Re \int i \braces{\P_H  \paren{\bracket{\f{1}{Z_{,\aa}}, \HH}\partial_\aa B}}\Th \bar{\psi} d\aa - \Re \braces{ \paren{ \f{1}{2}\avg D_\aa B} \int i \Th \bar{\psi} d\aa}
\\ & = \Re \int i \paren{\bracket{\f{1}{Z_{,\aa}}, \HH}\partial_\aa B}\P_A \paren{\Th \bar{\psi}} d\aa - \Re\braces{ \paren{ \f{1}{2}\avg D_\aa B} \int i \Th \bar{\psi} d\aa}.
  \end{aligned}
\end{equation}
To control the first term, we use Cauchy-Schwarz inequality, and then control the first factor with the $L^2 \times L^\infty$ estimate 
\eqref{eq:228} and control the second factor by rewriting it as a commutator by \eqref{eq:558} and then using the $\dot{H}^{1/2} \times L^2$ estimate
\eqref{eq:96}. We use \eqref{eq:251} to rewrite $\f{1}{\bar{Z}_{,\aa}} = -i\f{Z_{tt} + i}{A_1}$ in the second term.
\begin{equation}
  \label{eq:564}
  \begin{aligned}
    \abs{I_{11}} & \le \nm{\bracket{\f{1}{Z_{,\aa}}, \HH}\partial_\aa B}_{L^2} \nm{\P_A \paren{\Th \bar{\psi}}}_{L^2}+\abs{\paren{ \f{1}{2}\avg D_\aa B} \int i \Th \bar{\psi} d\aa}
\\ & \le \nm{\bracket{\f{1}{Z_{,\aa}}, \HH}\partial_\aa B}_{L^2} \nm{\f{1}{2}\bracket{\bar{\psi},\HH} \Th}_{L^2}
+ \nm{D_\aa B}_{L^2} \int \abs{\f{Z_{tt} + i}{A_1} \abs{\Th}^2} d\aa\\ & \lec \nm{\partial_\aa \f{1}{Z_{,\aa}}}_{L^2}\nm{B}_{L^\infty} \nm{\bar{\psi}}_{\dot{H}^{1/2}} \nm{\Th}_{L^2}+\nm{D_\aa B}_{L^2} \nm{Z_{tt} + i}_{L^\infty} E_{a,\th}.
  \end{aligned}
\end{equation}
We have controlled all the quantities on the last line, except for $\nm{D_\aa B}_{L^2}$.

{\bf Step 3.} {\it Controlling $\nm{D_\aa B}_{L^2}$.} 
We must control $\nm{D_\aa B}_{L^2}$, where $B$ is as defined by \eqref{eq:83}. By \eqref{eq:246a} and \eqref{eq:127}, rewriting the second and third terms in \eqref{eq:127} as commutators, we have
\begin{equation}
  \label{eq:530}
  B=  (h_t \circ h\i)_\aa - \Re D_\aa Z_t = 
   \Re D_\aa Z_t + \Re \braces{ \bracket{\f{1}{{Z}_{,\aa}},\HH} {Z}_{t,\aa} + [Z_t,\HH] \partial_\aa \f{1}{Z_{,\aa}}}.
\end{equation}
Therefore, noting that $\abs{\partial_\aa \Re f} \le \abs{\partial_\aa f}$ and so $\abs{D_\aa \Re f} \le \abs{D_\aa f}$,
\begin{equation}
  \label{eq:541}
  \nm{D_\aa B}_{L^2} \le 
  \nm{D_\aa^2 Z_t}_{L^2} + \nm{D_\aa   \bracket{\f{1}{{Z}_{,\aa}},\HH}  {Z}_{t,\aa}}_{L^2} + \nm{D_\aa[ Z_t, \HH]\partial_\aa \f{1}{Z_{,\aa}}}_{L^2}.
\end{equation}
We've controlled $\nm{D_\aa^2 Z_t}_{L^2}$, so it suffices to focus on the second and third terms. In what follows, we work on $D_\aa [f,\HH] \partial_\aa g$ for general functions $f$ and $g$ satisfying $\bound{f}=\bound{g}=0$. Once we have an appropriate estimate, we will apply it to $f=\f{1}{{Z}_{,\aa}}$, $g={Z}_{t}$ for the second term, and $f= Z_t$, $g=\f{1}{Z_{,\aa}}$ for the third term.

We know
\begin{equation}
  \label{eq:1081}
  \begin{aligned}
      D_\aa [f,\HH] \partial_\aa g
     & = \f{1}{Z_{,\aa}} \partial_\aa \f{1}{\hdenomconst} \int (f(\aa ) - f(\bb)) \cot(\f{\pi}{2}(\aa -\bb)) \partial_\bb g(\bb) d\bb
    \\ & = \f{1}{Z_{,\aa}} (\partial_\aa f) \HH \partial_\aa g - \f{1}{Z_{,\aa}} \f{1}{\hdenomconst} \int \f{\pi}{2} \f{f(\aa) - f(\bb)}{ \sin^2(\f{\pi}{2}(\aa - \bb))} \partial_\bb g(\bb)d\bb.
  \end{aligned}
\end{equation}
Via the boundedness of the Hilbert transform, we control the first of these terms by $\nm{D_\aa f}_{L^\infty} \nm{\partial_\aa g}_{L^2}$. Therefore, it suffices to focus on the second term.  We commute the $\f{1}{Z_{,\aa}}$ inside, getting
  \begin{equation}
    \label{eq:1091}
    -\f{\pi}{4i} \int \f{f(\aa) - f(\bb)}{\sin^2(\f{\pi}{2}(\aa - \bb))} D_\bb g(\bb) d\bb -\f{\pi}{4i} \int \f{(f(\aa) - f(\bb))}{\sin(\f{\pi}{2}(\aa - \bb))}\f{\paren{\f{1}{Z_{,\aa}}(\aa) - \f{1}{Z_{,\bb}}(\bb)}}{\sin(\f{\pi}{2}(\aa - \bb))} \partial_\bb g(\bb) d\bb.
  \end{equation}
We control the first term by \eqref{eq:258}:
\begin{equation}
  \label{eq:1101}
 \nm{ \f{\pi}{4i}\int \f{f(\aa) - f(\bb)}{\sin^2(\f{\pi}{2}(\aa - \bb))} D_\bb g(\bb) d\bb}_{L^2} \lec \nm{\partial_{\aa}f}_{L^2} \nm{D_\aa g}_{L^\infty}.
\end{equation}
We control the second term by \eqref{eq:113}:
\begin{equation}
  \label{eq:1111}
  \nm{\f{\pi}{4i}\int \f{(f(\aa) - f(\bb))}{ \sin(\f{\pi}{2}(\aa - \bb))}\f{\paren{\f{1}{Z_{,\aa}}(\aa) - \f{1}{Z_{,\bb}}(\bb)}}{\sin(\f{\pi}{2}(\aa - \bb))} \partial_\bb g(\bb) d\bb}_{L^2} \lec \nm{\partial_{\aa}f}_{L^2} \nm{\partial_\aa \f{1}{Z_{,\aa}}}_{L^2}\nm{\partial_{\aa}g}_{L^2}.
\end{equation}
We conclude that 
\begin{equation}
  \label{eq:5481}
  \begin{aligned}
    \nm{D_\aa [f, \HH] \partial_\aa g}_{L^2} & \lec \nm{D_\aa f}_{L^\infty} \nm{\partial_\aa g}_{L^2}+ \nm{D_\aa g}_{L^\infty} \nm{\partial_\aa f}_{L^2}+\nm{\partial_{\aa}f}_{L^2} \nm{\partial_\aa \f{1}{Z_{,\aa}}}_{L^2}\nm{\partial_{\aa}g}_{L^2}.
\end{aligned}
\end{equation}

We can conclude from \eqref{eq:541} and \eqref{eq:5481} that
\begin{equation}
  \label{eq:546}
    \nm{D_\aa B}_{L^2} 
 \lec \nm{D_\aa^2 Z_t}_{L^2} + \nm{\partial_\aa \f{1}{Z_{,\aa}}}_{L^2} \nm{D_\aa {Z}_t}_{L^\infty} + \nm{Z_{t,\aa}}_{L^2}\paren{ \nm{D_\aa \f{1}{Z_{,\aa}}}_{L^\infty} + \nm{\partial_\aa \f{1}{Z_{,\aa}}}_{L^2}^2}.
\end{equation}

{\bf Step 4.} {\it Conclusion.} 
We now combine our various estimates. We have
\begin{equation}
  \label{eq:552}
  \begin{aligned}
    \abs{ \Re  \int i \paren{\f{h_\a^2}{\abs{z_\a}^2}}_t  \th_\a \bar{\th} d\a} & \le \abs{I} + \abs{II} \le \abs{I_1} +\abs{I_2}+\eqref{eq:82}
    \\ &\le \abs{I_{11}} +\abs{I_{12}}  +\eqref{eq:87}+\eqref{eq:82}
    \\ & \le \eqref{eq:564}+\eqref{eq:92}+\eqref{eq:87}+\eqref{eq:82},
 \end{aligned}
\end{equation}
where we use \eqref{eq:546} to control $\nm{D_\aa B}_{L^2}$.

In particular, by specifying $\th=D_\a^2\bar z_t$, we have
\begin{equation}\label{eq:1027}
\abs{ \Re  \int i \paren{\f{h_\a^2}{\abs{z_\a}^2}}_t  \th_\a \bar{\th} d\a}\lec C(E).
\end{equation}

\section{Controlling $(I-\HH)((\partial_t+\frak b\partial_\aa)^2\Theta+i\AA \partial_\aa \Theta)$ }
We are left with controlling the $G_\theta$ terms in $E_a$ and $E_b$. Before we do so, we first study the quantity $(I-\HH)((\partial_t+\frak b\partial_\aa)^2\Theta+i\AA \partial_\aa \Theta)$ for a general $\Theta$ satisfying $(I-\HH)\Theta=0$. We have

\begin{lemma} Assume that $\Theta, (\partial_t+\frak b\partial_\aa)\Theta\in C^l([0, T], H^{k-l}(S^1))$ for $l=0,1$, $k\ge 2$, and $(I-\HH)\Theta=0$. And assume that the assumptions of Theorem~\ref{maintheorem} holds. Then
\begin{equation}\label{lemma-0}
(I-\HH) ((\partial_t+\frak b\partial_\aa)\Theta-Z_tD_\aa\Theta)=0;
\end{equation}
\begin{equation}\label{lemma-1}
(I-\HH)((\partial_t+\frak b\partial_\aa)^2\Theta+i\AA \partial_\aa \Theta)=[Z_t^2,\HH]D_\aa^2\Theta+2[ Z_t,\HH]D_\aa((\partial_t+\frak b\partial_\aa)\Theta-Z_tD_\aa\Theta)+2[Z_{tt},\HH]D_\aa\Theta;
\end{equation}
and
\begin{equation}\label{lemma-2}
\nm{(I-\HH)((\partial_t+\frak b\partial_\aa)^2\Theta+i\AA \partial_\aa \Theta)}_{L^2}\lec C(E)\paren{\nm{\Theta}_{L^2}+\nm{(\partial_t+\frak b\partial_\aa)\Theta}_{L^2}+\nm{\frac\Theta{Z_{,\aa}}}_{\dot H^{1/2}}}.
\end{equation}

\end{lemma}
\begin{proof}
By \eqref{eq:246a}, \begin{equation}\label{eq:2001}
\frak b=\mathbb P_A\paren{\frac{Z_t}{Z_{,\aa}}}+\mathbb P_H\paren{\frac{\bar Z_t}{\bar Z_{,\aa}}}=\frac{Z_t}{Z_{,\aa}}+\mathbb P_H\paren{\frac{\bar Z_t}{\bar Z_{,\aa}}-\frac{Z_t}{Z_{,\aa}}   },\end{equation}
so 
$$(\partial_t+\frak b\partial_\aa)\Theta=\partial_t\Theta+Z_tD_\aa\Theta+\mathbb P_H\paren{\frac{\bar Z_t}{\bar Z_{,\aa}}-\frac{Z_t}{Z_{,\aa}}   }\partial_\aa\Theta;$$
as a consequence of the dominated convergence theorem and  principles no.1 and no.2 in \S\ref{sec:20}, we have
\begin{equation}\label{eq:2002}
(I-\HH) ((\partial_t+\frak b\partial_\aa)\Theta-Z_tD_\aa\Theta)=0.
\end{equation}
Now since $\Psi:=(\partial_t+\frak b\partial_\aa)\Theta-Z_tD_\aa\Theta$ satisfies $(I-\HH)\Psi=0$, applying \eqref{eq:2002}  yields 
\begin{equation}\label{eq:2002a}
(I-\HH) ((\partial_t+\frak b\partial_\aa)\Psi-Z_tD_\aa\Psi)=0.
\end{equation}
 We compute
\begin{equation}\label{eq:2003}
\begin{aligned}
(\partial_t+\frak b\partial_\aa)\Psi-Z_tD_\aa\Psi&=(\partial_t+\frak b\partial_\aa)^2\Theta-(\partial_t+\frak b\partial_\aa)(Z_tD_\aa\Theta)-Z_tD_\aa (\partial_t+\frak b\partial_\aa)\Theta+Z_tD_\aa(Z_tD_\aa\Theta)\\&
=(\partial_t+\frak b\partial_\aa)^2\Theta-Z_{tt}D_\aa\Theta-2Z_tD_\aa( (\partial_t+\frak b\partial_\aa)\Theta-Z_tD_\aa\Theta)-Z_t^2D_\aa^2\Theta,
\end{aligned}
\end{equation}
and by \eqref{eq:179},
\begin{equation}\label{eq:2004}
i\AA \partial_\aa\Theta=(Z_{tt}+i)D_\aa\Theta;
\end{equation}
and we know by principles no.1 and no.2 in \S\ref{sec:20} that $(I-\HH)D_\aa\Theta=0$.  \eqref{eq:2002a}-\eqref{eq:2004} then gives
\begin{equation}\label{eq:2005}
(I-\HH)((\partial_t+\frak b\partial_\aa)^2\Theta+i\AA \partial_\aa \Theta)= (I-\HH)(2Z_{tt}D_\aa\Theta+2Z_tD_\aa( (\partial_t+\frak b\partial_\aa)\Theta-Z_tD_\aa\Theta)+Z_t^2D_\aa^2\Theta);
\end{equation}
using the holomorphicity of the factors to rewrite the right hand side of \eqref{eq:2005} as commutators yields \eqref{lemma-1}.

While \eqref{lemma-1} is sufficient to give us control of the $G_\theta$ term in $E_b$, we need \eqref{lemma-2}, which is the result of some further analysis of \eqref{eq:2005},
to control the $G_\theta$ term in $E_a$.\footnote{We can also use \eqref{lemma-2} to control the $G_\theta$ term in $E_b$.}

We begin with \eqref{eq:2005}, rewriting, using the product rule and the identity
$\frac{Z_t}{Z_{,\aa}}=\mathbb P_A \paren{ \frac{Z_t}{Z_{,\aa}}}+\mathbb P_H\paren{ \frac{Z_t}{Z_{,\aa}}},$
\begin{equation}\label{eq:2006}
Z_t^2D_\aa^2\Theta=\paren{ \mathbb P_A \paren{ \frac{Z_t}{Z_{,\aa}}}+\mathbb P_H\paren{ \frac{Z_t}{Z_{,\aa}}} }^2\partial_\aa^2\Theta+\frac{Z_t^2}{Z_{,\aa}}\paren{\partial_\aa\frac1{Z_{,\aa}}}\partial_\aa\Theta,
\end{equation}
and
\begin{equation}\label{eq:2007}
\begin{aligned}
2Z_tD_\aa&( (\partial_t+\frak b\partial_\aa)\Theta-Z_tD_\aa\Theta)=2\paren{ \mathbb P_A \paren{ \frac{Z_t}{Z_{,\aa}}}+\mathbb P_H\paren{ \frac{Z_t}{Z_{,\aa}}} }\partial_\aa((\partial_t+\frak b\partial_\aa)\Theta-Z_t D_\aa\Theta)\\&= 2\mathbb P_A \paren{ \frac{Z_t}{Z_{,\aa}}}\partial_\aa\paren{(\partial_t+\frak b\partial_\aa)\Theta-\paren{ \mathbb P_A \paren{ \frac{Z_t}{Z_{,\aa}}}+\mathbb P_H\paren{ \frac{Z_t}{Z_{,\aa}}} }\partial_\aa\Theta}\\&+2\mathbb P_H\paren{ \frac{Z_t}{Z_{,\aa}}} \partial_\aa((\partial_t+\frak b\partial_\aa)\Theta-Z_t D_\aa\Theta).
\end{aligned}
\end{equation}
Observe that, by \eqref{eq:2002} and principles no.1 and no.2 in \S\ref{sec:20},  the last quantity on the RHS of \eqref{eq:2007} is holomorphic with mean zero. Expanding further the right hand sides of \eqref{eq:2006} and \eqref{eq:2007} and sum up, observe that certain terms cancel out with others, and the quantity $\paren{ \mathbb P_H\paren{ \frac{Z_t}{Z_{,\aa}}} }^2\partial_\aa^2\Theta$ in \eqref{eq:2006} is holomorphic with mean zero, by 
principles no.1 and no.2 in \S\ref{sec:20}. We have
\begin{equation}\label{eq:2008}
\begin{aligned}
(I-\HH)&(2Z_tD_\aa( (\partial_t+\frak b\partial_\aa)\Theta-Z_tD_\aa\Theta)+Z_t^2D_\aa^2\Theta)\\&=(I-\HH)\paren{2\mathbb P_A \paren{ \frac{Z_t}{Z_{,\aa}}}\partial_\aa\paren{(\partial_t+\frak b\partial_\aa)\Theta- \mathbb P_A \paren{ \frac{Z_t}{Z_{,\aa}}}\partial_\aa\Theta}} \\&-(I-\HH)\paren{2\mathbb P_A \paren{ \frac{Z_t}{Z_{,\aa}}}\paren{\mathbb P_H\paren{ \frac{Z_{t,\aa}}{Z_{,\aa}}}+\mathbb P_H\paren{ {Z_{t}}\partial_\aa\frac1{Z_{,\aa}}}} \partial_\aa\Theta}
\\&+(I-\HH)\paren{\paren{ \mathbb P_A \paren{ \frac{Z_t}{Z_{,\aa}}} }^2\partial_\aa^2\Theta+\frac{Z_t^2}{Z_{,\aa}}\paren{\partial_\aa\frac1{Z_{,\aa}}}\partial_\aa\Theta    }.
\end{aligned}
\end{equation}
We analyze further the RHS of \eqref{eq:2008}. Observe that $(I-\HH)\paren{2\mathbb P_H \paren{ \frac{Z_t}{Z_{,\aa}}}\mathbb P_H\paren{ {Z_{t}}\partial_\aa\frac1{Z_{,\aa}}} \partial_\aa\Theta}
=0$ by principles no.1 and no.2 of \S\ref{sec:20}, we add this to the second term on the RHS of \eqref{eq:2008}; and observe further that
\begin{equation}\label{eq:2009}
-2 \frac{Z_t}{Z_{,\aa}}\mathbb P_H\paren{ {Z_{t}}\partial_\aa\frac1{Z_{,\aa}}} +\frac{Z_t^2}{Z_{,\aa}}\partial_\aa\frac1{Z_{,\aa}}= -\frac{Z_t}{Z_{,\aa}}\HH\paren{ {Z_{t}}\partial_\aa\frac1{Z_{,\aa}}}, \qquad\text{and}
\end{equation}
\begin{equation}\label{eq:2010}
-\frac{Z_t}{Z_{,\aa}}\HH\paren{ {Z_{t}}\partial_\aa\frac1{Z_{,\aa}}}\partial_\aa\Theta=-Z_t\HH\paren{ {Z_{t}}\partial_\aa\frac1{Z_{,\aa}}}\partial_\aa\paren{\frac{\Theta}{Z_{,\aa}}}+\paren{{Z_{t}}\partial_\aa\frac1{Z_{,\aa}}}\HH\paren{ {Z_{t}}\partial_\aa\frac1{Z_{,\aa}}}\Theta;
\end{equation}
and by straightforward expansion,
\begin{equation}\label{eq:2011}
\bracket{Z_t,\bracket{Z_t, \HH}}\partial_\aa\frac1{Z_{,\aa}}=-2Z_t\HH\paren{ {Z_{t}}\partial_\aa\frac1{Z_{,\aa}}}+(I+\HH)\paren{Z_t^2\partial_\aa\frac1{Z_{,\aa}}},
\end{equation}
and 
\begin{equation}\label{eq:2012}
{Z_{t}}\partial_\aa\frac1{Z_{,\aa}}\HH\paren{ {Z_{t}}\partial_\aa\frac1{Z_{,\aa}}}=\braces{\mathbb P_H\paren{{Z_{t}}\partial_\aa\frac1{Z_{,\aa}}}}^2-\braces{ \mathbb P_A\paren{{Z_{t}}\partial_\aa\frac1{Z_{,\aa}}}   }^2.
\end{equation}
Now beginning with \eqref{eq:2005}, combining the calculations in \eqref{eq:2008}-\eqref{eq:2012} and using principles no.1 and no.2 in \S\ref{sec:20}, we get
\begin{equation}\label{eq:2013}
  \begin{aligned}
 &(I-\HH)((\partial_t+\frak b\partial_\aa)^2\Theta +i\AA\partial_\aa\Theta)=(I-\HH)\paren{2\paren{\frac{Z_{tt}+i}{Z_{,\alpha'}}}\partial_{\alpha'} \Theta}\\&+(I-\HH)\paren{
 2\mathbb P_A\paren{\frac{Z_t} {Z_{,\alpha'}}}\partial_{\alpha'} \paren{(\partial_t+\frak b\partial_\aa)\Theta    -\mathbb P_A\paren{\frac{Z_t } {Z_{,\alpha'}}}\partial_{\alpha'} \Theta}}
\\& -(I-\HH)\paren{2\mathbb P_A \paren{ \frac{Z_t } {Z_{,\alpha'}} } \mathbb P_H\paren{ \frac{Z_{t,\alpha'} } {Z_{,\alpha'}}  }\partial_{\alpha'}\Theta}+ (I-\HH) \paren{\paren{\mathbb P_A\paren{\frac{Z_t } {Z_{,\alpha'}}}}^2\partial_{\alpha'}^2 \Theta}\\&+(I-\HH)\paren{\frac12\braces{[Z_t,[Z_t,\mathbb H]]\partial_{\alpha'}\frac1{Z_{,\alpha'}}}\partial_{\alpha'} \paren{\frac  \Theta{Z_{,\alpha'}}}-\braces{\mathbb P_A\paren{
 Z_t \partial_{\alpha'}\frac1{Z_{,\alpha'}}}}^2 \Theta}
.
\end{aligned}
 \end{equation}
 Using the holomorphicity of the factors to rewrite all, except for two,   terms on the right hand side of \eqref{eq:2013} as commutators gives
\begin{equation}\label{eq:2014}
  \begin{aligned}
 &(I-\HH)((\partial_t+\frak b\partial_\aa)^2\Theta +i\AA\partial_\aa\Theta)=2\bracket{\frac{Z_{tt}+i}{Z_{,\alpha'}},\HH}\partial_{\alpha'} \Theta\\&+2\bracket{
 \mathbb P_A\paren{\frac{Z_t} {Z_{,\alpha'}}},\HH}\partial_{\alpha'} \paren{(\partial_t+\frak b\partial_\aa)\Theta    -\mathbb P_A\paren{\frac{Z_t } {Z_{,\alpha'}}}\partial_{\alpha'} \Theta}
\\& -(I-\HH)\paren{2 \mathbb P_H\paren{ D_\aa Z_t }\mathbb P_A \paren{ \frac{Z_t } {Z_{,\alpha'}} }\partial_{\alpha'}\Theta}+ \bracket{\paren{\mathbb P_A\paren{\frac{Z_t } {Z_{,\alpha'}}}}^2,\HH}\partial_{\alpha'}^2 \Theta\\&+\bracket{\frac12\braces{[Z_t,[Z_t,\mathbb H]]\partial_{\alpha'}\frac1{Z_{,\alpha'}}},\HH}\partial_{\alpha'} \paren{\frac  \Theta{Z_{,\alpha'}}}-(I-\HH)\paren{\braces{\mathbb P_A\paren{
 Z_t \partial_{\alpha'}\frac1{Z_{,\alpha'}}}}^2 \Theta}
.
\end{aligned}
 \end{equation}
 By the identity\footnote{It is an easy consequence of integration by parts.}
 \begin{equation}\label{eq:2015}
 -2[g_1,\HH]\partial_\aa(g_1 g_2)+[g_1^2,\HH]\partial_\aa g_2=-[g_1, g_1; g_2],
 \end{equation}
 we combine part of the second term and the fourth term on the RHS of \eqref{eq:2014}:
 \begin{equation}\label{eq:2016}
 \begin{aligned}
 -2\bracket{
 \mathbb P_A\paren{\frac{Z_t} {Z_{,\alpha'}}},\HH}\partial_{\alpha'} \paren{\mathbb P_A\paren{\frac{Z_t } {Z_{,\alpha'}}}\partial_{\alpha'} \Theta}&+\bracket{\paren{\mathbb P_A\paren{\frac{Z_t } {Z_{,\alpha'}}}}^2,\HH}\partial_{\alpha'}^2 \Theta\\&=-\bracket{ \mathbb P_A\paren{\frac{Z_t} {Z_{,\alpha'}}}, \mathbb P_A\paren{\frac{Z_t} {Z_{,\alpha'}}}; \partial_\aa\Theta}.
 \end{aligned}
 \end{equation}
Observe that, using the identity $\mathbb P_A+\mathbb P_H=I$ and the fact that $(I-\HH)\partial_\aa\Theta=0$,
 \begin{equation}\label{eq:2017}
 \begin{aligned}
 (I-\HH)\paren{2 \mathbb P_H\paren{ D_\aa Z_t }\mathbb P_A \paren{ \frac{Z_t } {Z_{,\alpha'}} }\partial_{\alpha'}\Theta}&=(I-\HH)\paren{\mathbb P_H\paren{ D_\aa Z_t }\bracket{ \mathbb P_A \paren{ \frac{Z_t } {Z_{,\alpha'}} },\HH}\partial_{\alpha'}\Theta}\\&+(I-\HH)\paren{2 \mathbb P_H\paren{ D_\aa Z_t }\mathbb P_H\braces{\mathbb P_A \paren{ \frac{Z_t } {Z_{,\alpha'}} }\partial_{\alpha'}\Theta}};
 \end{aligned}
 \end{equation}
by \eqref{eq:160} the second term equals to the mean
\begin{equation}\label{eq:2018}
\begin{aligned}
(I-\HH)\paren{2 \mathbb P_H\paren{ D_\aa Z_t }\mathbb P_H\braces{\mathbb P_A \paren{ \frac{Z_t } {Z_{,\alpha'}} }\partial_{\alpha'}\Theta}}&=\frac12 \paren{\avg D_\aa Z_t\,d\aa}\paren{\avg \mathbb P_A \paren{ \frac{Z_t } {Z_{,\alpha'}} }\partial_{\alpha'}\Theta\,d\aa}\\&=-
\frac12 \paren{\avg D_\aa Z_t\,d\aa}\paren{\avg \partial_\aa\mathbb P_A \paren{ \frac{Z_t } {Z_{,\alpha'}} }\Theta\,d\aa}.
\end{aligned}
\end{equation}
We can now conclude, from \eqref{eq:2014}-\eqref{eq:2018} and  using \eqref{eq:305}, \eqref{eq:21} and \eqref{eq:6a} that
\begin{equation}\label{eq:2019}
\begin{aligned}
&\nm{(I-\HH)((\partial_t+\frak b\partial_\aa)^2\Theta +i\AA\partial_\aa\Theta)}_{L^2}\lec \nm{\partial_\aa{\frac{Z_{tt}+i}{Z_{,\alpha'}}}}_{L^\infty}\nm{\Theta}_{L^2}\\&+\nm{\partial_\aa{ \mathbb P_A\paren{\frac{Z_t} {Z_{,\alpha'}}}}}_{L^\infty}\nm{(\partial_t+\frak b\partial_\aa)\Theta}_{L^2}
+\nm{\partial_\aa{ \braces{[Z_t,[Z_t,\mathbb H]]\partial_{\alpha'}\frac1{Z_{,\alpha'}}}  } }_{L^2}\nm{\frac{\Theta}{Z_{,\aa}}}_{\dot H^{1/2}}\\&+ \paren{\nm{\partial_\aa{ \mathbb P_A\paren{\frac{Z_t} {Z_{,\alpha'}}}}}_{L^\infty}^2+\nm{\mathbb P_A\paren{
 Z_t \partial_{\alpha'}\frac1{Z_{,\alpha'}}}}_{L^\infty}^2}\nm{\Theta}_{L^2}
 \\&+(\nm{\mathbb P_H D_\aa Z_t}_{L^\infty}+\nm{D_\aa Z_t}_{L^\infty})\nm{\partial_\aa{ \mathbb P_A\paren{\frac{Z_t} {Z_{,\alpha'}}}}}_{L^\infty}\nm{\Theta}_{L^2}.
\end{aligned}
\end{equation}
We have estimated all the factors on the RHS of \eqref{eq:2019}, except for $\nm{\partial_\aa{ \braces{[Z_t,[Z_t,\mathbb H]]\partial_{\alpha'}\frac1{Z_{,\alpha'}}}  } }_{L^2}$, which can be controlled by \eqref{eq:144}:
\begin{equation}\label{eq:2020}
\nm{\partial_\aa{ \braces{[Z_t,[Z_t,\mathbb H]]\partial_{\alpha'}\frac1{Z_{,\alpha'}}}  } }_{L^2}\lec \nm{Z_{t,\aa}}_{L^2}^2\nm{\partial_\aa\frac1{Z_{,\aa}}}_{L^2}\lec C(E).
\end{equation}
This proves \eqref{lemma-2}.

\end{proof}

\section{Controlling \texorpdfstring{$G_\th$}{G theta} of \texorpdfstring{$E_b$}{Eb}}\label{sec:E2}
By \eqref{eq:173}, we must control
\begin{equation}
  \label{eq:421}
  \paren{\int \f{1}{\af} \abs{ D_\a(-i \af_t \bar{z}_\a) + [\partial_t^2 + i \af \partial_\a, D_\a] \bar{z}_t}^2 d\a}^{1/2}.
\end{equation}

We control the commutator via \eqref{eq:53}:
\begin{equation}
  \label{eq:422}
  \paren{\int \f{1}{\af} \abs{[\partial_t^2 + i \af \partial_\a, D_\a] \bar{z}_t}^2 d\a}^{1/2}  \lec (\nm{D_\a z_{tt}}_{L^\infty} + \nm{D_\a z_t}_{L^\infty}^2) \nm{D_\a \bar{z}_t}_{L^2(\f{1}{\af} d\a)},
\end{equation}
where we have controlled all the quantities on the RHS in \S \ref{sec:quants}. We are left with the term $\paren{\int \f{1}{\af} \abs{D_\a(-i\af_t \bar{z}_\a)}^2 d\a}^{1/2}$. 
Since ${\af}\abs{z_\a}^2=(A_1 \circ h) {h_\a}$ \eqref{eq:209}, $A_1\ge 1$  \eqref{eq:394}, and $(\af_t \bar{z}_\a)\circ h^{-1}=\AA_t \bar{Z}_{,\aa}$ we have
\begin{equation}\label{eq:1033}
\paren{\int \f{1}{\af} \abs{D_\a(-i\af_t \bar{z}_\a)}^2 d\a}^{1/2}\le \paren{\int \f1{h_\a} \abs{\partial_\a(-i\af_t \bar{z}_\a)}^2 d\a}^{1/2}=\paren{\int \abs{\partial_\aa(-i\AA_t \bar{Z}_{,\aa})}^2 d\aa}^{1/2},
\end{equation}
 where in the second step we changed to  Riemann mapping variables. We write $-i\AA_t \bar{Z}_{,\aa}$ as $\f{\AA_t}{\AA} (-i\AA \bar{Z}_{,\aa})$ and apply $\partial_\aa$. Since $\bar{Z}_{tt}-i =-i \AA \bar{Z}_{,\aa}$ \eqref{eq:179}, we have
 \begin{equation}
  \label{eq:218}
 \partial_\aa \paren{-i\AAt \bar{Z}_{,\aa}}=  (-i \AA \bar{Z}_{,\aa})\partial_\aa \paren{\f{\AAt}{\AA}} + \f{\AAt}{\AA} \bar{Z}_{tt,\aa}.
\end{equation}
Therefore,
\begin{equation}\label{eq:214}
\paren{\int \abs{\partial_\aa(-i\AA_t \bar{Z}_{,\aa})}^2 d\aa}^{1/2}\le \paren{\int  \abs{\mathcal A \bar{Z}_{,\aa} \partial_\aa \paren{\f{\mathcal A_t}{\mathcal A}}}^2 d\aa}^{1/2}+\nm{\f{\AA_t}{\AA}}_{L^\infty}\nm{
\bar{Z}_{tt,\aa}}_{L^2}.
\end{equation}
We controlled the factors in the second term on the RHS in \eqref{eq:406} and \eqref{eq:424}. We can therefore concentrate on the first term.  

We  seek a way of writing $ \mathcal A \bar{Z}_{,\aa} \partial_\aa \paren{\f{\AAt}{\AA}}$. The idea is to take advantage of the fact that $\partial_\aa \paren{\f{\AAt}{\AA}}$ is real, $|f|\le |(I-\mathbb H)f|$ for $f$ real, and $\mathcal A \bar{Z}_{,\aa}=i(\bar Z_{tt}-i)$ is controllable   to bound the term $ \mathcal A \bar{Z}_{,\aa} \partial_\aa \paren{\f{\AAt}{\AA}}$ by a sum of controllable terms and commutators.


Starting from \eqref{eq:218}, we replace the LHS by the derivative of the LHS of our quasilinear equation \eqref{eq:178}, and then apply $(I-\HH)$ to the equation. We get
\begin{equation}
  \label{eq:61}
  (I-\HH) \braces{(-i \AA \bar{Z}_{,\aa}) \partial_\aa \f{\AAt}{\AA}} = (I-\HH) \partial_\aa \paren{\bar{Z}_{ttt} + i \AA \bar{Z}_{t,\aa}} - (I-\HH) \braces{\f{\AAt}{\AA} \partial_\aa \bar{Z}_{tt}}.
\end{equation}
We next commute the factor  $-i\AA \bar{Z}_{,\aa}$ outside of $(I-\mathbb H)$ on the LHS, 
\begin{equation}
  \label{eq:69}
  (-i\AA \bar{Z}_{,\aa}) (I-\HH) \partial_\aa \f{\AAt}{\AA} = (I-\HH) \partial_\aa \paren{\bar{Z}_{ttt} + i \AA \bar{Z}_{t,\aa}} - (I-\HH)  \braces{\f{\AAt}{\AA} \partial_\aa \bar{Z}_{tt}} + [i\AA \bar{Z}_{,\aa},\HH] \partial_\aa \f{\AAt}{\AA}.
\end{equation}
Now $\f{\AAt}{\AA}$ is real and $\HH$ is purely imaginary, so $\abs{ \partial_\aa \f{\AAt}{\AA} }\le\abs{  (I-\HH) \partial_\aa \f{\AAt}{\AA} }$.
Taking absolute value on both sides,   we have
\begin{equation}
  \label{eq:482}
  \abs{\AA \bar{Z}_{,\aa} \partial_\aa \f{\AAt}{\AA}} \le \abs{(I-\HH) \partial_\aa \paren{\bar{Z}_{ttt} + i \AA \bar{Z}_{t,\aa}} - (I-\HH)  \braces{\f{\AAt}{\AA} \partial_\aa \bar{Z}_{tt}} + [i\AA \bar{Z}_{,\aa},\HH] \partial_\aa \f{\AAt}{\AA}}.
\end{equation}

We can easily control the $L^2$ norm of the second and third terms. By the $L^2$ boundedness of $\HH$ and H\"older's inequality for the second term and estimate \eqref{eq:228} for the third term, and since $i \AA \bar{Z}_{,\aa} = -(\bar{Z}_{tt} - i)$, 
\begin{equation}
  \label{eq:227}
  \nm{ - (I-\HH)  \braces{\f{\AAt}{\AA} \bar{Z}_{tt,\aa}} - [i\AA \bar{Z}_{,\aa},\HH] \partial_\aa \f{\AAt}{\AA} }_{L^2} \lec \nm{Z_{tt,\aa}}_{L^2} \nm{\f{\AAt}{\AA}}_{L^\infty}.
\end{equation}

We can now focus on controlling
\begin{equation}
  \label{eq:74}
  \nm{(I-\HH) \partial_\aa \paren{\bar{Z}_{ttt} + i \AA \bar{Z}_{t,\aa}}}_{L^2}=\nm{\partial_\aa  (I-\HH) \paren{\bar{Z}_{ttt} + i \AA \bar{Z}_{t,\aa}}}_{L^2},
\end{equation}
where we used $\paren{\bar{Z}_{ttt} + i \AA \bar{Z}_{t,\aa}}\in C^2(S^1)$, which follows from equations \eqref{eq:178}, \eqref{eq:179}, \eqref{eq:210} and the assumption of Theorem~\ref{maintheorem},
 to  commute $\partial_\aa$ outside $(I-\HH)$. 

By \eqref{lemma-1}, taking $\Theta=\bar Z_t$, we have 
\begin{equation}
  \label{eq:250}
   (I-\HH) \paren{\bar{Z}_{ttt} + i \AA \bar{Z}_{t,\aa} } = [Z_t^2,\HH] D_\aa^2 \bar{Z}_t + 2 [Z_t, \HH] D_\aa (\bar{Z}_{tt} - (D_\aa \bar{Z}_t) Z_t) + 2 [Z_{tt},\HH] D_\aa \bar{Z}_t.
\end{equation}
Therefore, by \eqref{eq:74} and \eqref{eq:250}, we have to control the $L^2$ norm of 
\begin{equation}
  \label{eq:163}
  \partial_\aa [Z_t^2,\HH] D_\aa^2 \bar{Z}_t + 2 \partial_\aa [Z_t, \HH] D_\aa (\bar{Z}_{tt} - (D_\aa \bar{Z}_t) Z_t) + 2 \partial_\aa [Z_{tt},\HH] D_\aa \bar{Z}_t.
\end{equation}
We use the identity 
\begin{equation}
  \label{eq:269}
  \partial_\aa [f,\HH] g = f' \HH g - \f{1}{\hdenomconst} \int \f{\pi}{2} \f{f(\aa) - f(\bb)}{\sin^2(\f{\pi}{2}(\aa -\bb))} g(\bb) d\bb
\end{equation}
to expand out each term in \eqref{eq:163}, and use \eqref{eq:315} and \eqref{eq:313} to remove the  $\HH$s from the RHS. We get
\begin{equation}
  \label{eq:36}
\begin{aligned}
\eqref{eq:163}&= 
   2Z_t Z_{t,\aa} D_\aa^2 \bar Z_t+2 Z_{t,\aa} D_\aa(\bar{Z}_{tt} - (D_\aa \bar{Z}_t) Z_t) + 2 Z_{tt,\a}  D_\aa \bar{Z}_t 
   - \f{\pi}{4i} \int  \f{Z_t^2(\aa) - Z_t^2(\bb)}{\sin^2(\f{\pi}{2}(\aa - \bb))} D_\bb^2 \bar Z_t d\bb 
\\& - \f{\pi}{2i} \int  \f{Z_t(\aa) - Z_t(\bb)}{\sin^2(\f{\pi}{2}(\aa - \bb))} D_\bb (\bar{Z}_{tt} - (D_\bb \bar{Z}_t) Z_t) d\bb 
 - \f{\pi}{2i} \int  \f{Z_{tt}(\aa) - Z_{tt}(\bb)}{\sin^2(\f{\pi}{2}(\aa -\bb))} D_\bb \bar{Z}_t d\bb.
\end{aligned}
\end{equation}
We expand out the RHS. We note that certain terms cancel out with others, and we further observe the following identity:
\begin{equation}
  \label{eq:235}
     \int \f {f^2(\aa) - f^2(\bb) } {\sin^2(\f{\pi}{2}(\aa - \bb)) }g(\bb) d\bb -2\int \f {(f(\aa) - f(\bb)) f(\bb) } {\sin^2(\f{\pi}{2}(\aa - \bb)) }g(\bb) d\bb 
 =  \int \f{(f(\aa) - f(\bb))^2 }{\sin^2(\f{\pi}{2}(\aa - \bb))} g(\bb) d\bb.
\end{equation}
We have
\begin{equation}\label{eq:1032}
\begin{aligned}
\eqref{eq:163}&= 2 Z_{t,\aa} (D_\aa \bar{Z}_{tt} - (D_\aa \bar{Z}_t) D_\aa Z_t) + 2 Z_{tt,\a}  D_\aa \bar{Z}_t 
   - \f{\pi}{4i} \int  \f{(Z_t(\aa) - Z_t(\bb))^2}{\sin^2(\f{\pi}{2}(\aa - \bb))} D_\bb^2 \bar Z_t d\bb 
\\& - \f{\pi}{2i} \int \f{Z_t(\aa) - Z_t(\bb)}{\sin^2(\f{\pi}{2}(\aa - \bb))} (D_\bb \bar{Z}_{tt} - (D_\bb \bar{Z}_t) D_\bb Z_t) d\bb 
 - \f{\pi}{2i} \int  \f{Z_{tt}(\aa) - Z_{tt}(\bb)}{\sin^2(\f{\pi}{2}(\aa -\bb))} D_\bb \bar{Z}_t d\bb.
\end{aligned}
\end{equation}
We now apply H\"older's inequality to the first two terms, \eqref{eq:113} to the third term, and \eqref{eq:258} to the last two terms.  We get
\begin{equation}
  \label{eq:310}
  \begin{aligned}
   &\nm{(I-\HH) \partial_\aa \paren{\bar{Z}_{ttt} + i \AA \bar{Z}_{t,\aa}}}_{L^2} =\nm{\eqref{eq:163}}_{L^2} \\&\lec \nm{Z_{t,\aa}}_{L^2} (\nm{D_\aa \bar{Z}_{tt}}_{L^\infty} + \nm{D_\aa \bar{Z}_t}_{L^\infty}^2) + \nm{Z_{tt,\aa}}_{L^2} \nm{D_\aa \bar{Z}_t}_{L^\infty} 
 +  \nm{Z_{t,\aa}}_{L^2}^2 \nm{D_\aa^2 \bar{Z}_t}_{L^2}.
 \end{aligned}
\end{equation}

We now combine our various estimates. We have
\begin{equation}\label{eq:327}
\nm{\AA \bar{Z}_{,\aa} \partial_\aa \f{\AAt}{\AA}}_{L^2}\le \eqref{eq:227}+\eqref{eq:310}\lec C(E)
\end{equation}
and
\begin{equation}\label{eq:311}
\begin{aligned}
 &\paren{\int \f{1}{\af} \abs{ D_\a(-i \af_t \bar{z}_\a) + [\partial_t^2 + i \af \partial_\a, D_\a] \bar{z}_t}^2 d\a}^{1/2}\le \eqref{eq:422}+\eqref{eq:214}\\ & \le \eqref{eq:422}+\eqref{eq:327}+\nm{\f{\AA_t}{\AA}}_{L^\infty} \nm{ \bar{Z}_{tt,\aa}}_{L^2}\lec C(E).
 \end{aligned}
 \end{equation}

We can now conclude that $\f d{dt} E_b$ is bounded by a polynomial of $E$.

We record here the estimate
\begin{equation}\label{eq:1034}
\nm{D_\aa \f{\AAt}{\AA}}_{L^2}\le \nm{\AA \bar{Z}_{,\aa} \partial_\aa \f{\AAt}{\AA}}_{L^2} \lec C(E),
\end{equation}
which holds because $\abs{\AA\bar{Z}_{,\aa}}=\f{A_1}{\abs{{Z}_{,\aa}}}\ge \f{1}{\abs{{Z}_{,\aa}}}$; we will use this in \S \ref{controlE1}.

\section{Controlling \texorpdfstring{$G_\th$}{G theta} of \texorpdfstring{$E_a$}{Ea}}\label{controlE1}
From \eqref{eq:66}, we must control
\begin{equation}
  \label{eq:399}
  \paren{\int  \abs{D_\a^2 (-i \af_t \bar{z}_\a) + [\partial_t^2 + i \af \partial_\a, D_\a^2]\bar{z}_t}^2\f{h_\a}{A_1 \circ h}d\a}^{1/2}.
\end{equation}

Recall that $A_1\ge 1$ \eqref{eq:394}. We control the commutator via  \eqref{eq:52} and H\"older's inequality: 
\begin{equation}
  \label{eq:415}
  \begin{aligned}
 & \nm{[\partial_t^2 + i \af \partial_\a, D_\a^2]\bar{z}_t}_{L^2(\f{h_\a}{A_1 \circ h}d\a)}\le  \nm{[\partial_t^2 + i \af \partial_\a, D_\a^2]\bar{z}_t}_{L^2({h_\a}d\a)}\\& \lec  \nm{D_\a z_{tt}}_{L^\infty} \nm{D_\a^2 \bar{z}_t}_{L^2({h_\a}d\a)} + \nm{D_\a z_t}_{L^\infty}^2 \nm{D_\a^2 \bar{z}_t}_{L^2({h_\a}d\a)} + \nm{D_\a z_t}_{L^\infty}\nm{D_\a \partial_t D_\a \bar{z}_t}_{L^2({h_\a}d\a)} \\&+ \nm{D_\a^2 z_{tt}}_{L^2({h_\a}d\a)} \nm{D_\a \bar{z}_t}_{L^\infty} + \nm{D_\a z_t}_{L^\infty}^2 \nm{D_\a^2 z_t}_{L^2({h_\a}d\a)} + \nm{D_\a^2 z_t}_{L^2({h_\a}d\a)} \nm{D_\a \bar{z}_{tt}}_{L^\infty} \\& + \nm{D_\a z_t}_{L^\infty} \nm{D_\a^2 \bar{z}_{tt}}_{L^2({h_\a}d\a)}.
 \end{aligned}
\end{equation}
We have controlled all quantities on the RHS in \S \ref{sec:quants}. We are left with the term $\paren{\int \abs{D_\a^2 \paren{\af_t \bar{z}_\a}}^2\f{h_\a}{A_1 \circ h} d\a}^{1/2}$. 

We know
\begin{equation}
  \label{eq:134a}
  \paren{\int \abs{D_\a^2 \paren{\af_t \bar{z}_\a}}^2\f{h_\a}{A_1 \circ h} d\a}^{1/2}\le \paren{\int \abs{D_\a^2 \paren{\af_t \bar{z}_\a}}^2{h_\a} d\a}^{1/2}=\paren{\int \abs{D_\aa^2 \paren{\AA_t \bar{Z}_{,\aa}}}^2 d\aa}^{1/2},
\end{equation}
where we  changed to Riemann mapping coordinate in the second step. We will now  focus on estimating 
\begin{equation}
  \label{eq:10}
  \paren{\int \abs{D_\aa^2 \paren{\AA_t \bar{Z}_{,\aa}}}^2 d\aa}^{1/2}.
\end{equation}

Our plan is to first turn the task of controlling $\nm{D_\aa^2 \paren{\AA_t \bar{Z}_{,\aa}}}_{L^2}$ to controlling $\nm{(I-\mathbb H)\paren{D_\aa^2 \paren{\AA_t \bar{Z}_{,\aa}}}}_{L^2}$. We will use the same idea as in the previous section, \S \ref{sec:E2},  that is, to take advantage of the fact that $\f{\AA_t}{\AA}$ is  real-valued and $\Re (I-\mathbb H)f=f$ for real valued $f$.  
We will then use \eqref{lemma-2} to control $\nm{(I-\mathbb H)\paren{D_\aa^2 \paren{\AA_t \bar{Z}_{,\aa}}}}_{L^2}$.

We begin by writing $\AA_t \bar{Z}_{,\aa} = \paren{\f{\AA_t}{\AA}} \AA \bar{Z}_{,\aa}$.  By the product rule,

\begin{equation}
  \label{eq:46}
  D_\aa^2 \paren{\AA_t \bar{Z}_{,\aa}} = \paren{D_\aa^2 \paren{\f{\AA_t}{\AA}}} \AA \bar{Z}_{,\aa} + 2 D_\aa \paren{\f{\AA_t}{\AA}} D_\aa(\AA \bar{Z}_{,\aa}) + \f{\AA_t}{\AA} D_\aa^2 (\AA \bar{Z}_{,\aa}).
\end{equation}
We can handle the second and third terms directly, using $\AA \bar{Z}_{,\aa} = i(\bar{Z}_{tt} - i)$ \eqref{eq:179}:
\begin{equation}
  \label{eq:316}
    \nm{2D_\aa \paren{\f{\AA_t}{\AA}}D_\aa(\AA \bar{Z}_{,\aa})+\f{\AA_t}{\AA} D_\a^2 (\AA \bar{Z}_{,\aa})}_{L^2}     \le 
    2\nm{D_\aa \paren{\f{\AA_t}{\AA}}}_{L^2} \nm{D_\a \bar{Z}_{tt}}_{L^\infty}+\nm{\f{\AA_t}{\AA}}_{L^\infty} \nm{D_\a^2 \bar{Z}_{tt}}_{L^2},
\end{equation}
where we have controlled all the quantities on the RHS in \S \ref{sec:quants} and in \eqref{eq:424} and \eqref{eq:1034}.   It therefore suffices to focus on the first term on the RHS of \eqref{eq:46}, $ \paren{D_\aa^2 \paren{\f{\AA_t}{\AA}}} \AA \bar{Z}_{,\aa}=i(\bar Z_{tt}-i) D_\aa^2 \paren{\f{\AA_t}{\AA}} $.

We now rearrange this term so that we can apply $(I-\HH)$ in a way so that we will be able to invert the operator by taking real parts.  Note that $\f{\AA_t}{\AA}$ is purely real. However, our derivative $D_\aa = \f{1}{Z_{,\aa}} \partial_\aa$ is not purely real.  To get around this, we factor the derivative into a real derivative and a complex modulus-one weight.  Recall our notation 
 $ \abs{D_\aa} = \f{1}{\abs{Z_{,\aa}}} \partial_\aa$. 
 Since $D_\aa = \paren{\f{\abs{Z_{,\aa}}}{Z_{,\aa}}} \abs{D_\aa}$, we rewrite
\begin{equation}
  \label{eq:320}
  D_\aa^2 = \paren{\f{\abs{Z_{,\aa}}}{Z_{,\aa}}}^2 \abs{D_\aa}^2 + \paren{\f{\abs{Z_{,\aa}}}{Z_{,\aa}}}\paren{\abs{D_\aa} \paren{\f{\abs{Z_{,\aa}}}{Z_{,\aa}}}} \abs{D_\aa}.
\end{equation}
Therefore, 
\begin{equation}
  \label{eq:321}
i(\bar Z_{tt}-i)    D_\aa^2 \f{\AA_t}{\AA} = i(\bar Z_{tt}-i)  \paren{\f{\abs{Z_{,\aa}}}{Z_{,\aa}}}^2 \abs{D_\aa}^2 \f{\AA_t}{\AA} + i(\bar Z_{tt}-i)  \paren{\f{\abs{Z_{,\aa}}}{Z_{,\aa}}}\paren{\abs{D_\aa} \paren{\f{\abs{Z_{,\aa}}}{Z_{,\aa}}}} \abs{D_\aa}\f{\AA_t}{\AA}.
\end{equation}
We use 
\begin{equation}  \label{eq:324}
  e := 
  i(\bar{Z}_{tt} - i) \paren{\f{\abs{Z_{,\aa}}}{Z_{,\aa}}}\paren{\abs{D_\aa} \paren{\f{\abs{Z_{,\aa}}}{Z_{,\aa}}}} \abs{D_\aa}\f{\AAt}{\AA}\end{equation}
to denote the second term,
which we will control directly, below at \eqref{eq:325}. 
We now apply $(I-\HH)$ to both sides of \eqref{eq:321}:
\begin{equation}
  \label{eq:55}
  (I-\HH) \braces{i(\bar{Z}_{tt} - i) D_\aa^2 \f{\AAt}{\AA}} = (I-\HH) \braces{i\f{(\bar{Z}_{tt} - i) }{\abs{Z_{,\aa}}}\paren{\f{\abs{Z_{,\aa}}}{Z_{,\aa}}}^2 \partial_\aa \paren{ \abs{D_\aa} \f{\AAt}{\AA}}} + (I-\HH) e.
\end{equation}
Observe that the first term on the RHS is purely real, except for the controllable factor 
$i\f{(\bar{Z}_{tt} - i) }{\abs{Z_{,\aa}}}\paren{\f{\abs{Z_{,\aa}}}{Z_{,\aa}}}^2 $.
We commute that part outside the $(I-\HH)$.  We get
\begin{equation}
  \label{eq:56}
  \begin{aligned}
      (I-\HH) \braces{i(\bar{Z}_{tt} - i) D_\aa^2 \f{\AAt}{\AA}} & = i\f{(\bar{Z}_{tt} - i)}{\abs{Z_{,\aa}}} \paren{\f{\abs{Z_{,\aa}}}{Z_{,\aa}}}^2  (I-\HH) \partial_\aa \paren{ \abs{D_\aa} \f{\AAt}{\AA}} 
\\ & + \bracket{i\f{(\bar{Z}_{tt} - i)}{\abs{Z_{,\aa}}}  \paren{\f{\abs{Z_{,\aa}}}{Z_{,\aa}}}^2,\HH}  \partial_\aa \paren{ \abs{D_\aa} \f{\AAt}{\AA}} + (I-\HH) e.
  \end{aligned}
\end{equation}
Because $\HH$ is purely imaginary,  $\abs{\partial_\aa \paren{ \abs{D_\aa} \f{\AAt}{\AA}}}\le \abs{(I-\HH) \partial_\aa \paren{ \abs{D_\aa} \f{\AAt}{\AA}} }$. By taking absolute values, we have
\begin{multline}
  \label{eq:58}
  \abs{i\f{(\bar{Z}_{tt} - i)}{\abs{Z_{,\aa}}} \partial_\aa \paren{ \abs{D_\aa} \f{\AAt}{\AA}}}
 \\ \le \abs{(I-\HH) \braces{i(\bar{Z}_{tt} - i) D_\aa^2 \f{\AAt}{\AA}} - \bracket{i\f{(\bar{Z}_{tt} - i)}{\abs{Z_{,\aa}}} \paren{\f{\abs{Z_{,\aa}}}{Z_{,\aa}}}^2,\HH}  \partial_\aa \paren{ \abs{D_\aa} \f{\AAt}{\AA}} - (I-\HH) e}.
\end{multline}

Now we may begin controlling these terms. Recall that what we needed to control was the $L^2$ norm of  \eqref{eq:321}. We can estimate this by
\begin{equation}
  \label{eq:71}
  \begin{aligned}
    \nm{-i(\bar{Z}_{tt} - i) D_\aa^2 \f{\AAt}{\AA} }_{L^2} &   \lec \nm{i\f{(\bar{Z}_{tt} - i)}{\abs{Z_{,\aa}}} \partial_\aa \paren{ \abs{D_\aa} \f{\AAt}{\AA}}}_{L^2} + \nm{e}_{L^2}
 \lec \nm{\eqref{eq:58}}_{L^2} + \nm{e}_{L^2}
\\ & \lec \nm{(I-\HH) \braces{i(\bar{Z}_{tt} - i) D_\aa^2 \f{\AAt}{\AA}}}_{L^2}
\\ & + \nm{\bracket{i \f{(\bar{Z}_{tt} - i)}{\abs{Z_{,\aa}}}\paren{\f{\abs{Z_{,\aa}}}{Z_{,\aa}}}^2,\HH}  \partial_\aa \paren{ \abs{D_\aa} \f{\AAt}{\AA}}}_{L^2} + \nm{e}_{L^2}.
\end{aligned}
\end{equation}
Thus, it suffices to focus on these three terms.  

First we check the error term, $e$ \eqref{eq:324}. We control
\begin{equation}
  \label{eq:325}
  \begin{aligned}
  \nm{e}_{L^2} &\le 
  \nm{(\bar{Z}_{tt} - i) \paren{\abs{D_\aa} \f{\abs{Z_{,\aa}}}{Z_{,\aa}}}}_{L^\infty}\nm{D_\aa \f{\AAt}{\AA}}_{L^2}\\&\lec \nm{(\bar{Z}_{tt} - i)\partial_\aa \f{1}{Z_{,\aa}}}_{L^\infty} \nm{D_\aa \f{\AAt}{\AA}}_{L^2},
  \end{aligned}
\end{equation}
where in the second step we used   $\abs{\partial_\aa \f{f}{\abs{f}}} \le \abs{\f{f'}{\abs{f}}}$ \eqref{eq:100}. We have controlled both factors on the RHS in \eqref{eq:1024} and \eqref{eq:1034}.

Now we estimate the second term on the RHS of \eqref{eq:71}. 
 We have, by $L^\infty \times L^2$ commutator estimate \eqref{eq:305},
    \begin{equation}
      \label{eq:337}
      \begin{aligned}
      \nm{\bracket{i\f{(\bar{Z}_{tt} - i) }{\abs{Z_{,\aa}}} \paren{\f{\abs{Z_{,\aa}}}{Z_{,\aa}}}^2,\HH}  \partial_\aa \abs{D_\aa} \f{\AAt}{\AA}}_{L^2}&\lec \nm{\partial_{\aa}\paren{\f{(\bar{Z}_{tt} - i) }{\abs{Z_{,\aa}}} \paren{\f{\abs{Z_{,\aa}}}{Z_{,\aa}}}^2}}_{L^\infty}\nm{D_\aa \f{\AAt}{\AA}}_{L^2}\\&\lec
    \paren{\nm{D_\aa Z_{tt}}_{L^\infty}+  \nm{(\bar{Z}_{tt} - i)\partial_\aa \f{1}{Z_{,\aa}}}_{L^\infty}} \nm{D_\aa \f{\AAt}{\AA}}_{L^2},
     \end{aligned}
    \end{equation}
   where in the second step we used  \eqref{eq:100}. We have controlled all factors on the RHS in \eqref{eq:17}, \eqref{eq:1024} and \eqref{eq:1034}.

We're left with the first, main term of the RHS of \eqref{eq:71}. Observe that by \eqref{eq:46}, our main equation $\bar Z_{ttt}+i\AA \bar Z_{t,\aa}=-i\AA_t\bar Z_{,\aa}$ \eqref{eq:178}, and the $L^2$ boundedness of $\HH$,
\begin{equation}
  \label{eq:341}
  \begin{aligned}
     \nm{(I-\HH) \braces{i(\bar{Z}_{tt} - i) D_\aa^2 \f{\AAt}{\AA}}}_{L^2}  & \lec \nm{(I-\HH) \paren{D_\aa^2 (\AA_t \bar{Z}_{,\aa})}}_{L^2}
 + \eqref{eq:316}\\&= \nm{(I-\HH) D_\aa^2(\bar{Z}_{ttt} + i \AA \bar{Z}_{t,\aa})}_{L^2}+ \eqref{eq:316}.
 \end{aligned}
\end{equation}
We have therefore reduced things to controlling  $ \nm{(I-\HH) D_\aa^2(\bar{Z}_{ttt} + i \AA \bar{Z}_{t,\aa})}_{L^2}$.

Observe that 
\begin{equation}\label{eq:2021}
D_\aa^2(\bar{Z}_{ttt} + i \AA \bar{Z}_{t,\aa})=\paren{(\partial_t+\frak b\partial_\aa)^2+i\AA\partial_\aa}D_\aa^2\bar Z_{t}+[D_\aa^2, (\partial_t+\frak b\partial_\aa)^2+i\AA\partial_\aa]\bar Z_t.
\end{equation}
We have controlled the second term on the RHS in \eqref{eq:415}. Applying   $(I-\HH)$ and using \eqref{lemma-2} on the first term then gives
\begin{equation}\label{eq:2022}
\nm{(I-\HH)D_\aa^2(\bar{Z}_{ttt} + i \AA \bar{Z}_{t,\aa})}_{L^2}\lec C(E).
\end{equation}

We can now conclude that
\begin{equation}\label{eq:2023}
\nm{D_\aa^2(\AA_t Z_\aa)}_{L^2}\lec C(E).
\end{equation}
and therefore
\begin{equation}\label{eq:1600}
\nm{G_{D^2_\a \bar z_t}}_{L^2(\frac{h_\a}{A_1\circ h})}\le \eqref{eq:415} +\eqref{eq:134a}\le \eqref{eq:415} +\eqref{eq:2023}\lec C(E).
\end{equation}

We have now shown that $\frac d{dt}E_a$ is bounded by a polynomial of $E$. This completes the proof of Theorem~\ref{maintheorem}. \qed

\section{A characterization of the energy}\label{sec:36e}
Our energy is expressed in terms of not only the free surface $Z$, the velocity $Z_t$, and their spatial derivatives, but also time derivatives of these quantities. In this section, we give a characterization of our energy in terms of the free surface $Z$, the velocity $Z_t$, and their spatial derivatives. 
In \S \ref{sec:13a}, we discuss which crest angles are allowed by a finite energy $E$.

\subsection{A characterization of the energy  in terms of position and velocity}\label{sec:36d}
  In this section, we translate the terms of our energy involving time derivatives into terms depending only on the free surface $Z$, the velocity $Z_t$, and their spatial derivatives. We do this using the basic equation \eqref{eq:251},
\eqref{eq:208}, \eqref{eq:246a} and the holomorphicity of $\bar Z_t$ and $\frac1{Z_{,\aa}}$. 
These basic  equations allow us to show that quantities involving $\bar{Z}_{tt}$ can be controlled by analogous quantities involving $\f{1}{Z_{,\aa}}$, along with various lower-order terms.\footnote{We remark that for these estimates we {\em do not} ever rely on (high order) $\dot{H}^{1/2}$ parts of the energies.}

The estimate we prove is
\begin{multline}
  \label{eq:532}
  E(t) \le C\bigg(\nm{\bar{Z}_{t,\aa}}_{L^2}, \nm{D_\aa^2 \bar{Z}_t}_{L^2}, \nm{\partial_\aa \f{1}{Z_{,\aa}}}_{L^2},
  \\ \nm{D_\aa^2 \f{1}{Z_{,\aa}}}_{L^2}, \nm{\f{1}{Z_{,\aa}} D_\aa^2 \bar{Z}_t}_{\dot{H}^{1/2}}, \nm{D_\aa \bar{Z}_t}_{\dot{H}^{1/2}}, \nm{\f{1}{Z_{,\aa}}}_{L^\infty}\bigg),
\end{multline}
where the constant depends polynomially on its terms. We remark that this inequality can be reversed: each of the factors on the RHS of \eqref{eq:532} is controlled by the energy. That is,
\begin{multline}
  \label{eq:575}
  \nm{\bar{Z}_{t,\aa}}_{L^2}, \nm{D_\aa^2 \bar{Z}_t}_{L^2}, \nm{\partial_\aa \f{1}{Z_{,\aa}}}_{L^2}, \nm{D_\aa^2 \f{1}{Z_{,\aa}}}_{L^2}, \nm{\f{1}{Z_{,\aa}} D_\aa^2 \bar{Z}_t}_{\dot{H}^{1/2}}, \\ \nm{D_\aa \bar{Z}_t}_{\dot{H}^{1/2}},   \nm{\f{1}{Z_{,\aa}}}_{L^\infty}   \lec C(E(t)).
\end{multline}
Therefore, these quantities fully characterize our energy. In the proof of our a priori estimate, we have shown \eqref{eq:575} for every term except $\nm{D_\aa^2 \f{1}{Z_{,\aa}}}_{L^2}$, which we never had a need to control. One can adapt the argument in \S \ref{sec:573} below to show that $\nm{D_\aa^2 \f{1}{Z_{,\aa}}}_{L^2}$ can be controlled by the energy.\footnote{To do this, it comes down once again to estimating $\nm{\f{1}{Z_{,\aa}} D_\aa^2 A_1}_{L^2}$, except this time we need to do this without the dependence on $\nm{D_\aa^2 \f{1}{Z_{,\aa}}}_{L^2}$. That dependence comes from estimate \eqref{eq:411}. (It also comes from using $\nm{D_\aa^2 \f{1}{Z_{,\aa}}}_{L^2}$ in the Sobolev inequality for $\nm{D_\aa \f{1}{Z_{,\aa}}}_{L^\infty}$; this is not a problem, since $\nm{D_\aa \f{1}{Z_{,\aa}}}_{L^\infty}$ is controlled by the energy.) To handle \eqref{eq:411}, we take advantage of the fact that $(I-\HH) \braces{\partial_\aa D_\aa \f{1}{Z_{,\aa}}} = 0$ (this is due to \eqref{eq:1026a} and the second principle in \S \ref{sec:20}) to rewrite the term in question as a commutator and then use commutator estimate \eqref{eq:228}:
  \begin{equation}
    \label{eq:574}
      \nm{(I-\HH) \braces{\f{A_1}{Z_{,\aa}} \partial_\aa D_\aa \f{1}{Z_{,\aa}}}}_{L^2}  = \nm{\bracket{\bar{Z}_{tt},\HH} \partial_\aa D_\aa \f{1}{Z_{,\aa}}}_{L^2}
  \lec \nm{\bar{Z}_{tt,\aa}}_{L^2} \nm{D_\aa \f{1}{Z_{,\aa}}}_{L^\infty},
  \end{equation}
both of which are controlled by the energy. }

We remark because both $\bar{Z}_t$ and $\f{1}{Z_{,\aa}}$ are the boundary values of periodic holomorphic functions, the weighted derivative $D_\aa$ corresponds to the complex derivative $\partial_z$, or the gradient of the corresponding quantities in the spatial domain $P^-$. We also note that ${Z_{,\aa}} = (\Phi^{-1})_z$ is a natural geometric quantity well-suited to this problem: it captures the geometry of the free surface directly through the Riemann mapping $\Phi^{-1}: P^-\to \Omega(t)$.

\subsubsection{The proof}
Throughout the following proof we will rely on the fact that $A_1 \ge 1$ \eqref{eq:394}, the estimate \eqref{eq:361}
\begin{equation}
  \label{eq:533}
  \nm{A_1}_{L^\infty} \lec 1 + \nm{\bar{Z}_{t,\aa}}_{L^2}^2,
\end{equation}
 the Sobolev estimate \eqref{eq:206}
\begin{equation}
  \label{eq:502}
  \nm{D_\aa \bar{Z}_t}_{L^\infty} \lec \nm{\bar{Z}_{t,\aa}}_{L^2} + \nm{D_\aa^2 \bar{Z}_t}_{L^2},
\end{equation}
and the estimate
\begin{equation}
  \label{eq:453}
  \nm{D_\aa \f{1}{Z_{,\aa}}}_{L^\infty} \lec \nm{D_\aa^2 \f{1}{Z_{,\aa}}}_{L^2} + \nm{\partial_\aa \f{1}{Z_{,\aa}}}_{L^2},
\end{equation}
which holds by Sobolev inequality \eqref{eq:638}.\footnote{Note that $\avg \paren{D_\aa \f{1}{Z_{,\aa}}}^2 = 0$ by the same argument that was used at \eqref{eq:48} to show $\avg (D_\aa \bar{Z}_t)^2 = 0$.}

We begin by noting that it suffices to control only the first terms of $E_a$ and $E_b$, since the remaining terms of the energy are (up to a factor of $A_1$) already on the RHS of \eqref{eq:532}.

For the first term of $E_a$, by the commutator identity \eqref{eq:121},
\begin{equation}
  \label{eq:24}
  \begin{aligned}
    \int \abs{\partial_t D_\a^2 \bar{z}_t}^2 \f{h_\a}{A_1 \circ h} d\a & \lec \int \abs{D_\a^2 \bar{z}_{tt}}^2 \f{h_\a}{A_1 \circ h} d\a + \int \abs{[\partial_t,D_\a^2] \bar{z}_t}^2 \f{h_\a}{A_1 \circ h} d\a
\\ & \lec \nm{D_\aa^2 \bar{Z}_{tt}}_{L^2}^2 + \int \abs{2(D_\a z_t)D_\a^2 \bar{z}_t + (D_\a^2 z_t) D_\a \bar{z}_t}^2 \f{h_\a}{A_1 \circ h} d\a
\\ & \lec \nm{D_\aa^2 \bar{Z}_{tt}}_{L^2}^2 + \nm{D_\a z_t}_{L^\infty}^2(\nm{D_\aa^2 \bar{Z}_{t}}_{L^2}^2 + \nm{D_\aa^2 Z_t}_{L^2}^2).
  \end{aligned}
\end{equation}
By \eqref{eq:65a} and \eqref{eq:100},
\begin{equation}
  \label{eq:64}
      \nm{D_\aa^2 Z_t}_{L^2} \lec \nm{D_\aa^2 \bar{Z}_t}_{L^2} + \nm{D_\aa \bar{Z}_t}_{L^\infty} \nm{\partial_\aa \f{1}{Z_{,\aa}}}_{L^2}.
\end{equation}
We conclude that
\begin{equation}
  \label{eq:65}
  \begin{aligned}
      \int \abs{\partial_t D_\a^2 \bar{z}_t}^2 \f{h_\a}{A_1 \circ h} d\a &\lec   C\paren{\nm{D_\aa^2 \bar{Z}_{tt}}_{L^2}, \nm{\bar{Z}_{t,\aa}}_{L^2}, \nm{D_\aa^2 \bar{Z}_t}_{L^2}, \nm{\partial_\aa \f{1}{Z_{,\aa}}}_{L^2}}.
\end{aligned}
\end{equation}

For the first term of $E_b$, we use the commutator identity \eqref{eq:120} to get
\begin{equation}
  \label{eq:412}
  \begin{aligned}
    \int &\abs{\partial_t D_\a \bar{z}_t}^2 \f{1}{\af} d\a  \lec \int \abs{D_\a \bar{z}_{tt}}^2 \f{1}{\af} d\a + \int \abs{[\partial_t,D_\a] \bar{z}_t}^2 \f{1}{\af} d\a
\\ & \lec \int \abs{D_\a \bar{z}_{tt}}^2 \f{(A_1 \circ h)}{\af} d\a + \int \abs{(D_\a z_t)D_\a \bar{z}_t}^2 \f{(A_1 \circ h)}{\af} d\a
\\ & \lec \nm{\bar{Z}_{tt,\aa}}_{L^2}^2 + \nm{D_\a z_t}_{L^\infty}^2 \nm{\bar{Z}_{t,\aa}}_{L^2}^2
 \le C(\nm{\bar{Z}_{tt,\aa}}_{L^2}, \nm{\bar{Z}_{t,\aa}}_{L^2}, \nm{D_\aa^2 \bar{Z}_t}_{L^2}).
  \end{aligned}
\end{equation}

All that remains to do from \eqref{eq:65} and \eqref{eq:412} is to estimate $\nm{\bar{Z}_{tt,\aa}}_{L^2}$ and $\nm{D_\aa^2 \bar{Z}_{tt}}_{L^2}$ in terms of $Z_t$ and $\f{1}{Z_{,\aa}}$, which we now do, in \S \ref{sec:534} and \S \ref{sec:573}.

\subsubsection{Controlling \texorpdfstring{$\nm{\bar{Z}_{tt,\aa}}_{L^2}$}{|| bar Ztt,a' || L2}}\label{sec:534}

Using \eqref{eq:251}, we estimate
\begin{equation}
  \label{eq:534}
    \nm{\partial_\aa \bar{Z}_{tt}}_{L^2}  \lec \nm{A_1}_{L^\infty} \nm{\partial_\aa \f{1}{Z_{,\aa}}}_{L^2} + \nm{D_\aa A_1}_{L^2}.
\end{equation}
To control $\nm{D_\aa A_1}_{L^2}$, we follow a similar procedure to what we did in \eqref{eq:254}-\eqref{eq:255}, except instead of using $\bar{Z}_{tt} - i$, we use $\f{1}{Z_{,\aa}}$ and estimate things in terms of $\f{1}{Z_{,\aa}}$. We get
\begin{equation}
  \label{eq:496}
    \nm{D_\aa A_1}_{L^2} \lec \nm{\bar{Z}_{t,\aa}}_{L^2}^2 \nm{\partial_\aa \f{1}{Z_{,\aa}}}_{L^2} + \nm{Z_{t,\aa}}_{L^2} \nm{D_\aa \bar{Z}_t}_{L^\infty}.
\end{equation}
Combining \eqref{eq:534} and \eqref{eq:496} we conclude that
\begin{equation}
  \label{eq:545}
  \begin{aligned}
    \nm{\partial_\aa \bar{Z}_{tt}}_{L^2} \le C\paren{\nm{\bar{Z}_{t,\aa}}_{L^2}, \nm{\partial_\aa \f{1}{Z_{,\aa}}}_{L^2}, \nm{D_\aa^2 \bar{Z}_t}_{L^2}}.
  \end{aligned}
\end{equation}

\subsubsection{Controlling \texorpdfstring{$\nm{D_\aa^2 \bar{Z}_{tt}}_{L^2}$}{|| Da' 2 bar Ztt || L2}}\label{sec:573}
From \eqref{eq:251}, we have
\begin{equation}\label{eq:1000}
iD_\aa^2 \bar{Z}_{tt}= \underbrace{A_1 D_\aa^2 \f{1}{Z_{,\aa}}+2(D_\aa A_1) D_\aa \f{1}{Z_{,\aa}}}_{e_1} + \f{1}{Z_{,\aa}} D_\aa^2 A_1.
\end{equation}
We estimate $\nm{D_\aa^2 \bar{Z}_{tt}}_{L^2}$ through  the following procedure. First we note that the only challenging term to control on the RHS of \eqref{eq:1000} is the last one, $\f{1}{Z_{,\aa}} D_\aa^2 A_1$. We observe that this is almost real, modulo factors of $\f{1}{Z_{,\aa}}$ and its derivatives. Therefore, we will be able to use the $\Re (I-\HH)$ trick and, through a series of commutators,  reduce the estimate for $\f{1}{Z_{,\aa}} D_\aa^2 A_1$ to an estimate 
of $(I-\HH) (D_\aa^2 \f{A_1}{Z_{,\aa}} )=(I-\HH)(iD_\aa^2 \bar{Z}_{tt})$. Since $\bar{Z}_t$ is holomorphic, we will be able to rewrite $(I-\HH)(iD_\aa^2 \bar{Z}_{tt})$ in terms of commutators and obtain desirable estimates. We now give the details.

We first estimate the error term $e_1$ in \eqref{eq:1000}:
\begin{equation}
  \label{eq:573}
  \begin{aligned}
    \nm{e_1}_{L^2} & \lec \nm{A_1}_{L^\infty} \nm{D_\aa^2 \f{1}{Z_{,\aa}}}_{L^2} + \nm{D_\aa A_1}_{L^2} \nm{D_\aa \f{1}{Z_{,\aa}}}_{L^\infty} 
\\ & \lec (1 + \nm{\bar{Z}_{t,\aa}}_{L^2}^2) \nm{D_\aa^2 \f{1}{Z_{,\aa}}}_{L^2} + \eqref{eq:496} \paren{\nm{D_\aa^2 \f{1}{Z_{,\aa}}}_{L^2} + \nm{\partial_\aa \f{1}{Z_{,\aa}}}_{L^2}}.
  \end{aligned}
\end{equation}

It remains to control $\nm{\f{1}{Z_{,\aa}} D_\aa^2 A_1}_{L^2}$. We want to use $(I-\HH)$ to turn our quantity into commutators, but to do so we need to factor $D_\aa$ into a real-weighted derivative $\abs{D_\aa} := \f{1}{\abs{Z_{,\aa}}} \partial_\aa$ so that we may invert $(I-\HH)$. From \eqref{eq:320}, we have
\begin{equation}
  \label{eq:168}
  \f{1}{Z_{,\aa}} D_\aa^2 A_1 = \f{1}{Z_{,\aa}} \paren{\f{\abs{Z_{,\aa}}}{Z_{,\aa}}}^2 \abs{D_\aa}^2 A_1 + \underbrace{\f{1}{Z_{,\aa}} \f{\abs{Z_{,\aa}}}{Z_{,\aa}}\paren{\abs{D_\aa} \f{\abs{Z_{,\aa}}}{Z_{,\aa}}} \abs{D_\aa} A_1}_{e_2}.
\end{equation}
We multiply both sides by $\paren{\f{Z_{,\aa}}{\abs{Z_{,\aa}}}}^3$ so that the first term on the RHS is purely real:
\begin{equation}
  \label{eq:274}
  \paren{\f{Z_{,\aa}}{\abs{Z_{,\aa}}}}^3\f{1}{Z_{,\aa}} D_\aa^2 A_1 = \paren{\f{Z_{,\aa}}{\abs{Z_{,\aa}}}}^3\f{1}{Z_{,\aa}} \paren{\f{\abs{Z_{,\aa}}}{Z_{,\aa}}}^2 \abs{D_\aa}^2 A_1 + \paren{\f{Z_{,\aa}}{\abs{Z_{,\aa}}}}^3 e_2.
\end{equation}
Now we apply $\Re(I-\HH)$ to each side, and conclude from the fact that $A_1 \in \R$ that
\begin{equation}
  \label{eq:369}
  \abs{\f{1}{Z_{,\aa}} \abs{D_\aa}^2 A_1} \lec \abs{(I-\HH) \braces{\paren{\f{Z_{,\aa}}{\abs{Z_{,\aa}}}}^3\f{1}{Z_{,\aa}} D_\aa^2 A_1}} + \abs{(I-\HH) \braces{\paren{\f{Z_{,\aa}}{\abs{Z_{,\aa}}}}^3 e_2}}.
\end{equation}
We conclude from \eqref{eq:168} and \eqref{eq:369} that
\begin{equation}
  \label{eq:372}
  \nm{\f{1}{Z_{,\aa}} D_\aa^2 A_1}_{L^2} \lec \nm{e_2}_{L^2} + \nm{(I-\HH) \braces{\paren{\f{Z_{,\aa}}{\abs{Z_{,\aa}}}}^3\f{1}{Z_{,\aa}} D_\aa^2 A_1}}_{L^2}.
\end{equation}
By \eqref{eq:100} and \eqref{eq:496} we estimate
\begin{equation}
  \label{eq:373}
  \nm{e_2}_{L^2} \lec \nm{D_\aa \f{1}{Z_{,\aa}}}_{L^\infty} \nm{ D_\aa A_1}_{L^2}\lec \eqref{eq:453}\eqref{eq:496}.
\end{equation}
It remains to estimate $\nm{(I-\HH) \braces{\paren{\f{Z_{,\aa}}{\abs{Z_{,\aa}}}}^3\f{1}{Z_{,\aa}} D_\aa^2 A_1}}_{L^2}$. To get the right commutator estimate, we first rewrite this as
\begin{equation}
  \label{eq:374}
  \begin{aligned}
      \nm{(I-\HH) \braces{\paren{\f{Z_{,\aa}}{\abs{Z_{,\aa}}}}^3\f{1}{Z_{,\aa}} D_\aa^2 A_1}}_{L^2} & \lec \nm{(I-\HH) \braces{\paren{\f{Z_{,\aa}}{\abs{Z_{,\aa}}}}^3\f{1}{Z_{,\aa}} \partial_\aa \paren{\f{1}{Z_{,\aa}} D_\aa A_1}}}_{L^2} 
\\ & + \nm{(I-\HH) \braces{\paren{\f{Z_{,\aa}}{\abs{Z_{,\aa}}}}^3\f{1}{Z_{,\aa}} \paren{\partial_\aa \f{1}{Z_{,\aa}}} D_\aa A_1}}_{L^2}.
    \end{aligned}
  \end{equation}
We estimate the second term on the RHS of \eqref{eq:374} directly:
\begin{equation}
  \label{eq:379}
  \begin{aligned}
    \nm{(I-\HH) \braces{\paren{\f{Z_{,\aa}}{\abs{Z_{,\aa}}}}^3\f{1}{Z_{,\aa}} \paren{\partial_\aa \f{1}{Z_{,\aa}}} D_\aa A_1}}_{L^2} & \lec \nm{D_\aa \f{1}{Z_{,\aa}}}_{L^\infty} \nm{ D_\aa A_1}_{L^2}
\\ & \lec \eqref{eq:453}\eqref{eq:496}.
\end{aligned}
\end{equation}
For the first term on the RHS of \eqref{eq:374}, we commute the  factor $\paren{\f{Z_{,\aa}}{\abs{Z_{,\aa}}}}^3$ outside the $(I-\HH)$, bringing along $\f{1}{Z_{,\aa}}$ to ensure that the commutator is controllable, and then bringing the $\f{1}{Z_{,\aa}}$ back inside:
\begin{equation}
  \label{eq:383}
  \begin{aligned}
    & \nm{(I-\HH) \braces{\paren{\f{Z_{,\aa}}{\abs{Z_{,\aa}}}}^3\f{1}{Z_{,\aa}} \partial_\aa \paren{\f{1}{Z_{,\aa}} D_\aa A_1}}}_{L^2}  
  \lec \nm{\bracket{\paren{\f{Z_{,\aa}}{\abs{Z_{,\aa}}}}^3\f{1}{Z_{,\aa}},\HH}\partial_\aa \paren{\f{1}{Z_{,\aa}} D_\aa A_1}}_{L^2} 
 \\ & + \nm{\bracket{\f{1}{Z_{,\aa}},\HH}\partial_\aa \paren{\f{1}{Z_{,\aa}} D_\aa A_1}}_{L^2} + \nm{(I-\HH)\braces{\f{1}{Z_{,\aa}} \partial_\aa \paren{\f{1}{Z_{,\aa}} D_\aa A_1}}}_{L^2}.
  \end{aligned}
\end{equation}
We estimate the first two terms on the RHS of \eqref{eq:383} using commutator estimate \eqref{eq:228}:  
\begin{equation}
  \label{eq:397}
  \begin{aligned}
    \nm{\bracket{\paren{\f{Z_{,\aa}}{\abs{Z_{,\aa}}}}^3\f{1}{Z_{,\aa}},\HH}\partial_\aa \paren{\f{1}{Z_{,\aa}} D_\aa A_1}}_{L^2} & + \nm{\bracket{\f{1}{Z_{,\aa}},\HH}\partial_\aa \paren{\f{1}{Z_{,\aa}} D_\aa A_1}}_{L^2} 
\\ & \lec \nm{\partial_\aa \f{1}{Z_{,\aa}}}_{L^2} \nm{\f{1}{Z_{,\aa}} D_\aa A_1}_{L^\infty}.
  \end{aligned}
\end{equation}
We will postpone estimating $ \nm{\f{1}{Z_{,\aa}} D_\aa A_1}_{L^\infty}$ until the end of this  series of calculations. For the moment, we take the last term from the RHS of \eqref{eq:383}:
\begin{equation}
  \label{eq:407}
  \begin{aligned}
    &\nm{(I-\HH)\braces{\f{1}{Z_{,\aa}} \partial_\aa \paren{\f{1}{Z_{,\aa}} D_\aa A_1}}}_{L^2}\\&  \lec \nm{(I-\HH)D_\aa^2\paren{\f{1}{Z_{,\aa}}A_1}}_{L^2}
 + \nm{(I-\HH) \braces{D_\aa \paren{A_1 D_\aa \f{1}{Z_{,\aa} } }}}_{L^2}.
  \end{aligned}
\end{equation}
We estimate the second term by
\begin{equation}
  \label{eq:411}
  \begin{aligned}
    \nm{(I-\HH) \braces{ D_\aa \paren{A_1 D_\aa \f{1}{Z_{,\aa} } }}  }_{L^2} & \lec \nm{D_\aa^2 \f{1}{Z_{,\aa}}}_{L^2} \nm{A_1}_{L^\infty} +\nm{D_\aa \f{1}{Z_{,\aa}}}_{L^\infty} \nm{ D_\aa A_1}_{L^2}
\\ & \lec \nm{D_\aa^2 \f{1}{Z_{,\aa}}}_{L^2} (1 + \nm{\bar{Z}_{t,\aa}}_{L^2}^2) +\eqref{eq:453}\eqref{eq:496}.  \end{aligned}
\end{equation}
Finally, for the first term on the RHS of \eqref{eq:407}, we use \eqref{eq:251} to replace   $\f{1}{Z_{,\aa}}A_1$ by  $i (\bar{Z}_{tt} - i)$, and apply $(I-\HH) D_\aa^2$ to this, and then use \eqref{lemma-0}: $(I-\HH)(\bar Z_{tt}-Z_tD_\aa \bar Z_t)=0$, \eqref{eq:1026a} and principles no.1 and no.2 of \S\ref{sec:20}. We get
\begin{equation}
  \label{eq:417}
  \begin{aligned}
    &\nm{(I-\HH)D_\aa^2\paren{\f{1}{Z_{,\aa}}A_1}}_{L^2}  = \nm{(I-\HH) D_\aa^2 (Z_t D_\aa \bar{Z}_t)}_{L^2}
\\ & \lec \nm{(I-\HH) \braces{(D_\aa^2 Z_t)D_\aa \bar{Z}_t}}_{L^2}  + \nm{(I-\HH) \braces{(D_\aa Z_t)(D_\aa^2 \bar{Z}_t)}}_{L^2} 
\\ & + \nm{(I-\HH) \braces{\f{Z_t}{Z_{,\aa}}  \partial_\aa  D_\aa^2 \bar{Z}_t}}_{L^2}.
\end{aligned}
\end{equation}
We estimate the first two terms directly by
\begin{equation}
  \label{eq:420}
  \begin{aligned}
      &\nm{(I-\HH) \braces{(D_\aa^2 Z_t)D_\aa \bar{Z}_t}}_{L^2}  + \nm{(I-\HH) \braces{(D_\aa Z_t)(D_\aa^2 \bar{Z}_t)}}_{L^2}  
\\ & \lec \nm{D_\aa \bar{Z}_t}_{L^\infty}(\nm{D_\aa^2 Z_t}_{L^2} + \nm{D_\aa^2 \bar{Z}_t}_{L^2})
 \lec \eqref{eq:502} (\nm{D_\aa^2 \bar{Z}_t}_{L^2} + \eqref{eq:64}).
\end{aligned}
\end{equation}
We are left with the last term on the RHS of \eqref{eq:417}. We first decompose $\f{Z_t}{Z_{,\aa}}$ into its holomorphic and antiholomorphic projections. The term with the holomorphic projection disappears by \eqref{eq:522}.  With what remains, we use \eqref{eq:474} to get a commutator, which we control by commutator estimate \eqref{eq:305}:
\begin{equation}
  \label{eq:429}
  \begin{aligned}
&\nm{(I-\HH) \braces{\f{Z_t}{Z_{,\aa}} \partial_\aa  D_\aa^2 \bar{Z}_t}}_{L^2}  =  \nm{(I-\HH) \braces{\paren{\P_A \f{Z_t}{Z_{,\aa}}} \partial_\aa D_\aa^2 \bar{Z}_t}}_{L^2}
\\ & = \nm{\bracket{\paren{\P_A \f{Z_t}{Z_{,\aa}}},\HH}\partial_\aa D_\aa^2 \bar{Z}_t}_{L^2}
 \lec \nm{\partial_\aa \P_A \f{Z_t}{Z_{,\aa}}}_{L^\infty} \nm{ D_\aa^2 \bar{Z}_t}_{L^2}
\\ & \lec \paren{\nm{D_\aa^2 \bar{Z}_t}_{L^2} + \nm{Z_{t,\aa}}_{L^2} \paren{1 +\nm{\partial_\aa \f{1}{Z_{,\aa}}}_{L^2}}}  \nm{D_\aa^2 \bar{Z}_t}_{L^2}
  \end{aligned}
\end{equation}
by \eqref{eq:513} and \eqref{eq:502}.

We now give the estimate for $\nm{\f{1}{Z_{,\aa}} D_\aa A_1}_{L^\infty}=\nm{\f{1}{\abs{Z_{,\aa}}^2} \partial_\aa A_1}_{L^\infty}$ in \eqref{eq:397}. We do so using \eqref{eq:254}. We have
\begin{equation}\label{eq:1001}
\begin{aligned}
 \frac1{\abs{Z_{,\aa}}^2} \partial_\aa A_1 & = \Im \f{1}{\hdenomconst} \int \f{\pi}{2} \f{(Z_t(\aa) - Z_t(\bb))}{\sin^2(\f{\pi}{2}(\aa - \bb))} \paren{\frac1{\abs{Z_{,\aa}}^2} -\frac1{\abs{Z_{,\bb}}^2}    } \bar{Z}_{t,\bb}(\bb) d\bb 
\\&+ \Im \f{1}{\hdenomconst} \int \f{\pi}{2} \f{(Z_t(\aa) - Z_t(\bb))}{\sin^2(\f{\pi}{2}(\aa - \bb))} \frac1{\bar Z_{,\bb}}D_\bb \bar{Z}_{t}(\bb) d\bb 
\\ & = I + II.
\end{aligned}
\end{equation}
From H\"older's inequality, Hardy's inequality \eqref{eq:77} and the mean value theorem,\footnote{Note that  $\f{1}{\abs{Z_{,\aa}}^2}$ is periodic.}  we have
\begin{equation}\label{eq:1002}
\nm{I}_{L^\infty}\lec \nm{Z_{t,\aa}}_{L^2}^2\nm{D_\aa\frac1{Z_{,\aa}}}_{L^\infty}\lec \nm{Z_{t,\aa}}_{L^2}^2\eqref{eq:453}.
\end{equation}
We rewrite $II$ using integration-by-parts identity \eqref{eq:98}:
\begin{equation}
\label{eq:1003}
\begin{aligned}
 \f{1}{\hdenomconst}& \int \f{\pi}{2} \f{(Z_t(\aa) - Z_t(\bb))}{\sin^2(\f{\pi}{2}(\aa - \bb))} \frac1{\bar Z_{,\bb}}D_\bb \bar{Z}_{t}(\bb) d\bb 
\\ & =-[Z_t, \HH]\partial_\aa \paren{\frac1{\bar Z_{,\aa}}D_\aa \bar{Z}_{t}}+ \HH \paren{Z_{t,\aa}\frac1{\bar Z_{,\aa}}D_\aa \bar{Z}_{t}}
\\&= -[Z_t, \HH]\partial_\aa \paren{\frac1{\bar Z_{,\aa}}D_\aa \bar{Z}_{t}}- \bracket{\f{1}{Z_{,\aa}}\bar{D_\aa \bar{Z}_{t} }, \HH}\bar{Z}_{t,\aa}+\abs{ D_\aa \bar{Z}_{t}}^2.
\end{aligned}
\end{equation}
Using \eqref{eq:32} on the first two terms on the RHS of \eqref{eq:1003}, we get
 \begin{equation}\label{eq:1004}
 \begin{aligned}
 \nm{II}_{L^\infty} & \lec \nm{Z_{t,\aa}}_{L^2}\nm{\partial_\aa \paren{\frac1{\bar Z_{,\aa}}D_\aa \bar{Z}_{t}}}_{L^2}+\nm{ D_\aa \bar{Z}_{t}}_{L^\infty}^2
\\ & \lec \nm{Z_{t,\aa}}_{L^2}\paren{ \nm {D^2_\aa \bar{Z}_{t}}_{L^2}+ \nm{\partial_\aa \frac1{Z_{,\aa}}}_{L^2}\nm{D_\aa \bar{Z}_{t}}_{L^\infty}}+\nm{ D_\aa \bar{Z}_{t}}_{L^\infty}^2.
\end{aligned}
\end{equation}
Combining \eqref{eq:1001}, \eqref{eq:1002}, \eqref{eq:1004}, \eqref{eq:453}, and \eqref{eq:502}, we have
\begin{equation}
\label{eq:1005}
\begin{aligned}
\nm{\frac1{\abs{Z_{,\aa}}^2} \partial_\aa A_1}_{L^\infty} & \le \nm{I}_{L^\infty}+\nm{II}_{L^\infty}
\\ & \le C\paren{ \nm{\bar{Z}_{t,\aa}}_{L^2}, \nm{D_\aa^2 \bar{Z}_t}_{L^2}, \nm{\partial_\aa \f{1}{Z_{,\aa}}}_{L^2}, \nm{D_\aa^2\f{1}{Z_{,\aa}}   }_{L^2}}.
\end{aligned}
\end{equation}
We now sum up these estimates. From \eqref{eq:372}, and using  \eqref{eq:1005} to estimate $\nm{\f{1}{Z_{,\aa}} D_\aa A_1}_{L^\infty}$, we have
\begin{equation}
  \label{eq:489}
  \begin{aligned}
    \nm{\f{1}{Z_{,\aa}}D_\aa^2 A_1}_{L^2} & \lec \eqref{eq:373} + \eqref{eq:379} + \eqref{eq:397} + \eqref{eq:411} + \eqref{eq:420} + \eqref{eq:429}
\\ & \lec 
 C\paren{\nm{D_\aa^2 \f{1}{Z_{,\aa}}}_{L^2},\nm{\bar{Z}_{t,\aa}}_{L^2},\nm{D_\aa^2 \bar{Z}_t}_{L^2},\nm{\partial_\aa \f{1}{Z_{,\aa}}}_{L^2}}.
  \end{aligned}
\end{equation}

Combining \eqref{eq:573} and \eqref{eq:489} we conclude that
\begin{equation}
  \label{eq:454}
  \nm{D_\aa^2 \bar{Z}_{tt}}_{L^2} \lec C\paren{\nm{D_\aa^2 \f{1}{Z_{,\aa}}}_{L^2},\nm{\bar{Z}_{t,\aa}}_{L^2},\nm{D_\aa^2 \bar{Z}_t}_{L^2},\nm{\partial_\aa \f{1}{Z_{,\aa}}}_{L^2}}.
\end{equation}

\subsection{Singularities and the angle of the crest}\label{sec:13a}
In \S \ref{sec:36d} at \eqref{eq:532} and \eqref{eq:575} we characterized our energy in terms of various $L^2$ and $\dot{H}^{1/2}$ norms of quantities in Riemann mapping variables (as well as a single quantity, $\f{1}{Z_{,\aa}}$, in $L^\infty$).  When there is a non-right angle $\ang$ at the corner, or when there is a singularity in the middle of the free surface, the Riemann mapping will have a singularity. In this section, we discuss what it suggests about the angle $\ang$ or the interior angle of an angled crest in the middle of the free surface, continuing the discussion from \S \ref{sec:27}.  

As in \S \ref{sec:27}, we will move the corner at the wall to $0$, and  phrase our discussion in terms of a singularity at the corner, but it applies more broadly to singularities in the middle of the free surface.\footnote{The angled crests in the middle of the free surface don't have to be symmetric.} We thus henceforth focus on the angle $\ang$ at the corner. 

We first observe that our energy is finite in the regime when the interface and velocity are smooth and the angle $\ang = \f{\pi}{2}$, so we can focus on the case where $\ang < \f{\pi}{2}$.\footnote{Recall from the discussion in \S \ref{sec:27}  that we cannot have $\ang > \f{\pi}{2}$ in our energy regime.}

If $\ang$ is the  angle the water wave makes with the wall, the Riemann mapping $\Phi(z)$ should behave like $z^r$ at the corner, where $r \ang = \f{\pi}{2}.$ For $\ang < \f{\pi}{2},$ we have $r > 1$. 

Recall from \eqref{eq:532} and \eqref{eq:575} that among the quantities that  characterize the energy, $\nm{\partial_\aa \f{1}{Z_{,\aa}}}_{L^2}$, $\nm{D_\aa^2 \f{1}{Z_{,\aa}}}_{L^2}$ and $\nm{\f{1}{Z_{,\aa}}}_{L^\infty}$  are the terms purely related to the surface. We will see through the following calculation what non-right angles $\ang$ are allowed if $\nm{\partial_\aa \f{1}{Z_{,\aa}}}_{L^2}$, $\nm{D_\aa^2 \f{1}{Z_{,\aa}}}_{L^2}$ and $\nm{\f{1}{Z_{,\aa}}}_{L^\infty}$ are finite.

Note that 
 $ Z(\aa) = \Phi\i(\aa) \approx (\aa)^{1/r}$
so $ Z_{,\aa} = \partial_\aa(\Phi\i) \approx (\aa)^{1/r - 1}$,
and
\begin{equation}
  \label{eq:672}
  \f{1}{Z_{,\aa}} \approx (\aa)^{1 - 1/r},  \quad
 \quad
\partial_\aa \f{1}{Z_{,\aa}} \approx (\aa)^{-1/r} \quad (r \not= 1).
\end{equation}
Therefore, assuming $r > 1$, $\partial_\aa \f{1}{Z_{,\aa}} \in L^2$ if and only if $r > 2$ if and only if $\ang < \f{\pi}{4}$.  Similarly, $D_\aa^2 \f{1}{Z_{,\aa}} \in L^2$ so long as $r > 2$.\footnote{We remark that, even though $E_a$ (which roughly includes $\nm{D_\aa^2 \f{1}{Z_{,\aa}}}_{L^2}$) is higher-order than $E_b$ (which roughly includes $\nm{\partial_\aa \f{1}{Z_{,\aa}}}_{L^2}$) in terms of the number of derivatives, the two energies are comparable in the sense that they allow precisely the same angles.}

We conclude from this discussion that our energy will be finite only when $\ang < \f{\pi}{4}$ (or when $\ang = \f{\pi}{2}$). This coincides precisely with the angles in the self-similar solutions of \cite{wu:self}.\footnote{We recall that our energy is finite for these solutions.}  For singularities in the middle of the free surface, this suggests that the interior angle  must be less than $\f{\pi}{2}$.\footnote{We note that our energy does not apply to Stokes waves of maximum height (interior angle $ = \f{2\pi}{3}$).}

Now let's consider the behavior of the angle $\ang$ over time. This angle is determined by $\Im \ln z_\a$ at the corner.  Therefore, the behavior of $D_\alpha z_{t}=\partial_t\paren{ \ln z_\alpha}$  
at the corner should determine how the angle changes.  Since $Z_{t,\aa} = \paren{D_\a z_{t}}\circ h\i Z_{,\aa} $ and $ Z_{,\aa} \approx (\aa)^{1/r - 1}$, we must have $D_\a z_{t} \to 0$ at the corner if $Z_{t,\aa} \in L^2$, as our energy assumes. 
This suggests that if initially $\ang<\frac\pi 4$,  the angle would not change while the energy remained finite. This holds true also for interior angles at the angled crests. 

\begin{appendices}

\section{Holomorphicity and mean}\label{sec:hmb}
\subsection{The Hilbert transform \texorpdfstring{$\HH$}{H}}\label{sec:18}
Recall that in \S \ref{sec:11} we introduced the Hilbert transform $\HH$ associated to the periodic domain $P^-$:
\begin{equation}
  \label{eq:440}
  \HH f(\aa) := \f{1}{2 i}  \int_I \cot(\f{\pi}{2}(\aa - \bb)) f(\bb) d\bb,\qquad\text{for }\aa \in [-1,1].
\end{equation}
We know from Proposition~\ref{prop:hilbe} that a function $g\in L^p$ is the boundary value of a periodic holomorphic function on $P^-$ if and only if  $(I-\HH) g=\avg_I g$, and that for any function $f\in L^p$, $(I+\HH)f$ is the boundary value of a periodic holomorphic function on $P^-$, with $(I-\HH)(I+\HH)f=\avg f$.  Recall also that we defined at \eqref{eq:1581} the holomorphic and antiholomorphic projection operators
\begin{equation}
  \label{eq:158}
  \P_A f := \f{(I-\HH)}{2} f;\qquad \P_H f := \f{(I + \HH)}{2} f.
\end{equation}
Here we gather some basic properties of the Hilbert transform $\HH$ that will be used in this paper: 

\begin{proposition}\label{prop:1}

a.  Let $1 < p < \infty$. Then there exists $C_p < \infty$ such that for all $f \in L^p$
  \begin{equation}
    \label{eq:441}
    \nm{\HH f}_{L^p} \le C_p \nm{f}_{L^p}.
  \end{equation}
  
 b.  Let $f \in L^p$ for some $p > 1$. Then
  \begin{equation} \label{eq:159}
     \HH^2 f  = f - \avg f;\qquad  \P_A \P_H f  =  \P_H \P_A f =\f{1}{4} \avg f.
\end{equation}

 c.  Let $f \in L^p$, $g \in L^q$, where $\f{1}{p} + \f{1}{q} = 1$, $1 < p, q < \infty$. Then
  \begin{equation}
    \label{eq:471}
    \int f(\HH g)   = - \int (\HH f) g; \qquad
  \int (\P_A f) g  = \int f (\P_H g).
\end{equation}

d. Let $f \in L^p, g \in L^q$ for $p>1$, $q>1$, $\f{1}{p} + \f{1}{q} < 1$.  Then
\begin{equation}
  \label{eq:160}
  \P_A \braces{(\P_H f) (\P_H g)}  = \f{1}{8} \paren{\avg f} \paren{\avg g}.
\end{equation}
\end{proposition}

Parts a., b., and c.~are classical results. We know the product of periodic holomorphic functions is periodic holomorphic. Part d.~is an easy consequence of Proposition~\ref{prop:hilbe}, parts b.~and c., and the fact that $(\P_H f) (\P_H g)$ is the boundary value of a periodic holomorphic function on $P^-$.


We will also need the following proposition.
\begin{proposition}\label{prop:iden}
  Let $f \in L^\infty$, $g\in L^p$ for some $p> 1$.  Suppose $(I-\HH) f = \avg f$. Then
  \begin{equation}
    \label{eq:19}
    [f,\HH]g = \f12(I+\HH) \paren{[f,\HH]g} - \frac12\avg (fg) + \frac12\paren{\avg f} \paren{\avg g}.
  \end{equation}
\end{proposition}
\begin{proof}
We begin by observing that
\begin{equation}
  \label{eq:634}
 \bracket{f,\HH} g 
     = f(I+\HH)g - (I+\HH) (fg).
  \end{equation}
  Because $(I-\HH) f = \avg f$, by Proposition~\ref{prop:hilbe}, $[f,\HH]g$ is the boundary value of a periodic holomorphic function on $P^-$.  Therefore,
\begin{equation}
  \label{eq:635}
  \begin{aligned}
    (I-\HH)([f,\HH]g) &=\avg [f,\HH]g  =   \avg f(I+\HH)g-\avg fg \\& =  \avg ((I-\HH) f) g-\avg( fg)    = \paren{\avg f} \paren{\avg g} - \avg (fg)
  \end{aligned}
\end{equation}
by \eqref{eq:471}.
\end{proof}

\subsection{Periodic holomorphic functions 
}\label{sec:20}

 In this section, we note which of the functions we are dealing with are the boundary values of periodic holomorphic functions---and which, further, have mean zero. 
 From  Proposition~\ref{prop:hilbe} we know that, to show that $(I-\HH)$ of various functions disappears, assume they are in $L^p$, $p\ge1$,  it suffices  to show that they are boundary values of periodic holomorphic functions and that their means are zero.
 
 We start with the following basic facts. Assume that the quantities involved are sufficiently smooth, and that the assumption \eqref{riemanntime} holds.

First, we have that the complex conjugate velocity is periodic holomorphic, and goes to zero as $y \to -\infty$ by \eqref{eq:620}, so
\begin{equation}
  \label{eq:571}
  (I-\HH) \bar{Z}_t = 0.
\end{equation}

Then we have three identities about the Riemann mapping. Recall that
\begin{equation}
  \label{eq:627}
  Z_{,\aa} = \partial_\aa \Phi\i(\aa,t);\qquad \f{1}{Z_{,\aa}} = \Phi_z \circ Z.
\end{equation}
Both of these are clearly periodic holomorphic. Therefore by \eqref{riemannmean} we have
\begin{equation}
  \label{eq:516}
  (I-\HH) \f{1}{Z_{,\aa}} = 1,\qquad (I-\HH) Z_{,\aa} = 1.
\end{equation}
The mean $\avg Z_{,\aa} = 1$ can also be checked directly by the fundamental theorem of calculus, since $Z(1,t) = 1, Z(-1,t) =-1$ for all time.

Finally, we have
\begin{equation}
  \label{eq:312a}
  (I-\HH) \braces{\Phi_t \circ Z} = 0
\end{equation}
by \eqref{riemanntime}. Here, we have $\Phi_t \circ \Phi\i$ is holomorphic because it is the limit of holomorphic functions, and we know that $\Phi_t$ is periodic by the periodicity of our domain $\Omega(t)$.

From these facts, we will be able to deduce everything else that we need in the derivation of the free surface equations, from the following principles:

1. Assume that $f\in C^0(S^1)$ with $\partial_\aa f\in L^p$, $p\ge 1$. If $(I-\HH)f = c$, then $(I-\HH) \partial_\aa f = 0$.

This is straightforward by taking derivatives. One may also use the  correspondence between periodic holomorphic functions on $P_-$ and holomorphic functions on the unit disc, and take derivative to conclude.

2.  Assume that $f\in L^p$, $g\in L^q$ and $fg\in L^r$ with $p,\ q,\ r\ge 1$. If $(I-\HH) f = 0$ and $(I-\HH) g = c$ then $(I-\HH) (fg) = 0$.

 This is because the product of  periodic holomorphic functions is periodic holomorphic. If one of the factors goes to $0$ as $y \to -\infty$, then the product goes to $0$  as $y\to-\infty$.
 
3.  Assume $G(z,t)$ is a periodic holomorphic function on $\Omega(t)$ going to zero as $y \to -\infty$, 
$G$ is continuous differentiable with respective to $t$, and $G_t\circ Z\in L^q$, $G\circ Z\in C^0(S^1)$, $\partial_\aa \paren{G\circ Z}\in L^p$, $q, p\ge 1$, so $(I-\HH) (G \circ Z) = 0$.
Then $(I-\HH) G_t\circ Z = 0$.\footnote{Note that this argument does not apply to $\Phi_t \circ Z$ itself, because $\Phi$ is not periodic.} 

It is clear that $G_t \circ Z$ is periodic holomorphic, since it is the limit of periodic holomorphic functions.  It remains to show that $\avg G_t \circ Z = 0$. Note that $\avg G \circ Z = 0$ for all time. Also, since $\Phi(\Phi^{-1}(\aa, t), t)=\a'$, we have $(\Phi_z\circ Z) \cdot(\Phi^{-1})_t + \Phi_t\circ Z=0$, and therefore
\begin{equation}\label{eq:588}
(\Phi^{-1})_t= (-Z_{,\aa})  \Phi_t\circ Z.
\end{equation}
Therefore, using \eqref{eq:588} and the fact $Z=\Phi^{-1}$, we have
  \begin{equation}
    \label{eq:57}
    \begin{aligned}
      0 & = \f{d}{dt} \avg G(Z(\aa,t),t) 
       = \avg G_t(Z(\aa,t),t) +\avg G_z \circ Z\, (\Phi^{-1})_t(\aa,t)
      \\ & = \avg G_t \circ Z - \avg (\partial_\aa (G \circ Z)) \Phi_t\circ \Phi^{-1}.
    \end{aligned}
  \end{equation}

The second integral is zero by principles no. 1 and no.2 above.

Let $F=\bar{\vec{v}}$ be the complex conjugate velocity. Note that $F_{t} \circ Z= \bar{Z}_{tt} - (D_\aa \bar{Z}_t) Z_t$ as shown in \S \ref{sec:28}. As immediate consequences of  the basic facts and the three principles above,  we have the following statements: 

\begin{equation}
  \label{eq:186a}
  (I-\HH) \braces{(Z_{,\aa})^j (F_t \circ Z)} = 0, \qquad\text{for } j=0,1
\end{equation}
\begin{equation}
  \label{eq:504a}
  (I-\HH) \braces{(Z_{,\aa})^j (F_{tt} \circ Z) }= 0, \qquad\text{for } j=0,1.
\end{equation}
\begin{equation}
  \label{eq:474a}
 (I-\HH) (\partial_\aa^j D^k_\aa \bar{Z}_t )= 0,\qquad \text{for }j, k=0, 1.
\end{equation}
\begin{equation}
    \label{eq:222a}
    (I-\HH) \partial_\aa  (\bar{Z}_{tt} - (D_\aa \bar{Z}_t) Z_t) = 0.
\end{equation}

\subsubsection{Some identities used in the proof of Theorem~\ref{maintheorem}}\label{sec:21a}
In this and the following subsection \S\ref{sec:21b},  we assume only the assumption in Theorem~\ref{maintheorem} and its consequence \eqref{assumption2} hold. 
From  \eqref{eq:1026} and \eqref{eq:1026a}:
$$(I-\HH)\bar Z_t=0,\qquad (I-\HH)\frac1{Z_{,\aa}}=1,$$
and the above principles no.1 and no.2, we have
\begin{equation}
  \label{eq:242}
  (I-\HH) \partial_\aa \f{1}{Z_{,\aa}} = 0
\end{equation}
\begin{equation}
  \label{eq:315}
  (I-\HH) D_\aa^k \bar{Z}_t = 0,\qquad \text{for } 0\le k\le 3
\end{equation}
\begin{equation}
  \label{eq:474}
 (I-\HH) \partial_\aa D^k_\aa \bar{Z}_t = 0,\qquad \text{for } 0\le k\le 2
\end{equation}
\begin{equation}
  \label{eq:456}
  (I-\HH) \braces{\f{1}{Z_{,\aa}} D_\aa^k \bar{Z}_t} = 0, \qquad \text{for } 0\le k \le 2
\end{equation}
\begin{equation}
  \label{eq:522}
  (I-\HH) \braces{(\partial_\aa^j D_\aa^2 \bar{Z}_t)\partial_\aa^l\paren{\paren{\P_H \f{Z_t}{Z_{,\aa}}}^k}} = 0,\quad\text{for }j=1,2, k=1,2, l=0,1.
\end{equation}
By \eqref{lemma-0},  $(I-\HH)(\bar Z_{tt}-Z_t D_\aa\bar Z_t)=0$. This then gives, by an application of principles no.1 and no.2 above, 
\begin{equation}
  \label{eq:313}
  (I-\HH) D_\aa^k  (\bar{Z}_{tt} - (D_\aa \bar{Z}_t) Z_t) = 0,\quad \text{for } 0\le k\le 3.
\end{equation}

\subsubsection{Mean conditions}\label{sec:21b}
We have implicitly in the preceding section shown that various quantities are mean-zero, but we don't use that fact other than in those identities. 
We will at one point require an explicit mean-zero condition, to use a variant of the Sobolev inequality.  This is that
\begin{equation}
  \label{eq:48}
  \avg (D_\aa \bar{Z}_t)^2 d\aa = 0.
\end{equation}
We note that 
because $(I-\HH) D_\aa \bar{Z}_t = 0$ \eqref{eq:315}, $D_\aa \bar{Z}_t$ is the boundary value of a periodic holomorphic function going to $0$ as $y \to -\infty$, so its square will also be the boundary value of a periodic holomorphic function going to $0$ as $y \to -\infty$.

\section{Useful inequalities and identities}\label{sec:tech}
We present here assorted inequalities and identities that we need in our paper. None of the results here are original, so we omit the proofs in most cases. We refer the reader to \cite{kinsey} for details of the proofs.

\subsection{Sobolev inequalities}\label{sec:sob}
We present here the one-dimensional Sobolev inequality we  use in our proof. 

\begin{proposition}[Weighted Sobolev Inequality with $\e$]\label{sobolev-with-weight}
  Let $\e > 0$. Then for all $f \in C^1(-1,1)\cap L^2(-1,1)$,
  \begin{equation}
    \label{eq:396}
    \nm{f}_{L^\infty} \lec \f{1}{\e} \nm{f}_{L^2(\f{1}{\omega})} + \e \nm{f'}_{L^2(\omega)} + \nm{f}_{L^2}
  \end{equation}
for any weight $\omega \ge 0$.

Furthermore
\begin{equation}
  \label{eq:638}
  \avg f^2 = 0 \To \nm{f}_{L^\infty} \lec \f{1}{\e} \nm{f}_{L^2(\f{1}{\omega})} + \e \nm{f'}_{L^2(\omega)}.
\end{equation}
\end{proposition}

We omit the proof since it is fairly standard.

\subsection{Derivatives and complex-valued functions}\label{sec:comp}

Because our functions will be complex-valued, and we will often be looking at derivatives of angular and modular parts of these functions, we note here a few elementary facts about such functions.

 Let $f(\alpha) = r(\alpha) e^{i\th(\alpha)}$, where $r$, $\theta$ are real-valued functions. Then
\begin{equation}
  \label{eq:100}
 \abs{\partial_\a \abs{f}} \le \abs{f'};\qquad \quad
  \abs{\partial_\a \f{f}{\abs{f}}} \le \abs{\f{f'}{\abs{f}}}.
\end{equation}

From $-\frac{\partial P}{\partial\vec{n}}\ge 0$, we have $\af\ge 0$ and 
\begin{equation}
  \label{eq:279}
    \f{z_{tt} + i}{\abs{z_{tt} + i}} = i \f{z_\a}{\abs{z_\a}}.
\end{equation}
We will use this fact in this paper to replace the angular part of the spatial derivative $z_\a$ with the angular part of the time derivative $z_{tt} + i$.

\subsection{Hardy's inequality and commutator estimates }\label{sec:31}
We present here the basic estimates we will rely on for this paper.  Several of these estimates control quantities of the form $[f,\HH] g'$ by something involving $f'$ and $g$; they thus reduce the amount of regularity required on $g$, at the expense of further regularity on $f$.

For many of these estimates, we must pay close attention to the boundary conditions. Recall that $f \in C^0(S^1)$ implies that $\bound{f} = 0$.  Many of these estimates do not hold if this periodic boundary condition is removed. To save space, we have not always explicitly cited these boundary conditions when we quote these estimates, but they are always met, by our assumption of Theorem~\ref{maintheorem} and its consequence \eqref{assumption2}.

\begin{proposition}[Hardy's Inequality]
\label{hardy-inequality}  
Let $f \in C^0(S^1) \cap C^1(-1,1)$, 
with $f' \in L^2$. Then there exists $C > 0$ independent of $f$ such that for any $\aa \in I$,
\begin{equation}
  \label{eq:77}
\abs{\int_I \f{(f(\aa) - f(\bb))^2}{\sin^2(\f{\pi}{2}(\aa-\bb))} d\bb} \le C \nm{f'}_{L^2}^2.
\end{equation}
\end{proposition}
\begin{proposition}[$L^2 \times L^\infty$ Estimate]\label{prop-l2-linfty-hh}
 There exists a constant $C > 0$ such that for any $f \in C^0(S^1) \cap C^1(-1,1)$ with $f' \in L^2$, $g \in C^0[-1,1]$ with $g' \in L^p$ for some $p > 1$\footnote{
 We require this only to ensure that $[f,\HH] g'$ is well-defined.} (so $\bound{f} = 0$, though possibly $\bound{g} \not= 0$),
\begin{equation}
  \label{eq:228}
  \nm{[f,\HH]\partial_\aa g}_{L^2} \le C \nm{f'}_{L^2} \nm{g}_{L^\infty}.
\end{equation}
\end{proposition}

\begin{proof}
This result is the periodic modification of a result from \cite{wu2009}, which in turn is a consequence of the $T(b)$ theorem \cite{david-journe-semmes}.

We begin by integrating by parts,
\begin{equation}
  \label{eq:98}
[f,\HH] \partial_\aa g = \HH (f'g) - \f{1}{2i} \int \f{\pi}{2}\f{f(\aa) - f(\bb)}{\sin^2(\f{\pi}{2}(\aa - \bb))} g(\bb) d\b + \bound{\f{1}{2i} (f(\aa) - f(\bb))\cot(\f{\pi}{2}(\aa - \bb)) g(\bb)},
\end{equation}
where we  have a boundary term because we didn't place periodic boundary assumption on $g$.
We control the first term by the $L^2$ boundedness of $\HH$ and H\"older. We control the last term by Hardy's inequality \eqref{eq:77}.

We handle the second term by using the identity
\begin{equation}\label{eq:286}
\paren{\f\pi 2}^2\f1{\sin^2(\f \pi 2 \aa)}=\sum_{l\in \mathbb Z}\f 1{(\aa+2l)^2}
\end{equation}
 and reducing it to  the real line version, Proposition 3.2 in \cite{wu2009}. For details, see \cite{kinsey}.
 
\end{proof}

\begin{proposition}[$L^2 \times L^\infty$ Estimate Variant] 
 There exists a constant $C > 0$ such that for any $f \in C^0(S^1) \cap C^1(-1,1)$ with $f' \in L^2$, $g\in L^\infty[-1,1]$ (and so  $\bound{f} = 0$ but possibly $\bound{g} \not=0$), 
  \begin{equation}
    \label{eq:258}
    \nm{ \int \f{(f(\aa) - f(\bb))}{\sin^2(\f{\pi}{2}(\aa - \bb))} g(\bb) d\bb}_{L^2} \le C \nm{f'}_{L^2} \nm{g}_{L^\infty}.
  \end{equation}
\end{proposition}

This is just the second term on the RHS of \eqref{eq:98}.

\begin{proposition}[$L^\infty \times L^2$ Estimate \cite{calderon}]\label{prop-linfty-l2-hh}
 There exists a constant $C > 0$ such that for any $f \in C^1[-1,1] \cap C^0(S^1)$, $g \in C^0(S^1) \cap C^1(-1,1)$ with $g' \in L^p$ for some $p > 1$\footnote{We assume $g' \in L^p$ only to ensure $[f,\HH] g'$ is well-defined.}, 
\begin{equation}
  \label{eq:305}
  \nm{[f,\HH] \partial_\aa g}_{L^2} \le C \nm{f'}_{L^\infty} \nm{g}_{L^2}.
\end{equation}
\end{proposition}
To prove Proposition~\ref{prop-linfty-l2-hh}, we begin with the integration-by-part formula \eqref{eq:98}, noting that the third term, the boundary term, is zero since $\bound{f}=\bound{g}=0$. We handle the first term by the $L^2$ boundedness of $\HH$ and H\"older, and the second term by identity \eqref{eq:286}, reducing it to 
  the classical result on $\R$ by \cite{calderon}.\footnote{The result was later extended by \cite{coifman-meyer}.}

\begin{proposition}\label{l2-half-deriv-smoothing}
 There exists a constant $C > 0$ such that for any $f, g \in C^1(-1,1) \cap C^0(S^1)$ with $f' \in L^2$ and $g' \in L^p$ for some $p>1$, 
\begin{equation}
  \label{eq:21}
  \nm{[f,\HH] \partial_\aa g}_{L^2} \le C \nm{f'}_{L^2} \nm{g}_{\dot{H}^{1/2}}.
\end{equation}
\end{proposition}

\begin{proof}
  We integrate by parts, first rewriting $\partial_\bb g(\bb) = \partial_\bb (g(\bb) - g(\aa))$:
  \begin{equation}
    \label{eq:13}
    [f,\HH] \partial_\aa g = \f{1}{\hdenomconst} \int f_\bb(\bb)  \cot(\f{\pi}{2}(\aa - \bb))  (g(\bb) - g(\aa)) d\bb - \f{1}{\hdenomconst} \int \f{\pi}{2}\f{f(\aa) - f(\bb)}{\sin^2(\f{\pi}{2}(\aa -\bb))} (g(\bb) - g(\aa)) d\bb,
  \end{equation}
where there is no boundary term because of the periodic boundary conditions.

For the first term, we apply Cauchy-Schwarz inequality:
\begin{equation}
  \label{eq:14}
  \abs{\int f'(\bb) \cot(\f{\pi}{2}(\aa - \bb)) (g(\bb) - g(\aa)) d\bb} \le \nm{f'}_{L^2} \paren{\int \abs{g(\aa) - g(\bb)}^2 \abs{\cot^2(\f{\pi}{2}(\aa - \bb))} d\bb}^{1/2}.
\end{equation}
Taking $L^2$ of this in $\aa$ and using the boundedness of cosine to replace $\cot^2$ with $\f{1}{\sin^2}$, we get the desired estimate.

For the second term, we use Cauchy-Schwarz inequality:
\begin{equation}
  \label{eq:22}
  \nm{ \int \f{f(\aa) - f(\bb)}{\sin^2(\f{\pi}{2}(\aa - \bb))} (g(\bb) - g(\aa)) d\bb}_{L^2_\aa} \le \paren{\int \int \f{\abs{f(\aa) - f(\bb)}^2}{\sin^2(\f{\pi}{2}(\aa - \bb))} d\bb \int \f{\abs{g(\aa) - g(\bb)}^2}{\sin^2(\f{\pi}{2}(\aa - \bb))} d\bb d\aa}^{1/2},
\end{equation}
and then we use Hardy's inequality \eqref{eq:77} to $f$ to get our inequality \eqref{eq:21}.
\end{proof}

\begin{proposition}
  There exists a constant $C > 0$ such that for any  $f \in \dot{H}^{1/2},  g \in L^2$, 
\begin{equation}
  \label{eq:96}
  \nm{[f,\HH]g}_{L^2} \le C \nm{f}_{\dot{H}^{1/2}} \nm{g}_{L^2}.
\end{equation}
\end{proposition}
\begin{proof}
  This is immediate by Cauchy-Schwarz inequality and the boundedness of cosine.
\end{proof}
\begin{proposition}
 There exists a constant $C > 0$ such that for any $f, g \in C^1(-1,1) \cap C^0(S^1)$ with $f', g' \in L^2$ and $h \in L^2$, 
  \begin{equation}
    \label{eq:113}
    \nm{[f,g;h]}_{L^2} := \nm{\frac\pi{4i}\int \f{f(\aa) - f(\bb)}{\sin(\f{\pi}{2}(\aa - \bb))} \f{g(\aa) - g(\bb)}{\sin(\f{\pi}{2}(\aa - \bb))} h(\bb) d\bb}_{L^2} \le C \nm{f'}_{L^2} \nm{g'}_{L^2} \nm{h}_{L^2}.
  \end{equation}
\end{proposition}

\begin{proof}
By Cauchy-Schwarz inequality,
  \begin{multline}
    \label{eq:114}
    \abs{\int \f{f(\aa) - f(\bb)}{\sin(\f{\pi}{2}(\aa - \bb))} \f{g(\aa) - g(\bb)}{\sin(\f{\pi}{2}(\aa - \bb))} h(\bb) d\bb} \\\le \paren{\int \abs{\f{f(\aa) - f(\bb)}{\sin(\f{\pi}{2}(\aa - \bb))}}^2 d\bb}^{1/2} \paren{ \int  \abs{\f{g(\aa) - g(\bb)}{\sin(\f{\pi}{2}(\aa - \bb))} h(\bb)}^{2} d\bb}^{1/2}.
  \end{multline}
Now we take the $L^2$ of this in the $\alpha'$ variable. By Hardy's inequality \eqref{eq:77}, we control the $f$ factor by $\nm{f'}_{L^2}$, and are left with
\begin{equation}
  \label{eq:115}
  \paren{\int  \int  \abs{\f{g(\aa) - g(\bb)}{\sin(\f{\pi}{2}(\aa - \bb))} h(\bb)}^{2} d\bb d\aa}^{1/2}.
\end{equation}
Applying Fubini's Theorem and then using Hardy's inequality \eqref{eq:77} once more gives the result.
\end{proof}

\begin{proposition}
 There exists a constant $C > 0$ such that for any $f \in C^1(-1,1) \cap C^0(S^1)$ with $f' \in L^2$, $g \in L^2$,  
  \begin{equation}
    \label{eq:32}
    \nm{[f,\HH] g}_{L^\infty} \le C \nm{f'}_{L^2} \nm{g}_{L^2}.
  \end{equation}
 \end{proposition}
\begin{proof}
  Estimate \eqref{eq:32} holds by Cauchy-Schwarz inequality and Hardy's inequality \eqref{eq:77}.
\end{proof}

\begin{proposition}
 There exists a constant $C > 0$ such that for any $f, g \in C^1(-1,1) \cap C^0(S^1)$ with $f', g' \in L^2$, and $h \in L^2$, 
\begin{equation}
  \label{eq:144}
  \nm{\partial_\aa [f,[g,\HH]] h}_{L^2} \le C \nm{f'}_{L^2} \nm{g'}_{L^2} \nm{h}_{L^2}.
\end{equation}
\end{proposition}

\begin{proof}
We differentiate:
\begin{multline}
  \label{eq:363}
  \partial_\aa  \f{1}{\hdenomconst} \int (f(\aa) - f(\bb)) (g(\aa) - g(\bb)) \cot(\f{\pi}{2}(\aa -\bb)) h(\bb) d\bb
\\ = f' [g,\HH] h + g'[f,\HH] h  -  \f{1}{\hdenomconst} \int \f{\pi}{2}\f{(f(\aa) - f(\bb))(g(\aa) - g(\bb))}{\sin^2(\f{\pi}{2}(\aa -\bb))} h(\bb) d\bb.
\end{multline}
We control the first two terms by H\"older's inequality and then \eqref{eq:32}. We control the last term by \eqref{eq:113}.
\end{proof}

\begin{proposition}[Higher-Order Calderon Commutator \cite{coifman-meyer}]
There exists a constant $C > 0$ such that for any $f \in C^1(-1,1) \cap C^0(S^1)$ with $f' \in L^\infty$, $h \in C^0(S^1) \cap C^1(-1,1)$ with $h' \in L^p$ for some $p > 1$, 
  \begin{equation}
  \label{eq:6a}
  \nm{[f,f; \partial_\aa h]}_{L^2} \le C \nm{f'}_{L^\infty}^2 \nm{h}_{L^2}.
\end{equation}
\end{proposition}
\begin{proof}
The proof is entirely analogous to the proof of \eqref{eq:305}, and now follows from the work of \cite{coifman-meyer}, which extends the original result of \cite{calderon} used for \eqref{eq:305} to allow two difference quotient factors, instead of one. To move from $\R$ to our compact domain, we do the same infinite summation argument, with the only difference being that instead of \eqref{eq:286},  we use $(\f\pi 2)^2\partial_\aa \f{1}{\sin^2(\f{\pi}{2}\aa)} =  -2\sum \f{1}{(\aa + 2l)^3}$.
\end{proof}

\subsection{The \texorpdfstring{$\dot{H}^{1/2}$}{dot H 1/2} norm}\label{sec:half}
We present here the following proposition.
\begin{proposition}\label{prop:half}
  Let $f \in C^1(-1,1) \cap C^0(S^1)$ with $f' \in L^2$. Then 
\begin{equation}
  \label{eq:549}
  (I-\HH)f = \avg f \To \nm{f}_{\dot{H}^{1/2}}^2 =  \int i (\partial_\aa f) \bar{f} d\aa.
\end{equation}
\end{proposition}

\eqref{eq:549} holds because for $f$ satisfying the assumption of Proposition~\ref{prop:half}, 
\begin{equation}
  \label{eq:26}
  \int i (\partial_\aa f) \bar{f} d\aa= \int i (\partial_\aa \HH f) \bar{f} d\aa, 
\end{equation}
and $i \,\partial_\aa \HH f=|D|f$, where $|D|$ is the positive operator satisfying $|D|^2=-\partial^2_\aa$.
It is easy to see that $\int i (\partial_\aa f) \bar{f} d\aa$ is real-valued by integration by parts.

\subsection{Commutator identities}\label{sec:33}
We include here for reference the various commutator identities that are necessary.
\begin{align}
  \label{eq:120}
  [\partial_t,D_\a] &= - (D_\a z_t) D_\a;\\ 
  \label{eq:121}
    \bracket{\partial_t,D_\a^2} &   = [\partial_t,D_\a] D_\a + D_\a [\partial_t,D_\a]
        = -2(D_\a z_t) D_\a^2 - (D_\a^2 z_t) D_\a;\\
  \label{eq:212}
     \bracket{\partial_t^2,D_\a} &=\partial_t\bracket{\partial_t,D_\a} + \bracket{\partial_t,D_\a}\partial_t= (-D_\a z_{tt}) D_\a + 2(D_\a z_t)^2 D_\a - 2(D_\a z_t) D_\a \partial_t.
  \end{align}
To calculate $[i\mathfrak{a}\partial_\a,D_\a]$, we use $i\mathfrak{a} z_\a = z_{tt} + i$ \eqref{eq:8} to rewrite $i\mathfrak{a} \partial_\a = i \mathfrak{a} z_\a D_\a = (z_{tt} + i) D_\a$. Therefore
\begin{equation}
  \label{eq:47}
  [i\mathfrak{a}\partial_\a,D_\a] = [(z_{tt}+i)D_\a,D_\a] = -(D_\a z_{tt}) D_\a.
\end{equation}
Adding \eqref{eq:212} and \eqref{eq:47}, we conclude that
\begin{equation}
  \label{eq:53}
   \bracket{\partial_t^2 + i \mathfrak{a} \partial_\a, D_\a} =  (-2D_\a z_{tt}) D_\a + 2(D_\a z_t)^2 D_\a - 2(D_\a z_t) D_\a \partial_t.
\end{equation}

Because $[(\partial_t^2 + i \af \partial_\a),D_\a^2] = [(\partial_t^2 + i \af \partial_\a),D_\a] D_\a + D_\a[(\partial_t^2 + i \af \partial_\a),D_\a],$ we have
\begin{equation}
  \label{eq:52}
  \begin{aligned}
    \bracket{(\partial_t^2 + i \af \partial_\a),D_\a^2} & = 
(-4D_\a z_{tt}) D_\a^2 + 4(D_\a z_t)^2 D_\a^2 - 2(D_\a z_t) D_\a \partial_t D_\a - (2D_\a^2 z_{tt}) D_\a
\\ & + 4(D_\a z_t) (D_\a^2 z_t) D_\a - 2(D_\a^2 z_t) D_\a \partial_t - 2(D_\a z_t) D_\a^2 \partial_t.
\end{aligned}
\end{equation}

\section{Summary of notation}\label{notation}
We list here the various notations we've introduced in the paper. See also \S \ref{sec:5} for a discussion of the conventions used.

\begin{itemize}
\item $\bound{f} := f(1,t) - f(-1,t)$.
\item  $\ang$ is the angle the water wave makes with the wall. See Figure~\ref{fig:b}.
\item $I := [-1,1]$ (except when it's used for the identity or as an abbreviation for a quantity to be controlled).
\item $\Re z$, $\Im z$ are the real and imaginary parts of a complex number $z$.
\item  Function spaces $C^k(-1,1), C^k[-1,1], C^k(S^1)$, $H^k(S^1)$,  $L^p$, and $\dot{H}^{1/2}$    are defined in \S \ref{sec:5}. We define $\nm{f}_{\dot{H}^{1/2}} := \paren{\frac\pi 8\iint \f{\abs{f(\aa) - f(\bb)}^2}{\sin^2(\f{\pi}{2}(\aa -\bb))} d\aa d\bb}^{1/2}$.
\item $\avg f = \avg_I f  := \f{1}{2} \int_I f(\bb) d\bb$.
\item $z(\a,t)$ is the Lagrangian parametrization, and $z_t(\a,t) = \vec{v}(z(\a,t),t)$ is the velocity. $z_{tt} = \partial_t z_t$, etc.
\item $\af = \f{\abs{\bar{z}_{tt} - i}}{\abs{z_\a}} = -\f{\partial P}{\partial \vec{n}} \f{1}{\abs{z_\a}}$. 
We refer to $-\f{\partial P}{\partial \vec{n}}$ as the Taylor coefficient; $\vec{n}$ is the outward-facing normal to $\Sigma(t)$.
\item $h: \a \mapsto \aa$ is defined by $h(\a) = \Phi(z(\a,t),t)$ and gives the Riemann mapping variables, where $\Phi$ is the Riemann mapping defined in \S \ref{sec:12a}.  $\partial_\aa (f \circ h\i) = \f{\partial_\a f}{h_\a} \circ h\i$.  $d\aa = h_\a d\a$. Full details are in \S \ref{sec:12a}.
\item $\alpha$ and $\beta$ are our variables in Lagrangian coordinates; $\aa$ and $\bb$ are our variables in Riemann mapping coordinates.
\item $\HH$ is the Hilbert transform in Riemann mapping variables, defined by
  \begin{equation}
    \label{eq:665}
    \HH f(\aa) := \f{1}{\hdenomconst}  \int_I \cot(\f{\pi}{2}(\aa-\bb)) f(\bb) d\bb.
  \end{equation}
\item We define $\P_A := \f{(I-\HH)}{2}$ and $\P_H := \f{(I+\HH)}{2}$ as the antiholomorphic and holomorphic projections.
\item $[f,g;h](\aa) := \frac\pi{4i}\int \f{(f(\aa) - f(\bb))(g(\aa) - g(\bb))}{\sin^2(\f{\pi}{2}(\aa - \bb))} h(\bb) d\bb$, the higher-order Calderon commutator.
\item We use $F(z(\a,t),t) := \bar{z}_t(\a,t)$ at several points (and do not use $F$ for any other purpose). 
\item $Z := z \circ h\i, Z_t := z_t \circ h\i, Z_{tt} := z_{tt} \circ h\i$, $Z_{,\aa} = \partial_\aa(z \circ h\i), Z_{t,\aa} = \partial_\aa(z_t \circ h\i)$, etc.  
\item Compositions and inverses are always with respect to the spatial variable.
\item $\AA := (\af h_\a) \circ h\i$, $\AAt := (\af_t h_\a) \circ h\i$.
 \item $D_\a := \f{1}{z_\a} \partial_\a$, $\abs{D_\a} := \f{1}{\abs{z_\a}} \partial_\a$, $D_\aa := \f{1}{\abs{Z_{,\aa}}} \partial_\aa$, $\abs{D_\aa} := \f{1}{\abs{Z_{,\aa}}} \partial_\aa$.
\item $A_1 = \AA \abs{Z_{,\aa}}^2 = i Z_{,\aa}(\bar{Z}_{tt} - i) \in \R$ \eqref{eq:208}. On changing variables, we have 
  \begin{equation}
    \label{eq:663}
    A_1 \circ h = \f{\af \abs{z_\a}^2}{h_\a},
  \end{equation}
  originally derived at \eqref{eq:209}; we use this repeatedly without citation.  We  often use $A_1 \ge 1$ \eqref{eq:394}, and also use $\f{1}{Z_{,\aa}} = i \f{\bar{Z}_{tt} - i}{A_1}$ \eqref{eq:251}.
\item We define our energies in \S \ref{sec:35}. We define generic energies $E_{a,\th}$ and $E_{b,\th}$, and then specialize to $E_a := E_{a,D_\a^2 \bar{z}_t}$ and $E_b := E_{b,D_\a \bar{z}_t}$.  We use $G_\th$ to describe the RHS of the equation $(\partial_t^2 + i \af \partial_\a)\th  = G_\th$. For $\th = D_\a^k \bar{z}_t$, $G_\th = D_\a^k(-i\mathfrak{a}_t \bar{z}_\a) + [\partial_t^2 + i \mathfrak{a}\partial_\a,D_\a^k] \bar{z}_t.$
\item $\Th := \th \circ h\i$;  $B := \paren{\f{h_{t\a}}{h_\a} - \Re D_\aa z_t} \circ h\i$; $\psi := \paren{\f{h_\a}{z_\a} \th} \circ h\i$ \eqref{eq:83}.
\item See \S \ref{sec:5} for a discussion of how broadly $I, II, I_1, I_{12}$, etc., are defined. In short, they are unambiguous within each section, but ambiguous between sections. \item We use $C(E)$ to represent a polynomial of the energy $E$.
\end{itemize}

\section{Main quantities controlled} \label{quantities}
We list here the various quantities that are controlled by our energy, for ease of reference.  We don't list every single quantity we have controlled, but we do include any quantities that we give at the end of a concluding inequality without further explanation.

\begin{itemize}
\item 
In \S \ref{sec:quants}, we controlled 
\begin{equation}\label{eq:1550}
\begin{aligned}
&\nm{(\partial_t+\frak b\partial_\aa)D_\aa^2\bar Z_t}_{L^2}, \nm{D_\aa^2 \bar{Z}_{tt}}_{L^2}, \nm{D_\aa^2 {Z}_{tt}}_{L^2},  \nm{D_\aa^2 \bar{Z}_t}_{L^2}, \nm{D_\aa^2 {Z}_t}_{L^2}, \\&\nm{D_\a \partial_t D_\a \bar{z}_t}_{L^2({h_\a}d\a)}, \nm{\f{1}{Z_{,\aa}} D_\aa^2 \bar{Z}_t}_{\dot{H}^{1/2}}, \nm{D_\aa \bar{Z}_{tt}}_{L^\infty}, \nm{D_\aa {Z}_{tt}}_{L^\infty},  \nm{D_\aa \bar{Z}_t}_{L^\infty}, \nm{D_\aa {Z}_t}_{L^\infty},\\&  \nm{ \bar{Z}_{tt, \aa}}_{L^2},   \nm{ \bar{Z}_{t, \aa}}_{L^2},  \int \abs{D_\a \bar{z}_t}^2 \f{d\a}{\af},\  \int \abs{D_\a \bar{z}_{tt}}^2 \f{d\a}{\af},  \ \nm{\f{1}{Z_{,\aa}}}_{L^\infty}, \nm{Z_{tt}+i}_{L^\infty} , \nm{A_1}_{L^\infty}.
\end{aligned}
\end{equation}

\item $\nm{\f{\af_t}{\af}}_{L^\infty} = \nm{\f{\AAt}{\AA}}_{L^\infty}$ is controlled at \eqref{eq:424} in \S \ref{sec:at-a}.
\item $\nm{\partial_\aa \f{1}{Z_{,\aa}}}_{L^2}$ is controlled at \eqref{eq:425} in \S \ref{sec:44}.
\item $\nm{\f{h_{t\a}}{h_\a}}_{L^\infty}$ is controlled at \eqref{eq:451}, 
 $\nm{(I+\HH) D_\aa Z_t}_{L^\infty}$ is controlled at \eqref{eq:170} in \S \ref{sec:45c}.
\item $\nm{D_\aa \f{1}{Z_{,\aa}}}_{L^\infty}$ is controlled at \eqref{eq:1023} in \S \ref{sec:46}. The related terms $\nm{(Z_{tt} + i) \partial_\aa \f{1}{Z_{,\aa}}}_{L^\infty}$ and $\nm{\partial_\aa\f{Z_{tt} + i} {Z_{,\aa}}}_{L^\infty}$ are also estimated there.
\item $\nm{\partial_\aa \P_A \f{Z_t}{Z_{,\aa}}}_{L^\infty}$ is controlled at \eqref{eq:513} in \S \ref{sec:45c}. The related term $\nm{\P_A \paren{Z_t \partial_\aa \f{1}{Z_{,\aa}}}}_{L^\infty}$ we estimated at \eqref{eq:151}.
\item
$\nm{D_{\aa} B}_{L^2}$ is controlled at \eqref{eq:546}. 
\end{itemize}

\end{appendices}

\phantomsection
\addcontentsline{toc}{section}{References}


\Addresses

 \end{document}